\documentclass[10pt,leqno]{amsart}

\usepackage{amssymb,amsmath,amsthm}
\usepackage[latin1]{inputenc}
\usepackage[english]{babel}
\usepackage{mathrsfs}
\usepackage[margin=3cm]{geometry}

\usepackage{graphicx}
\usepackage{stackrel}
\usepackage[toc,page]{appendix}
\usepackage{float} \usepackage[colorlinks, citecolor=blue]{hyperref}
\usepackage{enumerate, enumitem}

\usepackage{tikz}
\usepackage{mathdots}
\usepackage{yhmath}
\usepackage{cancel}
\usepackage{color}
\usepackage{siunitx}
\usepackage{array}
\usepackage{multirow}
\usepackage{textcomp,gensymb}
\usepackage{tabularx}
\usepackage{extarrows}
\usepackage{booktabs}
\usetikzlibrary{fadings}
\usetikzlibrary{patterns}
\usetikzlibrary{shadows.blur}
\usetikzlibrary{shapes}

\newcommand{\R}{\mathbb R}
\newcommand{\N}{\mathbb N}
\newcommand{\Z}{\mathbb Z}

\newcommand{\al}{\alpha}

\newcommand{\ve}{\varepsilon}
\newcommand{\vt}{\vartheta}
\newcommand{\vp}{\varphi}
\newcommand{\TT}{\mathbb{S}} 

\newcommand{\Ddt}{\dfrac{D}{dt}} 

\newcommand{\sset}{\subseteq} 
\newcommand{\uguale}{\stackrel{.}{=}} 
\newcommand{\wconv}{\rightharpoonup} 

\theoremstyle{plain}
\newtheorem{thm}{Theorem}[section]

\newtheorem{lem}[thm]{Lemma}
\newtheorem{prop}[thm]{Proposition}

\theoremstyle{definition}

\newtheorem{defn}[thm]{Definition}
\newtheorem{rmk}[thm]{Remark}
\newtheorem{rem}[thm]{Remark}

\theoremstyle{remark}

\definecolor{ceruleanblue}{rgb}{0.16, 0.32, 0.75}
\definecolor{cocoabrown}{rgb}{0.82, 0.41, 0.12}
\definecolor{crimsonglory}{rgb}{0.75, 0.0, 0.2}
\definecolor{aqua}{rgb}{0, 0.51, 0.5}

\title{Chaotic phenomena for generalised $N$-centre problems}
\author{Stefano Baranzini and Gian Marco Canneori}
\address{Dipartimento di Matematica ``G. Peano''
	\newline\indent
	Universit\`a degli Studi di Torino
	\newline\indent
	Via Carlo Alberto 10, 10123 Torino, Italy\\}
\email{stefano.baranzini@unito.it}
\email{gianmarco.canneori@unito.it}

\date{\today}
\thanks{The second author is partially supported by by the 2022 INdAM-GNAMPA Project {\sl Dinamica simbolica e soluzioni periodiche per problemi singolari della Meccanica Celeste} and the 2023 INdAM-GNAMPA Project {\sl Biliardi galattici rifrattivi: dalla teoria dei Sistemi Dinamici alla Meccanica Celeste}}

\date{\today}

\begin{document}

	\begin{abstract}
		We study a class of singular dynamical systems which generalise the classical $N$-centre problem of Celestial Mechanics to the case in which the configuration space is a Riemannian surface. We investigate the existence of topological conjugation with the archetypal chaotic dynamical system, the Bernoulli shift. After providing infinitely many geometrically distinct and collision-less periodic solutions, we encode them in bi-infinite sequences of symbols. Solutions are obtained as minimisers of the Maupertuis functional in  suitable free homotopy classes of the punctured surface, without any collision regularisation. For any sufficiently large value of the energy, we prove that the generalised $N$-centre problem admits a symbolic dynamics. Moreover, when the Jacobi-Maupertuis metric curvature is negative, we construct chaotic invariant subsets.
	\end{abstract}
	
	\maketitle
	\begingroup
	\hypersetup{linkcolor=black}
	\tableofcontents
	\endgroup 
	
	\section{Introduction and main results}
	
	The Euclidean $N$-centre problem has been the object of extensive investigations, starting from the pioneering papers \cite{Kna1987,Bol1984,KleKna1992,BolNeg2001,Kna2002} which emphasised the intricate nature of this problem. In the most classical setting, the singularity set is made of $N$ \emph{heavy} poles $c_1,\ldots,c_N$ in the configuration space $\R^d$ (here $d=2,3$), with associated masses $m_1,\ldots,m_N>0$. For a test particle $x(t)\in\R^d$, the equation of motion reads
	\[
	\ddot{x}(t)=-\sum\limits_{j=1}^N\dfrac{m_j(x(t)-c_j)}{|x(t)-c_j|^{\al+2}}, 
	\]
	which can be also written as $\ddot{x}=-\nabla \mathcal{V}(x)$, where 
	\begin{equation}\label{eq:classical_ncentre}
		\mathcal{V}(x)=-\sum\limits_{j=1}^N\dfrac{m_j}{\al|x-c_j|^\al},
	\end{equation}
	with $\al\geq 1$. 	
	The main focus of this paper is the analysis of complex and chaotic behaviours for a class of $2D$ dynamical systems driven by singular homogeneous potentials. We study some \emph{generalised} $N$-centre problems, whose most relevant instances include the motion of a test particle on a Riemannian surface $(M,g)$, gravitating under the attraction of $N$ fixed heavy bodies. In analogy with the classical Keplerian gravitational laws with flat metric $g$, the particle is subjected to an attractive force depending on the Riemannian distance induced by $g$.
	
	 As a consequence of the presence of singularities, the flow associated to the motion equation for the generalized $N$-centre problems is not complete. Indeed, the point particle $x(t)$ may cross a centre $c_j$ in finite time and usually we refer to this phenomenon as a \emph{collision}. Another common feature of the generalized $N$-centre problems (and \eqref{eq:classical_ncentre}) is  how the homogeneity degree $\al$ drastically affects the orbits structure, already when $N=1$ (see \cite{Gor1977,Gor1975}).  We say that the singularities of $\mathcal{V}$ are \emph{Newtonian} when $\al  = 1$, \emph{weak force} when $\al\in(1,2)$ and \emph{strong force} when $\al\ge 2$. 
	 
	 As a further indicator of complexity of these systems, note that,  except  for the classical completely integrable cases  $N=1$ (the Kepler problem) or $N=2$ (the 2 centre problem, solved by Euler and Jacobi), the analytic integrability of \eqref{eq:classical_ncentre} is destroyed  as the number of non-linear interactions between the particle and the centres increases (see \cite{Bol1984}).  This fact, together with the relevant number of applications in Celestial Mechanics,  has fostered many different approaches to investigate the rise of chaotic behaviours. In particular, topological methods relying on global regularisation of collisions and classical perturbative approaches have brought to light chaotic invariant subsets of the phase space for the flat $N$-centre problem, as discussed at the end of this section.
	
	\subsection{Problem setting} 
	
	In this paper we consider a family of singular dynamical systems -- the \emph{generalised $N$-centre problem} -- defined on an orientable and complete Riemannian surface $(M,g)$. The metric $g$ provides a natural way of measuring the length of regular curves on $M$ and a natural distance function $d_g(p,q)$ for any $p,q\in M$. 
	Let us introduce the $N$-centre problem on the surface $M$.  Consider $ \mathcal{C}\uguale\{c_1, \ldots, c_N\} \subset M$, the set of \emph{centres}, and let $\widehat{M}\uguale M\setminus \mathcal{C}$ be the \emph{configuration surface}. We wish to define a \emph{potential energy} $V$ on $\widehat{M}$ depending on the reciprocal Riemannian distance $d_g(\cdot,c_j)$. Recalling the classical Euclidean potential defined in \eqref{eq:classical_ncentre}, a natural way to introduce one on $(\widehat{M},g)$ could be the following:
	\begin{equation}\label{eq:def_potential}
		\tilde{V}(q)= - \sum_{j=1}^N \dfrac{m_j}{\alpha_j d_g(q,c_j)^{\alpha_j}},
	\end{equation}
	where $\alpha_j \in [1, 2)$ and $m_1, \ldots, m_N\in\R^+$ stand for the masses associated to each centre. However,
	such a function $\tilde{V}$ may fail to be differentiable in $\widehat{M}$ if, for instance, $M$ is compact or the curvature of $g$ is positive somewhere in $M$. Indeed, if we fix $q\in M$, the distance function $p\mapsto d_g(p,q)$ is smooth only as long as there is a unique minimiser of $\ell$ joining $p$ and $q$. We are thus brought to consider potentials $V$ which behave as \eqref{eq:def_potential} only locally around the singularities. More formally, we will assume that $V \in \mathscr{C}^2(\widehat{M})$ and that  there exists $r>0$ such that, in every metric ball $B_r(c_j)$, the potential $V$ has the form
	\begin{equation}\label{eq:potential}
		V(q)\sim -\dfrac{m_j}{\alpha_j d_g(q,c_j)^{\alpha_j}}+W_j(q),
	\end{equation}
	where $W_j$ is a smooth function in $B_r(c_j)$ (cf \cite{BolKoz2017}). This means that, close to every centre $c_j$, the particle $q$ is under the attraction of a perturbed $-\al_j$-homogeneous potential, with $\al_j\in[1,2)$. In addition, we require the function $V$ to be bounded away from the centres
	\[
	\sup\limits_{q\in M\setminus\bigcup B_r(c_j)} \vert V(q)\vert<+\infty.
	\]
    The $N$-centre problem in $\widehat{M}$ is then the following Lagrangian system, defined on the tangent bundle
	\begin{equation}\label{eq:Newton}
		\Ddt \dot{u} = -\nabla V(u),
	\end{equation}
	where the gradient $\nabla$ and the covariant derivative $\frac{D}{dt}$ are defined by the Riemannian metric $g$. The associated Lagrangian function reads
	\[
	L(u,\dot{u})=\dfrac{1}{2}|\dot{u}|^2_g-V(u).
	\]      
	Formally, we say that $u\colon J\to\widehat{M}$ is a \emph{classical solution} of \eqref{eq:Newton} if $u(t)$ solves \eqref{eq:Newton} for any $t\in J$. Note that, from the Hamiltonian viewpoint, any classical solution $u$ of \eqref{eq:Newton} verifies the following conservation of energy law:
	\begin{equation}\label{eq:energy}
		\frac12|\dot{u}(t)|^2_g +V(u(t))=h,\quad\forall\,t\in J,
	\end{equation}
	and thus it makes sense to study \eqref{eq:Newton} in fixed energy levels. For the purposes of this paper, we will consider only energy levels $h$ above a certain threshold, namely:
	\begin{equation}\label{hyp:energy_bound}
		h>\sup\limits_M V.
	\end{equation}
	This is the natural extension to $(M,g)$ of the \emph{positive energy} $N$-centre problem on $\R^2$ with standard flat metric $g_e$ (see \cite{Kna1987,Bol1984,KleKna1992,Cas2017}).
	
	\begin{rem}
	So far, we have made no assumptions on the compactness of $M$ and the non-compact case is also an object of our study. However, in this case, some control on the metric $g$ is needed. To be precise, we will assume that, together with $M$, an embedding $\psi:M \to \mathbb{R}^3$ is given, and that the Riemannian metric $g$ can be controlled with the pull-back of the Euclidean metric through $\psi$, which we have already denoted by $g_e$. Namely, we will assume that there exist constants $\Lambda, \lambda>0$ such that:
	\begin{equation}
		\label{hyp:bounds_metric}
		\lambda g \le g_e \le \Lambda g.
	\end{equation}
	\end{rem}
	The purpose of this paper is twofold: as a first result we provide infinitely many distinct collision-less periodic orbits for \eqref{eq:Newton} with constant positive energy satisfying \eqref{hyp:energy_bound} and prescribed homotopy class. Then, we relate this result with the presence of invariant subsets of the phase space on which the first return map acts in a (possibly) chaotic way. More precisely, we construct a topological semi-conjugation (which in some cases is actually a conjugation) with  the paradigmatic chaotic dynamical system: the Bernoulli shift on bi-infinite sequences.
	
    \subsection{Infinitely many periodic orbits} 
   At first we provide some multiplicity results for periodic solutions of the generalised $N$-centre problem. We say that two periodic solutions $\gamma_1$ and $\gamma_2$ of \eqref{eq:Newton} (defined respectively on $[0,T_i]$) are \emph{geometrically distinct} if their supports do not coincide, namely $\gamma_1([0,T_1])\ne \gamma_2([0,T_2])$. In particular, non homotopic loops are geometrically distinct.
   
   Our first main result (Theorem \ref{thm:technical_thm}) states that there are infinitely many periodic trajectories belonging to suitable homotopy classes which, according to Definition \ref{def:admissible-class}, we call \emph{admissible}.    
   An application of Theorem \ref{thm:technical_thm} yields the following illustrative statement which generalises for instance \cite[Theorems 1.2 and 1.3]{Cas2017}.
   \begin{thm}\label{thm:main_theorem}
   	In every energy level $h$ satisfying \eqref{hyp:energy_bound}, there are infinitely many geometrically distinct, periodic and classical solutions to equation \eqref{eq:Newton} in each of the following cases:
   	\begin{itemize}
   		\item $(M,g)=(\R^2,g_e)$, where $g_e$ stands for the Euclidean metric and $N\ge 3$;
   		\item $(M,g)=(\TT^2,g)$ and $N\ge 5$;
   		\item $M$ has genus greater than or equal to 1 and $N\ge 1$.
   	\end{itemize}
   \end{thm}
   For the proof, we opt for a variational argument, in which solutions are obtained as critical points of a suitable functional. We introduce the so-called \emph{Maupertuis functional}, defined as
   \[
   \mathcal{M}_h(\gamma)\uguale\int_0^1|\dot{\gamma}(t)|_g^2\int_0^1\left[h-V(\gamma(t))\right]\,dt,
   \]  
    where $|\cdot|_g$ is the norm induced by $g$ on the tangent bundle (rigorous definitions are given in Section \ref{sec:variational_frame}). Non-constant critical points of $\mathcal{M}_h$ are collision-less solutions of \eqref{eq:Newton} at energy $h$ (\emph{Maupertuis principle}, see \cite{AmbCotZel1993}). Let us mention another useful characterization of critical points of $\mathcal{M}_h$: minimisers of the Maupertuis functional are re-parametrised minimising geodesics of the so-called \emph{Jacobi-Maupertuis metric}
    \[
    g_J(v,v)\uguale (h-V(x))g(v,v),\quad x\in T_xM,
    \]
    which is conformal to the ambient metric $g$ whenever $h$ satisfies \eqref{eq:energy}. 
   
    The first step in the proof of Theorem \ref{thm:main_theorem} is to minimise $\mathcal{M}_h$ over those closed $H^1$ paths which belong to suitable homotopy classes (see Section \ref{sec:variational_frame}). To exclude possible collisions, a blow-up analysis and a refinement of the classical obstacle technique for singular problems (see \cite{TerVen2007, SoaTer2012}) is developed in Section \ref{sec:obstacle}. Not every homotopy class on $\widehat{M}$ matches our purposes, especially when more than one centre is Newtonian. In Definition \ref{def:admissible-class}, we give a notion of \emph{admissible classes} which extends the one introduced in \cite{Cas2017, Yu2016} and we prove that infinitely many admissible classes, which contain collision-less periodic minimisers, exist in any of the situations listed in Theorem \ref{thm:technical_thm}.
   
    The existence proof of minimisers can be replicated for any lower semi-continuous functional on a weakly closed set. In this sense, it is reasonable to consider potentials of the form \eqref{eq:def_potential} which may correspond to continuous, but \emph{non-differentiable} functionals $\mathcal{M}_h$. In this case, the dynamical system \eqref{eq:Newton} has additional singularities on the cut locus of $d_g(\cdot,c_j)$, which have to be treated separately. This is addressed in Section \ref{sec:true_kepler}, under some additional regularity assumptions on $d_g(.,c_j)$. In Theorem \ref{thm:weak_solutions} we construct periodic $\mathscr{C}^1$ weak solutions of \eqref{eq:Newton} with prescribed energy $h$ in infinitely many homotopy classes.

    \subsection{Invariant chaotic subsets, conjugation, symbolic dynamics} In this work we will use the following definition of a chaotic dynamical system:
    \begin{defn}[Devaney \cite{Dev_book}]
   	\label{def:chaos}
   	If $(X,d)$ is a metric space, we say that a continuous map $f\colon X\to X$ is \emph{chaotic} if
   	\begin{itemize}
   		\item periodic points are dense in $X$;
   		\item $f$ is \emph{transitive};
   		\item $f$ has \emph{sensitive dependence on initial conditions}. 
    \end{itemize}
    \end{defn}
    For a continuous dynamical system, a straightforward verification of these three properties is usually highly difficult and mostly unfeasible. This is where the tool of conjugation becomes very useful. As a matter of fact, there is a prototypical dynamical system which easily verifies the above definition of chaos, the so-called \emph{Bernoulli shift}. It is a discrete dynamical system, which acts on bi-infinite sequences of symbols, chosen in a finite set. Let $S=\{s_1,\ldots,s_n\}$ be a finite set endowed with the usual discrete metric $\rho(s_k,s_j)=\delta_{kj}$, where $\delta_{kj}$ stands for the classical Kronecker delta. We define the set of bi-infinite sequences in $S$ as 
    \[
    S^\Z\uguale\{(s_k)_{k\in\Z}:\,s_k\in S\},
    \]
    and we endow $S^\Z$ with the following distance:
    \begin{equation}
   	\label{eq:def_metric_bernoulli}
   	d_1((s_k),(t_k))\uguale\sum\limits_{k\in\Z}\dfrac{\rho(s_k,t_k)}{4^{|k|}},
    \end{equation}
    so that $(S^\Z,d_1)$ is a metric space. The Bernoulli shift is then the discrete dynamical system $(S^\Z,\sigma)$, where the map $\sigma$ acts in the following way
    \begin{equation}
   	\label{eq:def_bernoulli_shift}
   	\begin{aligned}
   		\sigma\colon &\mathcal{S}^\Z\longrightarrow \mathcal{S}^\Z \\
   		&(s_k)\mapsto \sigma((s_k))\uguale (s_{k+1}),
   	\end{aligned}
    \end{equation}
    which means that the whole sequence $(s_k)$ is \emph{shifted} on the right. It is well-known that the Bernoulli shift is a chaotic map (for a proof see \cite{Dev_book,KatHas1995}), but it can also be used to prove that other dynamical systems $(X,f)$ possess invariant subsets on which the restriction of the map $f$ is chaotic. 
    \begin{defn}\label{def:semi_conj}
   	Let $X,Y$ be two metric spaces. A map $g\colon Y\to Y$ is \emph{topologically semi-conjugate} to a map $f\colon X\to X$ if there exists a continuous and surjective map $\pi\colon X\to Y$ such that $g\circ \pi = \pi\circ f$. In addition, if $\pi$ is a homeomorphism, we say that the maps $f$ and $g$ are \emph{topologically conjugated}.
    \end{defn}
   
    \begin{defn}\label{def:symb_dyn}
   	Let $S$ be a finite set, $\Sigma$ be a metric space and $\mathcal{R}\colon\Sigma\to\Sigma$ be a continuous map. We say that the dynamical system $(\Sigma,\mathcal{R})$ admits a \emph{symbolic dynamics} with set of symbols $S$ if there exists a $\mathcal{R}$-invariant subset $\Pi$ of $\Sigma$ such that the map $\mathcal{R}|_\Pi$ is semi-conjugated to the Bernoulli shift map $\sigma$. Furthermore, if the map $\mathcal{R}|_{\Pi}$ is conjugated to $\sigma$, we say that $(\Sigma,\mathcal{R})$ admits a \emph{chaotic symbolic dynamics}. 
    \end{defn}
    Our first result on this direction is the following:
    \begin{thm}\label{thm:symb_dyn}
   	The $N$-centre problem on $\R^2$ displays a symbolic dynamics on every energy level satisfying \eqref{hyp:energy_bound} for $N\ge3$.
    \end{thm} 
    The proof is given in Section \ref{sec:symbolic_dynamics} and the main idea behind this construction is to encode all information about a given homotopy class into a proper sequence of intersection numbers. This approach traces back to the seminal works of \cite{morse_recurrent,hedlund_geodesic} for geodesics flow on negatively curved surfaces. Let us remark that the symbolic dynamics we build is collision-less, i.e., all admissible sequences are realised by non collision solutions. Compared to the results given in \cite{BolKoz2017}, our construction is completely explicit and elementary. Moreover, let us recall that it is known that the presence of a semi-conjugation with a chaotic map is enough to conclude that our dynamical system possesses positive topological entropy (see \cite[Proposition 3.1.6]{KatHas1995}).
   
    Notice that Theorem \ref{thm:symb_dyn} is local in nature. Starting from a surface $(M,g)$, whenever we can prove the existence of a closed solution of \eqref{eq:Newton} which bounds at least three centres, our construction provides an invariant compact subset topologically semi-conjugated to the Bernoulli shift. This, for instance, allows us to prove an analogous statement for the sphere $\mathbb{S}^2$ with at least $5$ centres, as discussed in Section \ref{sec:symbolic_dynamics}.   
    A similar construction can be carried out on any surface of genus greater than or equal to $1$ with at least one singularity (we refer in particular to Subsection \ref{subsec:genus_greater_1}).
   
    At last, we present the most relevant application of Theorems \ref{thm:main_theorem} and \ref{thm:symb_dyn}. Under some suitable assumptions on the curvature of $g$, we prove that our dynamical system is \emph{conjugated} to the Bernoulli shift. The proof relies strongly on the uniqueness of minimisers of $\mathcal{M}_h$ in each homotopy class. This is strictly related to the sign of the scalar curvature of the Jacobi-Maupertuis metric $g_J$ (see in particular Theorem \ref{thm:conjugation_neg_curvature}). 
    \begin{thm}\label{thm:chaotic_sym_dyn}
   	The following holds:
   	\begin{itemize}
   		\item If $M=\R^2$, $N\ge 3$ and $g_e$ is the Euclidean metric, then the $N$-centre problem on $(\R^2,g_e)$ admits a chaotic symbolic dynamics on any energy level $h$ as in \eqref{hyp:energy_bound}.
   		\item Let $(M,g)$ be such that $g$ has negative curvature and assume that $N\ge3$ when $M=\R^2$; then, there exists $h^*=h^*(M,g,V)>\sup_M V$ for which the $N$-centre problem on $(M,g)$ has a chaotic symbolic dynamics on any energy level $h>h^*$.
   	\end{itemize}
    \end{thm}
    As a consequence of the topological conjugation, in all the situations identified by this result, the $N$-centre problem has a chaotic first return map acting on an invariant subset of the energy shell. It is important to notice that, when $M=\mathbb{R}^2$ is endowed with a non negatively curved metric, we can consider potentials of the form \eqref{eq:def_potential}, for there exist no conjugate points nor closed geodesics. Moreover, we notice that our arguments are flexible enough to extend also Theorem \ref{thm:chaotic_sym_dyn} to the case of a surface with genus $g\geq 1$ and $N\geq 1$ centres. 

    There are several important contributions in the literature which are strictly related to our results. For instance, in \cite{BolKoz2017,Kna1987} (and in \cite{BolNeg2001} for the spatial case) several collisions regularisation schemes have been proposed, yielding positive \emph{topological entropy} and  existence of invariant chaotic subsets. A deep study of high energy scattering phenomena for Newtonian potentials have been carried out in \cite{KleKna1992} and in \cite{Kna2002} for the tridimensional case. 
   
    A significant difference between the aforementioned results and ours is that we make use of variational methods alone, we build a symbolic dynamic using non collision orbits and describe chaotic behaviours for these \emph{singular} systems without employing any regularisation argument. Almost all the arguments presented here naturally extend to more singular situations in which regularization is not possible at all, for instance due to the presence of anisotropy in the asymptotic expansion of the potential (see for instance \cite{BarCanTer2021}). Indeed, one of the  advantages of our arguments is that we work directly with the homotopy classes of the punctured surfaces, without using any compactification procedure. Moreover, our results considerably extend the ones presented in \cite{Cas2017}, on which a different and preliminary construction of a semi-conjugation is illustrated in the case $(M,g)=(\R^2,g_e)$. 
    
    It is also interesting to compare our results to the classical ones for complete compact manifolds. The study of the geodesic flow on higher genus surfaces has attracted a lot of interest in the past century, starting from the seminal works of Morse, Hedlund, Hopf and many others. The main reason behind these investigations is a subtle interplay between topological and dynamical properties of the geodesic flow. The configuration spaces we consider display a similar topological complexity, but lack the regularity and compactness properties of closed manifolds. Nevertheless, many of the phenomena arising in the classical case persist. For instance, it is known that the geodesic of flow on any higher genus surface has positive topological entropy (\cite{katok_1982}). Moreover, the geodesic flow on negatively curved surfaces is the prototypical example of Anosov flows (see \cite{Anosov1969GeodesicFO,Kling_Anosov_type}): it possesses periodic dense orbits, which are hyperbolic, and it is \emph{ergodic}.  
   
    \subsection{Outline of the paper} The structure of the paper is the following. In Section \ref{sec:top_framework} we recall some basic facts about loops on surfaces and their intersections properties. In Section \ref{sec:variational_frame} we present the variational framework and prove some elementary properties of Maupertuis minimisers. Section \ref{sec:obstacle} contains all the basic information about the \emph{obstacle} technique and its adaptation to the non Euclidean setting. 
   Section \ref{sec:proof_thm} contains the proof of Theorem \ref{thm:main_theorem} and its more general version (Theorem \ref{thm:technical_thm}). Section \ref{sec:symbolic_dynamics}  is devoted to the construction of symbolic dynamics and to the proofs of semi-conjugation and conjugation (see Theorem  \ref{thm:technical_symbolic_dyn} and \ref{thm:conjugation_neg_curvature}). Finally, in Section \ref{sec:true_kepler} we consider non regular potentials of the form \eqref{eq:def_potential}.
   
	\section{Topological framework and admissible loops}
	\label{sec:top_framework}
	
	In this section we recall some basic definitions and results about the geometrical self-intersection index, in order to fix our notation and to introduce a notion of {\em admissible classes of loops}.
	
	A {\em curve} or {\em path} on the configuration surface $\widehat{M}$ is a continuous map $\gamma\colon[0,1]\to\widehat M$. If $\gamma(0)=\gamma(1)$, we refer to $\gamma$ as a {\em closed curve} or a {\em loop}, which can also be seen as a continuous map from $\TT^1$ into $\widehat M$, if $\TT^1$ denotes the oriented unit circle. 
	
	Roughly speaking, two loops are \emph{equivalent} if they can be deformed continuously one into another, as it is made rigorous in the following:
	\begin{defn}
		Given two loops $\gamma, \tau: \TT^1 \to \widehat M$, we say that $\gamma$ and $\tau$ are {\em homotopic}, and we write $\gamma \sim \tau$, if there is a continuous map  $h\colon[0,1]\times \TT^1 \to \widehat M$ such that
		\begin{itemize}
		\item $h(0,t)= \gamma(t)$; \item $h(1,t)= \tau(t)$,
		\end{itemize}
		for all $t\in \TT^1$.
		
		We say that a loop is {\em contractible} if it is {\em homotopically trivial}, i.e., if it is homotopic to a constant loop. If $\gamma, \tau\colon[0,1] \to \widehat M $ are continuous paths with the same endpoints, we say that $\gamma$ and $\tau$ are \emph{homotopic} (or \emph{homotopic rel boundary}) if there exists $h\colon[0,1]\times[0,1]\to \widehat M$ such that 
		\begin{itemize}
		\item $\gamma(0) = h(s,0) = \tau(0)$;
		\item $\gamma(1) = h(s,1) = \tau(1)$;
		\item $h(0,t) = \gamma(t)$; \item $h(1,t) = \tau(t)$,
     	\end{itemize}
        for any $t\in[0,1]$.
	\end{defn}
	
	\begin{defn}
		Let $q \in \widehat M$ and denote by $\pi_1(\widehat M,q)$ the set of homotopy equivalence classes of loops with base point $q$. We call $\pi_1(\widehat M,q)$  the \emph{fundamental group} of $\widehat M$ based at $q$.	
	\end{defn}
	
	\begin{defn}[Concatenation of paths]
		Given two paths $\gamma\colon[a,b]\to M$ and $\tau\colon[b,c]\to M$ such that $\gamma(b)=\tau(b)$, we define their \emph{concatenation} as the path $\gamma\#\tau\colon[a,c]\to M$.
	\end{defn}

	\subsection{Intersection indices and minimal position loops}
	
	In this paragraph we briefly recall the notion of geometric self-intersection number for homotopy classes as well as the definition of \emph{taut loop} that is required in the next sections. Our basic reference is \cite{HasSco1985} and references therein.

	Given two closed curves in $\widehat M$, there are two natural ways to count the number of intersection points between them: {\em signed and unsigned.} 
	Unless otherwise specified, we will mainly refer to {\em unsigned intersections}. 
	\begin{defn}\label{def:number-intersections}
		For two loops $\gamma$ and $\tau$ in $\widehat M$, their  {\em number of intersections} is 
		\[
		|\gamma \cap \tau|\uguale \{(t,t'): t, t' \in \TT^1 \textrm{ and } \gamma(t)=\tau(t')\}| \in \N\cup\{+\infty\}.
		\]
		The {\em number of self-intersections} of $\tau$ is given by 
		\[
		|\tau|\uguale\dfrac12\left\vert\{(t, t'): t \neq t' \in \TT^1 \textrm{ and } \tau(t)=\tau(t')\}\right\vert \in \N\cup\{+\infty\}.
		\]
	\end{defn}
	\begin{rem}
		Notice that the factor $1/2$ appearing in Definition~\ref{def:number-intersections} comes from the identification of $(t,t')$ with $(t',t)$. Moreover, the number of intersections or self-intersections is  always finite if the curves are in {\em general position}, i.e., their intersections are always transversal. 
	\end{rem} 
	
	\begin{defn}\label{em:geometric-intersection-number}
		The {\em geometric intersection number} between two homotopy classes $[\tau]$ and $[\gamma]$ of simple closed curves in a surface $\widehat M$ is defined to be the minimal number of intersection points between a representative curve in the class $[\tau]$ and a representative curve in the class $[\gamma]$: 
		\[
		i([\tau],[\gamma])\uguale \min\{|\tau' \cap \gamma'|: \tau' \in [\tau], \ \gamma' \in [\gamma]\}.
		\]	
		The {\em geometric self-intersection number} $i([\tau])$ is defined to be the minimal number of self-intersection points over all closed curves  in the class $[\tau]$:
		\[
		i([\tau])\uguale \min\{|\tau'|: \tau' \in [\tau]\}.
		\]	
		Given two loops  $\tau$ and $\gamma$, we say that they are in {\em minimal position} if they realise the intersection number of their homotopy class. Similarly, a loop $\tau$ is in {\em minimal position} (or \emph{taut}) if it realises the self-intersection number of its homotopy class. 
	\end{defn}
    
	If a loop $\tau$ is not in minimal position it is usually said that it \emph{exceeds the number of self-intersections}. It is straightforward to check that the minimum in the definition of intersection number above is always achieved by curves that intersect transversally. 
	
	However, unlike what happens for the \emph{signed} intersection number, $i([\tau])$ cannot be computed directly using any representative in general position. Thanks to some results from \cite{HasSco1985}, which we will recall below, it is possible to compute $i([\tau])$ starting from any representative in general position and using just a finite set of moves. To state the result we need, we recall the following definitions (see Figure \ref{fig:examples_1gon_2gon} for some examples).
	\begin{defn}\label{def:monogon-bigon}
		Let $\gamma: \TT^1 \to \widehat M$ be a loop. 
		\begin{itemize}
			\item We say that $\gamma$ has a  \emph{singular  1-gon} if there exists a sub-arc $[a,b] \subset \TT^1$  such that $\gamma(a) = \gamma(b)$ and  $\gamma \vert_{[a,b]}$ is contractible.
			\item We say that $\gamma$ has a \emph{singular 2-gon} if there exist two disjoint sub-arcs $[a,b],[c,d] \subset \TT^1$  such that $\gamma(a) = \gamma(c)$ , $\gamma(b ) = \gamma(d)$ and $\gamma \vert_{[a,b]\cup [c,d]}$ is contractible.
		\end{itemize}
	\end{defn}
	
	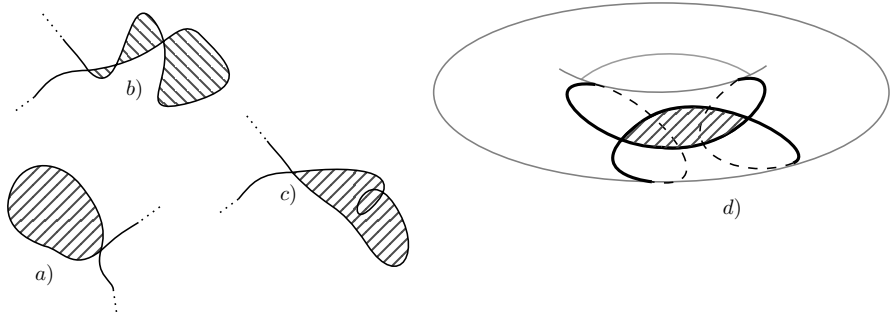
\begin{figure}[t]
	\resizebox{.82\textwidth}{!}{

\tikzset{
	pattern size/.store in=\mcSize, 
	pattern size = 5pt,
	pattern thickness/.store in=\mcThickness, 
	pattern thickness = 0.3pt,
	pattern radius/.store in=\mcRadius, 
	pattern radius = 1pt}
\makeatletter
\pgfutil@ifundefined{pgf@pattern@name@_8irt2onid}{
	\pgfdeclarepatternformonly[\mcThickness,\mcSize]{_8irt2onid}
	{\pgfqpoint{0pt}{0pt}}
	{\pgfpoint{\mcSize+\mcThickness}{\mcSize+\mcThickness}}
	{\pgfpoint{\mcSize}{\mcSize}}
	{
		\pgfsetcolor{\tikz@pattern@color}
		\pgfsetlinewidth{\mcThickness}
		\pgfpathmoveto{\pgfqpoint{0pt}{0pt}}
		\pgfpathlineto{\pgfpoint{\mcSize+\mcThickness}{\mcSize+\mcThickness}}
		\pgfusepath{stroke}
}}
\makeatother


\tikzset{
	pattern size/.store in=\mcSize, 
	pattern size = 5pt,
	pattern thickness/.store in=\mcThickness, 
	pattern thickness = 0.3pt,
	pattern radius/.store in=\mcRadius, 
	pattern radius = 1pt}
\makeatletter
\pgfutil@ifundefined{pgf@pattern@name@_jl6pmgtow}{
	\pgfdeclarepatternformonly[\mcThickness,\mcSize]{_jl6pmgtow}
	{\pgfqpoint{0pt}{0pt}}
	{\pgfpoint{\mcSize+\mcThickness}{\mcSize+\mcThickness}}
	{\pgfpoint{\mcSize}{\mcSize}}
	{
		\pgfsetcolor{\tikz@pattern@color}
		\pgfsetlinewidth{\mcThickness}
		\pgfpathmoveto{\pgfqpoint{0pt}{0pt}}
		\pgfpathlineto{\pgfpoint{\mcSize+\mcThickness}{\mcSize+\mcThickness}}
		\pgfusepath{stroke}
}}
\makeatother


\tikzset{
	pattern size/.store in=\mcSize, 
	pattern size = 5pt,
	pattern thickness/.store in=\mcThickness, 
	pattern thickness = 0.3pt,
	pattern radius/.store in=\mcRadius, 
	pattern radius = 1pt}
\makeatletter
\pgfutil@ifundefined{pgf@pattern@name@_ymez1otdo}{
	\pgfdeclarepatternformonly[\mcThickness,\mcSize]{_ymez1otdo}
	{\pgfqpoint{0pt}{-\mcThickness}}
	{\pgfpoint{\mcSize}{\mcSize}}
	{\pgfpoint{\mcSize}{\mcSize}}
	{
		\pgfsetcolor{\tikz@pattern@color}
		\pgfsetlinewidth{\mcThickness}
		\pgfpathmoveto{\pgfqpoint{0pt}{\mcSize}}
		\pgfpathlineto{\pgfpoint{\mcSize+\mcThickness}{-\mcThickness}}
		\pgfusepath{stroke}
}}
\makeatother


\tikzset{
	pattern size/.store in=\mcSize, 
	pattern size = 5pt,
	pattern thickness/.store in=\mcThickness, 
	pattern thickness = 0.3pt,
	pattern radius/.store in=\mcRadius, 
	pattern radius = 1pt}
\makeatletter
\pgfutil@ifundefined{pgf@pattern@name@_3f979zaxb}{
	\pgfdeclarepatternformonly[\mcThickness,\mcSize]{_3f979zaxb}
	{\pgfqpoint{0pt}{0pt}}
	{\pgfpoint{\mcSize+\mcThickness}{\mcSize+\mcThickness}}
	{\pgfpoint{\mcSize}{\mcSize}}
	{
		\pgfsetcolor{\tikz@pattern@color}
		\pgfsetlinewidth{\mcThickness}
		\pgfpathmoveto{\pgfqpoint{0pt}{0pt}}
		\pgfpathlineto{\pgfpoint{\mcSize+\mcThickness}{\mcSize+\mcThickness}}
		\pgfusepath{stroke}
}}
\makeatother
\tikzset{every picture/.style={line width=0.75pt}} 

\begin{tikzpicture}[x=0.75pt,y=0.75pt,yscale=-1,xscale=1]
	
	\draw  [dash pattern={on 0.84pt off 2.51pt}]  (210,144) -- (196.33,154) ;
	\draw  [dash pattern={on 0.84pt off 2.51pt}]  (215.33,91) -- (228.33,107) ;
	\draw  [color={rgb, 255:red, 128; green, 128; blue, 128 }  ,draw opacity=1 ] (334.33,76) .. controls (334.33,44.52) and (399.25,19) .. (479.33,19) .. controls (559.41,19) and (624.33,44.52) .. (624.33,76) .. controls (624.33,107.48) and (559.41,133) .. (479.33,133) .. controls (399.25,133) and (334.33,107.48) .. (334.33,76) -- cycle ;
	\draw [color={rgb, 255:red, 155; green, 155; blue, 155 }  ,draw opacity=1 ]   (428,68) .. controls (452.33,47) and (508.67,46) .. (536.33,65) ;
	\draw [color={rgb, 255:red, 128; green, 128; blue, 128 }  ,draw opacity=1 ]   (414.33,61) .. controls (452.33,84) and (519.67,78) .. (546.33,59) ;
	\draw [line width=1.5]    (479.33,111) .. controls (532.33,117) and (568.33,59) .. (528,68) ;
	\draw  [dash pattern={on 4.5pt off 4.5pt}]  (528,68) .. controls (470.33,111) and (530.33,134) .. (563.33,121) ;
	\draw [line width=1.5]    (466.33,95) .. controls (506.33,65) and (584.33,113) .. (563.33,121) ;
	\draw [line width=1.5]    (473.33,133) .. controls (464.33,132) and (426.33,125) .. (466.33,95) ;
	\draw [line width=1.5]    (437.33,71) .. controls (394.33,66) and (436.33,111) .. (479.33,111) ;
	\draw  [dash pattern={on 4.5pt off 4.5pt}]  (437.33,71) .. controls (474.33,79) and (525.33,141) .. (473.33,133) ;
	\draw  [pattern=_8irt2onid,pattern size=4.5pt,pattern thickness=0.75pt,pattern radius=0pt, pattern color={rgb, 255:red, 74; green, 74; blue, 74}] (466.33,95) .. controls (474.33,91) and (494.33,76) .. (534.33,93) .. controls (524.33,103) and (514.33,107) .. (500.33,110) .. controls (495.67,111.55) and (489.95,112.05) .. (484.15,111.89) .. controls (471.3,111.56) and (458.09,108.07) .. (455.33,106) .. controls (448.33,112) and (458.33,99) .. (466.33,95) -- cycle ;
	\draw    (228.33,107) .. controls (232.33,112.53) and (237.33,116.53) .. (244.33,127) ;
	\draw  [pattern=_jl6pmgtow,pattern size=4.5pt,pattern thickness=0.75pt,pattern radius=0pt, pattern color={rgb, 255:red, 74; green, 74; blue, 74}] (244.33,127) .. controls (294.33,121) and (315.33,129) .. (295.33,149) .. controls (284.46,159.87) and (283.64,149.62) .. (288.36,142.82) .. controls (292.33,137.12) and (300.21,133.84) .. (309.33,147.53) .. controls (329.33,177.53) and (312.33,203.53) .. (292.33,173.53) .. controls (272.33,143.53) and (268.67,153) .. (244.33,127) -- cycle ;
	\draw    (210,144) .. controls (218.33,134) and (223.33,127) .. (244.33,127) ;
	\draw  [dash pattern={on 0.84pt off 2.51pt}]  (80,79) -- (66.33,89) ;
	\draw  [dash pattern={on 0.84pt off 2.51pt}]  (85.33,26) -- (98.33,42) ;
	\draw    (98.33,42) .. controls (102.33,47.53) and (107.33,51.53) .. (114.33,62) ;
	\draw  [pattern=_ymez1otdo,pattern size=4.5pt,pattern thickness=0.75pt,pattern radius=0pt, pattern color={rgb, 255:red, 74; green, 74; blue, 74}] (114.33,62) .. controls (164.33,56) and (168.33,27) .. (184.33,38) .. controls (203.33,57) and (219.33,71) .. (179.33,82.53) .. controls (137.91,93.58) and (174.45,64.28) .. (159.29,35.3) .. controls (158.71,34.2) and (158.06,33.1) .. (157.33,32) .. controls (137.33,2) and (138.67,88) .. (114.33,62) -- cycle ;
	\draw    (80,79) .. controls (88.33,69) and (93.33,62) .. (114.33,62) ;
	\draw  [pattern=_3f979zaxb,pattern size=4.5pt,pattern thickness=0.75pt,pattern radius=0pt, pattern color={rgb, 255:red, 74; green, 74; blue, 74}] (65,139) .. controls (71.41,120.67) and (88.99,118.96) .. (102.18,128.23) .. controls (104.77,130.06) and (107.2,132.3) .. (109.33,134.93) .. controls (122.33,150.93) and (126.33,162.93) .. (122.33,176.93) .. controls (104.67,190.67) and (97.6,176.48) .. (86.94,174.06) .. controls (86.74,174.01) and (86.54,173.97) .. (86.33,173.93) .. controls (75.33,169.93) and (57.33,160.93) .. (65,139) -- cycle ;
	\draw    (122.33,176.93) .. controls (133.33,166.93) and (133.33,166.93) .. (146.33,158.93) ;
	\draw    (122.33,176.93) .. controls (118.33,192.93) and (130.33,195.93) .. (130.33,201.93) ;
	\draw  [dash pattern={on 0.84pt off 2.51pt}]  (160,148.93) -- (146.33,158.93) ;
	\draw  [dash pattern={on 0.84pt off 2.51pt}]  (132.33,215.93) -- (130.33,201.93) ;
	
	\draw (79,184) node [anchor=north west][inner sep=0.75pt]   [align=left] {$\displaystyle a)$};
	\draw (137,66) node [anchor=north west][inner sep=0.75pt]   [align=left] {$\displaystyle b)$};
	\draw (235,133) node [anchor=north west][inner sep=0.75pt]   [align=left] {$\displaystyle c)$};
	\draw (517,142) node [anchor=north west][inner sep=0.75pt]   [align=left] {$\displaystyle d)$};

\end{tikzpicture}}
	\caption{The striped areas of this picture represent the regions bounded by some examples of singular 1-gons and 2-gons. In picture $a)$ we have an innermost singular 1-gon, while the ones showed in $b)$ and $c)$ are not innermost. In $d)$ we can see a singular 2-gon in the case $M=\mathbb{T}^2$.}\label{fig:examples_1gon_2gon}
	\end{figure}
    The next result contains some necessary conditions for a loop to be taut in its homotopy class.
	\begin{thm}{\bf (\cite[Theorem 4.2]{HasSco1985})}
		\label{thm:hass_scott} Let $\gamma: \TT^1 \to \widehat M$ be a general position loop. If $\gamma$ is not taut, then $\gamma$ has a singular 1-gon or a singular 2-gon. 
	\end{thm}
	
	\begin{defn}
		A loop $\gamma$ is called \emph{simple} if it has no self-intersections. 
		A sub-loop of $\gamma$ is said to be {\em innermost} if it is simple. 
		
		In a similar fashion, a singular 1-gon (resp. 2-gon) is {\em innermost} if, regarded as a loop, it does not contain a singular 1-gon or 2-gon. 	 
	\end{defn}
	
	By the well-known \emph{Jordan curve theorem} in $\R^2$, any simple curve $\gamma$ divides the plane in exactly two connected components: one is bounded and contractible, while the other is unbounded and not contractible. If $M$ is an orientable surface, this is no longer true in general. However, if a simple loop $\gamma$ is  contractible and $M \ne \TT^2$, $\gamma$ divides $M$ in exactly two components, one of which is contractible and the other is not. If $M = \TT^2$, $\gamma$ divides it into two disks. 
	
	These last remarks are crucial for our purposes, since we need to refine the possible choices for a homotopy class to work in. Indeed, facing the problem of avoiding collisions, we need to require some \emph{nice behaviours} inside the homotopy class. A rigorous notion of admissible class of loops is given in the next definition and concludes this section. Some remarkable examples are depicted in Figures \ref{fig:examples_admissible_class} and \ref{fig:example_non_admissible_sphere}.
	
	\begin{defn}\label{def:admissible-class}
		Given a non trivial homotopy class
		$[\tau] \in \pi_1(\widehat M)$, let $\gamma$ be a taut representative in $[\tau]$. We say that $[\tau]$ is {\em admissible} if $\gamma$ satisfies one of the following:
		\begin{itemize}
			\item if $M \sim \TT^2$, then any innermost and contractible (in $M$) sub-loop $\tilde \gamma$ of $\gamma$ contains at least two centres in \emph{both} the bounded components of $M\setminus \tilde \gamma$;
			\item if $M \not \sim \TT^2$, then any innermost and contractible (in $M$) sub-loop $\tilde \gamma$ of $\gamma$ contains at least two centres in the contractible component of $M\setminus \tilde\gamma$.
		\end{itemize} 
		Moreover, if $M$ is non compact, we require that there exists a compact subset $K\sset M$ such that any representative $\gamma$ of $[\tau]$ satisfies $\gamma \cap K \ne \emptyset.$
	\end{defn}

    \begin{figure}[t]
    	\centering
    	\resizebox{.8\linewidth}{!}{\input{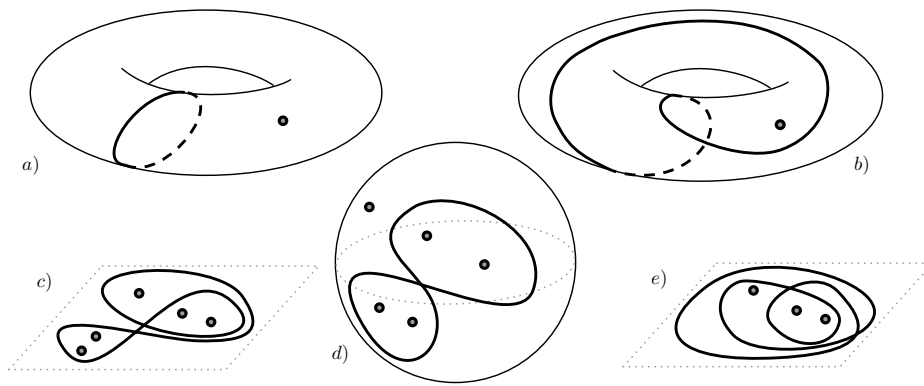}}
    	\caption{In this picture we see some examples of admissible homotopy classes of loops on $M$. Pictures $a)$ and $b)$ show that when $M=\mathbb{T}^2$ it is enough to have just one centre to define an admissible class. In $c)$ and $e)$ we have two examples of admissible classes when $M=\R^2$, while in picture $d)$ we see that at least 4 centres are required to put this definition in the context of a two-dimensional sphere.}\label{fig:examples_admissible_class}	
    \end{figure}
	
	\begin{figure}[t]
	\centering
	\resizebox{.7\linewidth}{!}{\tikzset{every picture/.style={line width=0.75pt}} 

\begin{tikzpicture}[x=0.75pt,y=0.75pt,yscale=-1,xscale=1]
	
	\draw   (97,104.5) .. controls (97,58.38) and (134.38,21) .. (180.5,21) .. controls (226.62,21) and (264,58.38) .. (264,104.5) .. controls (264,150.62) and (226.62,188) .. (180.5,188) .. controls (134.38,188) and (97,150.62) .. (97,104.5) -- cycle ;
	\draw  [color={rgb, 255:red, 155; green, 155; blue, 155 }  ,draw opacity=1 ][dash pattern={on 0.84pt off 2.51pt}] (137.64,81.01) .. controls (176.81,71.7) and (227.75,74.67) .. (251.42,87.64) .. controls (275.09,100.62) and (262.53,118.68) .. (223.36,127.99) .. controls (184.19,137.3) and (133.25,134.33) .. (109.58,121.36) .. controls (85.91,108.38) and (98.47,90.32) .. (137.64,81.01) -- cycle ;
	\draw  [fill={rgb, 255:red, 0; green, 0; blue, 0 }  ,fill opacity=1 ] (149,86.5) .. controls (149,84.01) and (151.01,82) .. (153.5,82) .. controls (155.99,82) and (158,84.01) .. (158,86.5) .. controls (158,88.99) and (155.99,91) .. (153.5,91) .. controls (151.01,91) and (149,88.99) .. (149,86.5) -- cycle ;
	\draw  [fill={rgb, 255:red, 155; green, 155; blue, 155 }  ,fill opacity=1 ] (150.67,86.5) .. controls (150.67,84.94) and (151.94,83.67) .. (153.5,83.67) .. controls (155.06,83.67) and (156.33,84.94) .. (156.33,86.5) .. controls (156.33,88.06) and (155.06,89.33) .. (153.5,89.33) .. controls (151.94,89.33) and (150.67,88.06) .. (150.67,86.5) -- cycle ;
	\draw  [fill={rgb, 255:red, 0; green, 0; blue, 0 }  ,fill opacity=1 ] (199,109.5) .. controls (199,107.01) and (201.01,105) .. (203.5,105) .. controls (205.99,105) and (208,107.01) .. (208,109.5) .. controls (208,111.99) and (205.99,114) .. (203.5,114) .. controls (201.01,114) and (199,111.99) .. (199,109.5) -- cycle ;
	\draw  [fill={rgb, 255:red, 155; green, 155; blue, 155 }  ,fill opacity=1 ] (200.67,109.5) .. controls (200.67,107.94) and (201.94,106.67) .. (203.5,106.67) .. controls (205.06,106.67) and (206.33,107.94) .. (206.33,109.5) .. controls (206.33,111.06) and (205.06,112.33) .. (203.5,112.33) .. controls (201.94,112.33) and (200.67,111.06) .. (200.67,109.5) -- cycle ;
	\draw  [fill={rgb, 255:red, 0; green, 0; blue, 0 }  ,fill opacity=1 ] (133,145.5) .. controls (133,143.01) and (135.01,141) .. (137.5,141) .. controls (139.99,141) and (142,143.01) .. (142,145.5) .. controls (142,147.99) and (139.99,150) .. (137.5,150) .. controls (135.01,150) and (133,147.99) .. (133,145.5) -- cycle ;
	\draw  [fill={rgb, 255:red, 155; green, 155; blue, 155 }  ,fill opacity=1 ] (134.67,145.5) .. controls (134.67,143.94) and (135.94,142.67) .. (137.5,142.67) .. controls (139.06,142.67) and (140.33,143.94) .. (140.33,145.5) .. controls (140.33,147.06) and (139.06,148.33) .. (137.5,148.33) .. controls (135.94,148.33) and (134.67,147.06) .. (134.67,145.5) -- cycle ;
	\draw [line width=1.5]    (101.33,78.37) .. controls (78.33,112.28) and (220.33,158.37) .. (260.33,128.37) ;
	\draw [line width=1.5]  [dash pattern={on 5.63pt off 4.5pt}]  (260.33,128.37) .. controls (271.33,87.37) and (146.33,33.28) .. (114.33,54.28) ;
	\draw [line width=1.5]    (114.33,54.28) .. controls (112.33,57.28) and (102.33,73.37) .. (145.33,107.37) .. controls (188.33,141.37) and (221.33,181.7) .. (246.33,156.7) ;
	\draw [line width=1.5]  [dash pattern={on 5.63pt off 4.5pt}]  (246.33,156.7) .. controls (254.77,140.27) and (239.14,118.62) .. (213.81,101.57) .. controls (182.12,80.23) and (135.24,66.09) .. (101.33,78.37) ;
	\draw   (345,106.5) .. controls (345,60.38) and (382.38,23) .. (428.5,23) .. controls (474.62,23) and (512,60.38) .. (512,106.5) .. controls (512,152.62) and (474.62,190) .. (428.5,190) .. controls (382.38,190) and (345,152.62) .. (345,106.5) -- cycle ;
	\draw  [color={rgb, 255:red, 155; green, 155; blue, 155 }  ,draw opacity=1 ][dash pattern={on 0.84pt off 2.51pt}] (385.64,83.01) .. controls (424.81,73.7) and (475.75,76.67) .. (499.42,89.64) .. controls (523.09,102.62) and (510.53,120.68) .. (471.36,129.99) .. controls (432.19,139.3) and (381.25,136.33) .. (357.58,123.36) .. controls (333.91,110.38) and (346.47,92.32) .. (385.64,83.01) -- cycle ;
	\draw  [fill={rgb, 255:red, 0; green, 0; blue, 0 }  ,fill opacity=1 ] (397,88.5) .. controls (397,86.01) and (399.01,84) .. (401.5,84) .. controls (403.99,84) and (406,86.01) .. (406,88.5) .. controls (406,90.99) and (403.99,93) .. (401.5,93) .. controls (399.01,93) and (397,90.99) .. (397,88.5) -- cycle ;
	\draw  [fill={rgb, 255:red, 155; green, 155; blue, 155 }  ,fill opacity=1 ] (398.67,88.5) .. controls (398.67,86.94) and (399.94,85.67) .. (401.5,85.67) .. controls (403.06,85.67) and (404.33,86.94) .. (404.33,88.5) .. controls (404.33,90.06) and (403.06,91.33) .. (401.5,91.33) .. controls (399.94,91.33) and (398.67,90.06) .. (398.67,88.5) -- cycle ;
	\draw  [fill={rgb, 255:red, 0; green, 0; blue, 0 }  ,fill opacity=1 ] (447,111.5) .. controls (447,109.01) and (449.01,107) .. (451.5,107) .. controls (453.99,107) and (456,109.01) .. (456,111.5) .. controls (456,113.99) and (453.99,116) .. (451.5,116) .. controls (449.01,116) and (447,113.99) .. (447,111.5) -- cycle ;
	\draw  [fill={rgb, 255:red, 155; green, 155; blue, 155 }  ,fill opacity=1 ] (448.67,111.5) .. controls (448.67,109.94) and (449.94,108.67) .. (451.5,108.67) .. controls (453.06,108.67) and (454.33,109.94) .. (454.33,111.5) .. controls (454.33,113.06) and (453.06,114.33) .. (451.5,114.33) .. controls (449.94,114.33) and (448.67,113.06) .. (448.67,111.5) -- cycle ;
	\draw  [fill={rgb, 255:red, 0; green, 0; blue, 0 }  ,fill opacity=1 ] (381,147.5) .. controls (381,145.01) and (383.01,143) .. (385.5,143) .. controls (387.99,143) and (390,145.01) .. (390,147.5) .. controls (390,149.99) and (387.99,152) .. (385.5,152) .. controls (383.01,152) and (381,149.99) .. (381,147.5) -- cycle ;
	\draw  [fill={rgb, 255:red, 155; green, 155; blue, 155 }  ,fill opacity=1 ] (382.67,147.5) .. controls (382.67,145.94) and (383.94,144.67) .. (385.5,144.67) .. controls (387.06,144.67) and (388.33,145.94) .. (388.33,147.5) .. controls (388.33,149.06) and (387.06,150.33) .. (385.5,150.33) .. controls (383.94,150.33) and (382.67,149.06) .. (382.67,147.5) -- cycle ;
	\draw  [line width=1.5]  (401.33,67) .. controls (435.33,40) and (496.33,87) .. (482.33,126) .. controls (468.33,165) and (339.33,111) .. (359.33,141) .. controls (379.33,171) and (402.33,164) .. (395.33,128) .. controls (388.33,97) and (381.33,87) .. (401.33,67) -- cycle ;
	
	\draw (85,170.4) node [anchor=north west][inner sep=0.75pt]    {$a)$};
	\draw (334,173.4) node [anchor=north west][inner sep=0.75pt]    {$b)$};

\end{tikzpicture}}
	\caption{An example of non admissible class in the case $M=\TT^2$. It is not difficult to check that the loops in $a)$ and $b)$ are in the same homotopy class. The one in $b)$ has a sub-loop which encloses a unique centre.}\label{fig:example_non_admissible_sphere}
	\end{figure}
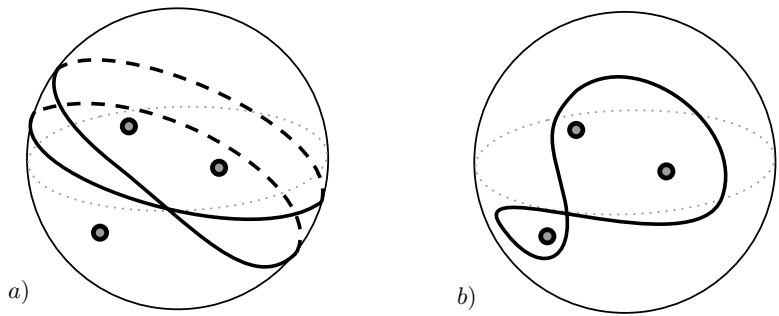
	
	In the next sections, we are going to stress how is it possible to find periodic solutions of the $N$-centres problem on $M$. Then, it will be clear that the admissibility notion introduced in the previous definition is translated into a topological constraint which we impose on such trajectories to avoid collisions with the centres. 
	
	\begin{rem}
		The requirement that $\gamma \cap K \ne \emptyset$ for all $\gamma \in [\tau]$ will be needed in Section \ref{sec:variational_frame} to prove the coercivity of the Maupertuis functional. Intuitively, this means that we are excluding those curves that can be continuously deformed into one of the ends of the manifold $M$. It is worth stressing that this is \emph{not} a purely topological issue. Indeed, the condition of Definition \ref{def:admissible-class} ensures the coercivity of the Maupertuis functional in many situations, but it is not necessary in general (consider for instance the plane $\mathbb{R}^2$ and a cylinder $C$). An example of this phenomenon is given in Figure \ref{fig:remark_coercivity}.
    \end{rem}
    
	\begin{figure}[t]
		\centering
		\resizebox{.8\textwidth}{!}{\tikzset{every picture/.style={line width=0.75pt}} 

\begin{tikzpicture}[x=0.75pt,y=0.75pt,yscale=-1,xscale=1]
	
	\draw  [color={rgb, 255:red, 74; green, 74; blue, 74 }  ,draw opacity=1 ] (42.84,188.41) -- (116.85,158.03) -- (229.16,187.59) -- (155.15,217.97) -- cycle ;
	\draw  [color={rgb, 255:red, 74; green, 74; blue, 74 }  ,draw opacity=1 ] (42.84,107.41) -- (116.85,77.03) -- (229.16,106.59) -- (155.15,136.97) -- cycle ;
	\draw [color={rgb, 255:red, 74; green, 74; blue, 74 }  ,draw opacity=1 ]   (139.33,133) .. controls (135.33,153) and (172.33,163.32) .. (156.33,188) ;
	\draw  [color={rgb, 255:red, 74; green, 74; blue, 74 }  ,draw opacity=1 ] (121.63,96.13) .. controls (132.41,93.85) and (146.49,97.06) .. (153.08,103.31) .. controls (159.66,109.55) and (156.26,116.46) .. (145.48,118.74) .. controls (134.69,121.02) and (120.61,117.81) .. (114.03,111.56) .. controls (107.44,105.32) and (110.85,98.41) .. (121.63,96.13) -- cycle ;
	\draw [color={rgb, 255:red, 74; green, 74; blue, 74 }  ,draw opacity=1 ]   (156.33,188) .. controls (144.33,199) and (103.33,195) .. (114.33,176) .. controls (116.33,151) and (97.33,151) .. (103,123) ;
	\draw [color={rgb, 255:red, 74; green, 74; blue, 74 }  ,draw opacity=1 ] [dash pattern={on 0.84pt off 2.51pt}]  (114.33,176) .. controls (118.33,163) and (167.33,172) .. (156.33,188) ;
	\draw [color={rgb, 255:red, 74; green, 74; blue, 74 }  ,draw opacity=1 ] [dash pattern={on 0.84pt off 2.51pt}]  (103,123) .. controls (105.33,114) and (105.33,117) .. (110.33,104) ;
	\draw [color={rgb, 255:red, 74; green, 74; blue, 74 }  ,draw opacity=1 ] [dash pattern={on 0.84pt off 2.51pt}]  (139.33,133) .. controls (141.67,124) and (151.33,125) .. (156.33,112) ;
	\draw  [fill={rgb, 255:red, 0; green, 0; blue, 0 }  ,fill opacity=1 ] (120,144.33) .. controls (120,143.14) and (120.97,142.17) .. (122.17,142.17) .. controls (123.36,142.17) and (124.33,143.14) .. (124.33,144.33) .. controls (124.33,145.53) and (123.36,146.5) .. (122.17,146.5) .. controls (120.97,146.5) and (120,145.53) .. (120,144.33) -- cycle ;
	\draw  [fill={rgb, 255:red, 155; green, 155; blue, 155 }  ,fill opacity=1 ] (120.8,144.41) .. controls (120.8,143.66) and (121.41,143.05) .. (122.17,143.05) .. controls (122.92,143.05) and (123.53,143.66) .. (123.53,144.41) .. controls (123.53,145.17) and (122.92,145.78) .. (122.17,145.78) .. controls (121.41,145.78) and (120.8,145.17) .. (120.8,144.41) -- cycle ;
	\draw  [fill={rgb, 255:red, 0; green, 0; blue, 0 }  ,fill opacity=1 ] (132,179.33) .. controls (132,178.14) and (132.97,177.17) .. (134.17,177.17) .. controls (135.36,177.17) and (136.33,178.14) .. (136.33,179.33) .. controls (136.33,180.53) and (135.36,181.5) .. (134.17,181.5) .. controls (132.97,181.5) and (132,180.53) .. (132,179.33) -- cycle ;
	\draw  [fill={rgb, 255:red, 155; green, 155; blue, 155 }  ,fill opacity=1 ] (132.8,179.41) .. controls (132.8,178.66) and (133.41,178.05) .. (134.17,178.05) .. controls (134.92,178.05) and (135.53,178.66) .. (135.53,179.41) .. controls (135.53,180.17) and (134.92,180.78) .. (134.17,180.78) .. controls (133.41,180.78) and (132.8,180.17) .. (132.8,179.41) -- cycle ;
	\draw [line width=1.5]    (113.33,166) .. controls (89.33,171.52) and (82.33,184.52) .. (101.33,192.52) .. controls (120.33,200.52) and (155.33,205.52) .. (171.33,195.52) .. controls (187.33,185.52) and (166.83,177) .. (159.33,174) ;
	\draw [line width=1.5]  [dash pattern={on 5.63pt off 4.5pt}]  (113.33,166) .. controls (139.33,163.32) and (145.33,168.32) .. (159.33,174) ;
	\draw  [color={rgb, 255:red, 74; green, 74; blue, 74 }  ,draw opacity=1 ] (432.69,101.14) -- (364.3,209.33) .. controls (362.37,212.38) and (352.56,209.64) .. (342.38,203.21) .. controls (332.2,196.78) and (325.52,189.09) .. (327.45,186.03) -- (395.84,77.85) .. controls (397.77,74.79) and (407.58,77.53) .. (417.76,83.97) .. controls (427.94,90.4) and (434.62,98.09) .. (432.69,101.14) .. controls (430.76,104.2) and (420.95,101.46) .. (410.77,95.02) .. controls (400.59,88.59) and (393.91,80.9) .. (395.84,77.85) ;
	\draw [line width=1.5]    (348.33,153) .. controls (332.33,176) and (376.33,191) .. (382.33,181) ;
	\draw [line width=1.5]  [dash pattern={on 5.63pt off 4.5pt}]  (348.33,153) .. controls (355.33,148) and (389.33,165) .. (382.33,181) ;
	\draw  [fill={rgb, 255:red, 0; green, 0; blue, 0 }  ,fill opacity=1 ] (394,134.33) .. controls (394,133.14) and (394.97,132.17) .. (396.17,132.17) .. controls (397.36,132.17) and (398.33,133.14) .. (398.33,134.33) .. controls (398.33,135.53) and (397.36,136.5) .. (396.17,136.5) .. controls (394.97,136.5) and (394,135.53) .. (394,134.33) -- cycle ;
	\draw  [fill={rgb, 255:red, 155; green, 155; blue, 155 }  ,fill opacity=1 ] (394.8,134.41) .. controls (394.8,133.66) and (395.41,133.05) .. (396.17,133.05) .. controls (396.92,133.05) and (397.53,133.66) .. (397.53,134.41) .. controls (397.53,135.17) and (396.92,135.78) .. (396.17,135.78) .. controls (395.41,135.78) and (394.8,135.17) .. (394.8,134.41) -- cycle ;
	\draw  [fill={rgb, 255:red, 0; green, 0; blue, 0 }  ,fill opacity=1 ] (368,143.33) .. controls (368,142.14) and (368.97,141.17) .. (370.17,141.17) .. controls (371.36,141.17) and (372.33,142.14) .. (372.33,143.33) .. controls (372.33,144.53) and (371.36,145.5) .. (370.17,145.5) .. controls (368.97,145.5) and (368,144.53) .. (368,143.33) -- cycle ;
	\draw  [fill={rgb, 255:red, 155; green, 155; blue, 155 }  ,fill opacity=1 ] (368.8,143.41) .. controls (368.8,142.66) and (369.41,142.05) .. (370.17,142.05) .. controls (370.92,142.05) and (371.53,142.66) .. (371.53,143.41) .. controls (371.53,144.17) and (370.92,144.78) .. (370.17,144.78) .. controls (369.41,144.78) and (368.8,144.17) .. (368.8,143.41) -- cycle ;
	
	\draw (193,138) node [anchor=north west][inner sep=0.75pt]   [align=left] {\footnotesize $ a)$};
	\draw (440,141) node [anchor=north west][inner sep=0.75pt]   [align=left] {\footnotesize $ b)$};

\end{tikzpicture}}
		\caption{In this picture, two topologically indistinguishable situations are shown. On the right, we have a standard cylinder, while on the left two planes are bridged. Both surfaces are endowed with the metric induced by $\R^3$. In $a)$, following the original argument of Gordon (see \cite{Gor1977}), one can show that the Maupertuis functional in the depicted homotopy class is coercive. This fails in $b)$, since one can have representatives with arbitrary large $L^2$ norm and bounded length.}
		\label{fig:remark_coercivity}
	\end{figure}
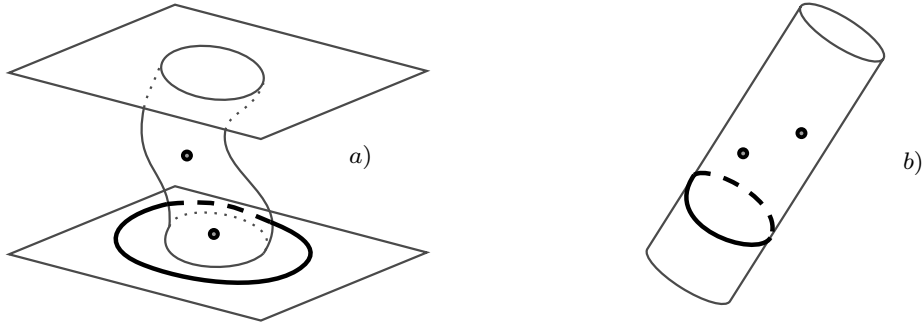

	\section{The variational framework}
	\label{sec:variational_frame}
	This paper relies on some variational techniques which succeed in the task of producing periodic solutions to our problem of study. For this reason, the present section is devoted to fix the abstract variational setting and to recall some basic properties of the {\em Maupertuis functional}. In particular, we are going to state a variational principle which guarantees that critical points of this functional actually corresponds to trajectories which solve the $N$-centre problem.

	\subsection{The Maupertuis functional on surfaces}
	
	As a first step, to introduce a topology in our loops space, we consider an embedding of our orientable surface $M$ in $\R^3$. With a slight abuse of notation, we identify $M$ with its image under this embedding. 
	
	Given $J=[t_0,t_1] \subset \R$ we define
	\begin{equation}\label{eq:def_banach_manifold}
	\mathcal{H}\uguale\left\lbrace \gamma\in H^1(J; \R^3):\,\gamma(t)\in M,\ \forall\,t\in[t_0,t_1]\right\rbrace,
	\end{equation}
	which is a Banach manifold modelled on the Sobolev space $H^1(J;\R^2)$. If $g$ is a Riemannian metric on $M$ (not necessarily the one induced by the ambient space), the tangent space of the manifold $\mathcal{H}$ can be endowed with the scalar product 
	\[
	\langle \xi, \eta \rangle_1 = \int_J \left\langle \Ddt \xi(t), \Ddt \eta(t) \right\rangle_g \,dt + \int_J \left\langle \xi(t), \eta(t)\right\rangle_g\, dt.
	\]
	
	When the surface $M$ is compact, all the metrics $g$ induce equivalent scalar products. If $M$ is non compact, however, this is no longer true. The assumption given in \eqref{hyp:bounds_metric} ensures precisely that this is the case. 
	
	We denote by $\widehat{\mathcal{H}}$ the subset of $\mathcal{H}$ made of those paths which do not intersect the singular set $\mathcal{C}$. In particular, $\widehat{\mathcal{H}}$ is an open dense submanifold of $\mathcal{H}$: its weak $H^1$-closure is exactly $\mathcal{H}$. The boundary $\partial\widehat{\mathcal{H}}$ is given by all those paths whose  preimage of the singular set $\mathcal{C}$ is not empty. The following result is well-known. 
	\begin{lem}{\bf (\cite[Proposition 2.4.1]{Kli1995})}\label{lem:evaluation}
		Let $M$ be a smooth manifold and $J=[t_0,t_1]$. The evaluation map 
		\[
		\begin{aligned}
		ev\colon&\mathcal{H}\longrightarrow M \times M \\
		&\gamma \longmapsto ev(\gamma)\uguale \left(\gamma(t_0), \gamma(t_1)\right)
		\end{aligned}
		\]
		is a submersion. As a consequence, if $N \subset M \times M$ is a submanifold of codimension $k$, $ev^{-1}(N)$ is a submanifold of $\mathcal{H}$ of codimension $k$.
	\end{lem}
	For the special case of $N=\{p,q\}$ we get the (closed) submanifold of all $H^1$ paths  starting at the point $p$ and ending at the point $q$. For $N= \Delta$, the diagonal of $M \times M$, we get the space of parametrised loops in $M$.
	
	In analogy with the introduction notation, we will often use $\widehat{M} = M \setminus \mathcal{C}$ to denote the configuration surface (without the centres). Moreover, given $N \sset M \times M$, $\widehat{N}$ denotes $N \cap \left(\widehat M \times \widehat M\right)$. We set 
	\[
	\widehat{\mathcal{H}}_N\uguale ev^{-1}(\widehat N)
	\]
	which is a $k$-codimensional submanifold of $\widehat{\mathcal{H}}$ and $\mathcal{H}_N $ its $H^1$ closure.	Looking for periodic solutions, we are interested in the case $N=\Delta$. Since we want to work with loops in a fixed homotopy class, we introduce the following definition.
	\begin{defn}
		For any $[\tau]\in\pi_1(\widehat M)$ we set
		\[
		\widehat{\mathcal{H}}_\Delta(\tau)\uguale\{\gamma \in \widehat{\mathcal{H}}_\Delta:\,\gamma\in[\tau]\}
		\]
		and $\mathcal{H}_\Delta(\tau)$ its weak $H^1$ closure. 
	\end{defn}
	
	For any fixed energy level $h$ satisfying \eqref{hyp:energy_bound}, let us consider the {\em Maupertuis functional} 
	\begin{equation}
		\label{eq:def_Maupertuis}
		\mathcal{M}_h: \mathcal{H} \to \R\cup\{+\infty\} \quad \textrm{ defined by } \quad 
		\mathcal{M}_h(\gamma)=\int_J |\dot\gamma(t)|_g^2\, dt \cdot \int_J\left[h-V(\gamma(t))\right]\, dt,
	\end{equation}	
	and, if $\gamma$ belongs to a positive level of $\mathcal{M}_h$, we set
	\begin{equation}\label{eq:omega}
		\omega^2\uguale \dfrac{\int_J \left[h-V(\gamma(t))\right]\, dt}{\frac12\int_J |\dot\gamma(t)|_g^2\,dt} \quad \textrm{ and } \quad J_\omega\uguale\left[\dfrac{t_0}{\omega}, \dfrac{t_1}{\omega}\right].
	\end{equation}
	 It is well-known that  the functional $\mathcal{M}_h$ is differentiable over $\widehat{\mathcal{H}}$. Moreover, as a direct consequence of Lemma \ref{lem:evaluation}, if $\widehat N\subset \widehat{M} \times \widehat{M}$ is a $k$-codimensional submanifold, then the restriction of the Maupertuis functional to this set is differentiable as well. Moreover, up to reparametrisation, its critical points are solutions of the fixed energy problem \eqref{eq:Newton}-\eqref{eq:energy}. More precisely, the following result holds true:
	\begin{prop}[Maupertuis principle]\label{prop:maupertuis_principle}
		Let $\gamma \in \widehat{\mathcal{H}}$ be a non-constant critical point of $\mathcal{M}_h$ at a positive level. Then $\gamma$ is a classical solution of  
		\[
		\begin{cases}
			\omega^2 \Ddt \gamma'(t)= -\nabla V(\gamma(t)) & \qquad t \in J \\[4pt]
			\dfrac12|\gamma'(t)|_g^2 +\dfrac{V(\gamma(t))}{\omega^2}= \dfrac{h}{\omega^2} & \qquad t \in J
		\end{cases},
		\]
		with boundary conditions
		\[
		\begin{cases}
			\left(\gamma(t_0),\gamma(t_1)\right) \in N & \\[4pt]
			D_v L(\gamma(t_0), \gamma'(t_0))[\xi_0]= D_v L(\gamma(t_1), \gamma'(t_1))[\xi_1] & \qquad \forall\, (\xi_0, \xi_1) \in T_{\gamma(t_0), \gamma(t_1)} N\\[4pt]
		\end{cases}
		\]
		where $D_vL(x,v)$ is the covariant derivative of the Lagrangian $L$ with respect to $v$. 
		
		In the same way, the function $\psi$ defined by $\psi(t)\uguale\gamma(\omega t)$ is a
		classical  solution of 
		\[
		\begin{cases}
			\Ddt \psi'(t)= -\nabla V(\psi(t)) & \qquad t \in J_\omega \\[4pt]
			\dfrac12|\psi'(t)|_g^2 +V(\psi(t))= h & \qquad t \in J_\omega
		\end{cases},
		\]
		with boundary conditions
		\[
		\begin{cases}
			\big(\psi(t_0/\omega), \psi(t_1/\omega) \big) \in N &\\[4pt]
			D_v L(\psi(t_0/\omega), \psi'(t_0/\omega))[\xi_{0}]= D_v L(\psi(t_1/\omega), \psi'(t_1/\omega))[\xi_{1}] & \qquad \forall\, (\xi_{0}, \xi_{1}) \in T_{\psi(t_0/\omega), \psi(t_1/\omega)} N
		\end{cases}
		\]
	\end{prop}
	\begin{proof}
		The proof follows from a slight modification of the one given in \cite[Theorem 4.1]{AmbCotZel1993}.
	\end{proof}
	\begin{rem}
		Unfortunately, the manifold $\widehat{\mathcal{H}}_N$ is not the right choice to employ direct variational methods, since it is not weakly closed. Naturally, we can overcome this problem by working in its weak $H^1$ closure $\mathcal{H}_N$ and find minimisers of $\mathcal{M}_h$. However, the price to pay is that some {\em ad hoc} arguments have to be developed in order to get rid of those minimisers which interact with the singular set $\mathcal{C}$. In particular, any loop $\gamma \in \mathcal{H}_\Delta(\tau)$ which does not collide with any of the centre is an interior point of $\mathcal{H}_{\Delta}(\tau)$. Thus any such point is a \emph{true} critical point of $\mathcal{M}_h$ and thus satisfies the conditions of Proposition \ref{prop:maupertuis_principle}.
	\end{rem}

	We recall some useful properties of the Maupertuis functional and we put them in the context of the minimising space $\mathcal{H}_\Delta(\tau)$. As a general remark, it is straightforward to check that $\mathcal{M}_h$
	\begin{enumerate}[label = \emph{\roman*)}]
		\item  is invariant with respect to  {\em affine time rescaling},
		\item  is \emph{not} additive with respect to concatenation, meaning that 
		\[
		\mathcal{M}_h(\gamma_1\# \gamma_2) \neq\mathcal{M}_h(\gamma_1) +\mathcal{M}_h(\gamma_2).
		\]
	\end{enumerate}
	
	Even so, the lack of additivity of the Maupertuis functional is partially recovered by exploiting the following two properties.
	
	\begin{lem}\label{lem:restriction}
		The Maupertuis functional $\mathcal{M}_h$ is super-additive: if $\gamma\in\mathcal{H}_\Delta(\tau)$ and $[a,b]\subset J$, then 
		\[
		\mathcal{M}_h(\gamma)\geq\mathcal{M}_h(\gamma|_{[a,b]})+\mathcal{M}_h(\gamma|_{J\setminus[a,b]}).
		\]
		Moreover, if $\gamma$ is a minimiser of $\mathcal{M}_h$ in $\mathcal{H}_\Delta(\tau)$, then the path $\gamma|_{[a,b]}$ is a minimiser in the space
		\[
		\mathcal{H}_{(\gamma(a),\gamma(b))}(\tau)=\{\eta\in \mathcal{H}:\,\eta(a)=\gamma(a),\ \eta(b)=\gamma(b),\ \eta\#\gamma|_{J\setminus[a,b]}\in\mathcal{H}_\Delta(\tau)\}.
		\] 
	\end{lem} 
	\begin{proof}
		The first assertion follows directly from the definition of $\mathcal{M}_h$. To see that, set $\gamma_1 = \gamma\vert_{J_1}$ and $\gamma_2 = \gamma\vert_{J_2}$, for $J_1=[a,b]$ and $J_2=J \setminus[a,b]$, and define 
		\[
		K_i \uguale \frac{1}{2} \int_{J_i} \vert \dot{\gamma}_i(t)\vert^2_g dt, \quad U_i =  \int_{J_i} h-V(\gamma_i(t))dt.
		\]
		With a slight abuse of notation on $\mathcal{M}_h$, we have
		\[
		\mathcal{M}_h(\gamma) = (K_1+K_2)(U_1+U_2) = K_1U_1+K_2U_2 + K_1U_2+K_2U_1 \ge\mathcal{M}_h(\gamma_1) + \mathcal{M}_h(\gamma_2),
		\]
		since the terms $U_i\ge0$.
		
		Now, assume by contradiction that $\gamma$ is a minimiser in $\mathcal{H}_\Delta(\tau)$ but $\gamma_1$ is not a minimiser in $\mathcal{H}_{(\gamma(a),\gamma(b))}(\tau)$. Take $v$ in the same space with $\mathcal{M}_h(v)<\mathcal{M}_h(\gamma_1)$ and define $K_v$ and $U_v$ as the kinetic and potential integrals of $v$, as before. By time-rescaling invariance, we can assume that $K_1 = K_v$ and the inequality on $\mathcal{M}_h$ turns then into $U_1>U_v$. At this point, if we set $\tilde \gamma = v \#\gamma_2$ we have:
		\[
			\mathcal{M}_h( \gamma) = K_vU_1+K_2U_2 + K_vU_2+K_2U_1 >K_vU_v+K_2U_2 + K_vU_2+K_2U_v = 	\mathcal{M}_h(\tilde \gamma),
		\]
		which is clearly against the minimality of $\gamma$ in $\mathcal{H}_\Delta(\tau)$. 
	\end{proof}

	\subsection{Existence of minimisers}
	
	As a starting point, by means of direct methods in the calculus of variations, we prove the existence of a minimiser of the Maupertuis functional, possibly interacting with the singularity set $\mathcal{C}$. We work in $\mathcal{H}_\Delta(\tau)$, which is a weakly closed subset of $H^1(J,\R^3)$.  Whenever we will use the symbols $\|\cdot\|_2$ and $\|\cdot\|_{H^1}$ we will mean the norms with respect to the ambient Sobolev space, while $\vert\cdot\vert$ stands for the Euclidean norm in $\R^3$. The proof is based on this preliminary result. 
	\begin{lem}
		Assume that $\tau$ is an admissible homotopy class as in \emph{Definition \ref{def:admissible-class}}. Then, the functional $\mathcal{M}_h$ is coercive and weakly lower semi-continuous on $\mathcal{H}_\Delta(\tau)$.
	\end{lem}
	\begin{proof}
		We first prove coercivity. Take $(\gamma_n)\sset\mathcal{H}_\Delta(\tau)$ such that $\|\gamma_n\|_{H^1}\to+\infty$; if $\|\dot{\gamma}_n\|_2\to+\infty$ we are done, since by \eqref{hyp:energy_bound} and \eqref{hyp:bounds_metric} we have
		\[
		\mathcal{M}_h(\gamma_n)\ge\frac{1}{2\Lambda}\|\dot{\gamma}_n\|_2^2\int_J\left[h-V(\gamma_n)\right]\ge C_1\|\dot\gamma_n\|_2^2,
		\]
	    for some $C_1>0$. 
	    
		Now assume that $\|\gamma_n\|_2\to+\infty$ and, without loss of generality, that every $\gamma_n$ is defined on the same interval $[0,1]$. Clearly, if the surface $M$ is compact, we have nothing to prove since $\|\gamma_n\|_2$ goes to infinity if and only if $\|\dot \gamma_n\|_2$ does so. If $M$ is non-compact, recalling Definition \ref{def:admissible-class},  without loss of generality we can assume that $\tau$ is such that there exists a compact subset $K$ of $M$ for which  $K \cap \gamma_n \ne \emptyset$ for any $n$. Recall that, for fixed $s\in[0,1]$, the distance between $\gamma_n(s)$ and $K$ is defined as
		\[
		d_g(\gamma_n(s),K) = \min_{p \in K} d_g(\gamma(s),p).
	    \]
        Since $K$ is compact, the minimum is obtained and there exists $p^*\in K$ (depending on s) such that 
        \[
        d_g(\gamma_n(s),K) = d_g(\gamma_n(s),p^*).
        \]
        Moreover $d_g(\gamma_n(s),p^*) \ge \lambda \vert \gamma_n(s)-p^*\vert$ for any $s$ (where $\lambda$ is given in \eqref{hyp:bounds_metric}) and thus:
	    \[	    \max_{s\in[0,1]}d_g(\gamma_n(s),K) \ge \lambda \max_{s\in[0,1]} \vert \gamma(s)-p^*\vert \ge \lambda\vert  \|\gamma_n\|_{\infty} -\vert p^*\vert \vert \ge \lambda \vert \|\gamma_n\|_2-\vert p^*\vert \vert
	    \]
        On the other hand, there exists $t_n \in \gamma_n^{-1}K$ and thus $d_g(\gamma_n(s),K) \le d_g(\gamma_n(s),\gamma_n(t_n))$. Moreover:
        \[
      	d_g(\gamma_n(s),\gamma_n(t_n)) \le \ell(\gamma_n\vert_{[s,t_n]}) \le \int_0^1 \vert \dot{\gamma}_n\vert_g \le C \mathcal{M}_h(\gamma_n).
        \]
        Combining the two inequalities we get the desired conclusion.
        
		Concerning the weakly lower semi-continuity, the product of two positive \emph{$\mathbb{R}-$valued} lower semi-continuous functions is a lower semi-continuous function. In principle $\mathcal{M}_h$ is $\bar{\mathbb{R}}$-valued.  However, the admissibility of the homotopy class $\tau$ guarantees that $\|\dot\gamma_n\|_2^2\geq C>0$ uniformly. Similarly, the potential part is uniformly bounded from below and lower semi-continuous. This is enough to conclude the proof.  
	\end{proof}
	As a direct consequence, we have the following result:
	\begin{prop}\label{prop:existence}
		Assume that $\tau$ is an admissible homotopy class as in \emph{Definition \ref{def:admissible-class}}. For every $h$ satisfying \eqref{hyp:energy_bound}, the functional $\mathcal{M}_h$ attains its minimum on  $\mathcal{H}_\Delta(\tau)$.
	\end{prop}

	\subsection{Properties of the minimisers}
	In this section we collect some qualitative properties of minimisers of the Maupertuis functional. First of all, we deal with points in the boundary of $\mathcal{H}_\Delta(\tau)$. 
	
	\begin{prop}[Collisions are isolated]
		\label{prop:collision_isolated}
		Assume that $\gamma \in \mathcal{H}_\Delta(\tau)$ is a minimiser of the Maupertuis functional \eqref{eq:def_Maupertuis} and has collisions. Then, the set
		\[
		\mathcal{I}_c = \{t\in J : \gamma(t) = c_j,\ j \in \{1,\dots, N \}\}
		\]
		is a finite set.
		\begin{proof}
			Assume by contradiction that $\vert \mathcal{I}_c\vert = \infty$. Since $\mathcal{M}_h(\gamma)<+\infty$, $\mathcal{I}_c$ has measure zero and does not contain any proper subinterval. Pick a strictly monotone sequence $(t_n)\sset\mathbb{N}$ contained in  $\mathcal{I}_c$ and label the corresponding collisions by $c_n$. As a consequence of \eqref{hyp:bounds_metric}, \eqref{hyp:energy_bound} and since $\mathcal{M}_h$ is finite on $\gamma$, we see that
			\[
			\sum_n d_g (c_{n},c_{n+1}) \le\sum_n \int_{t_n}^{t_{n+1}} \vert \dot \gamma \vert_g^2 \le \int_J\vert \dot \gamma \vert_g^2 \le C \mathcal{M}_h(\gamma) < +\infty,
			\]
			and so the sequence $(c_n)$ is definitively constant. 
			
			Since $\gamma$ is uniformly continuous, any accumulation point of $\mathcal{I}_c$ has a neighbourhood which contains only collisions with a fixed centre.
			
		    Arguing similarly and applying Proposition \ref{lem:restriction}, we see that if $s_0$ is an accumulation point of $\mathcal{I}_c$ and $(s_n)$ a strictly increasing sequence converging to it, we have that:
		    \[
		    \sum_n \mathcal{M}_h(\gamma\vert_{[s_n,s_{n+1}]})\le \mathcal{M}_h (\gamma)\ \text{ which implies }\ \mathcal{M}_h(\gamma\vert_{[s_n,s_{n+1}]}) \to 0.
		    \]		
		    Moreover, the image of the curves $\gamma\vert_{[s_n,s_{n+1}]}$ is arbitrarily close to the collision centre. Recall that since $\gamma$ is a minimiser, we must have that $\mathcal{M}_h(\gamma\vert_{[s_n,s_{n+1}]})$ are definitely all equal (and thus all equal to $0$) since we can exchange the segments $\gamma\vert_{[s_n,s_{n+1}]}$. Thus $\gamma$ is constant in a neighbourhood of $s_0$, which contradicts $\mathcal{M}_h(\gamma)<+\infty$.	    
		\end{proof} 
	\end{prop}
	
	\begin{rmk}[Lagrange-Jacobi inequality]\label{rem:lagrange_jacobi}
		Usually, Proposition \ref{prop:collision_isolated} is proved using the convexity of the function $d_g(c_j,\gamma(t))^2$ along collision solutions (see for instance \cite[Lemma 4.25]{SoaTer2012} for the case $M=\R^2$). In particular, it applies also to critical points. The proof given above has a different flavour and does not rely on the explicit form of the potential, just on the structure of the variations space and on minimality. In any case, a version of the classical Lagrange-Jacobi inequality can be easily proved in this context too. To lighten the notation, assume that $\gamma$ collides with a centre $c$ and that the homogeneity degree of $V$ close to $c$ is $\al$. By direct computation one gets:
		\[
		\begin{aligned}
			\frac12\frac{d^2}{dt^2} d_g(c,\gamma(t))^2 &= g(X,\dot\gamma)^2+ d_g(c,\gamma)\left(g \left(\Ddt X,\dot\gamma\right)+g\left(X,\Ddt\dot\gamma\right)\right) \\
			&=g\left(X,\dot\gamma\right)^2 + d_g(c,\gamma)\left(g\left(\Ddt X,\dot\gamma\right)-g\left(X,\nabla V(\gamma)\right)\right),
		\end{aligned}
	    \]
		where $X$ stands for the velocity of the unit speed geodesic joining $c$ and $\gamma(t)$. Let us consider the term $g\left(\Ddt X,\dot{\gamma}\right)$. Since integral curves of $X$ are geodesics, we have that $\nabla_X X =0$. We can extend $X$ to a orthonormal frame $\{X,Y\}$ on a small annulus around $c$. It follows that $k = g(\nabla_YX,Y)$ is the \emph{mean curvature} of the Riemannian balls around $c$. Close to $c$, we have the asymptotic relation $k(p) \sim d_g(c,p)^{-1} + O(d_g(c,p))$. It follows that: 
		\[
		\frac{D}{dt} X = g(\dot{\gamma},Y)\nabla_Y X \vert_\gamma= g(\dot{\gamma},Y) k(\gamma) Y\vert_\gamma,
		\]
		from which we see that:
		\[
		d_g(c,\gamma) g \left( \Ddt X,\dot\gamma\right) = d_g(c,\gamma) \, k(\gamma)\,  g(Y,\dot \gamma)^2 \sim  g(Y,\dot \gamma)^2.
		\]
		Now, recalling the definition of the potential $V$ (see \eqref{eq:potential}), close to $c$ we can write $V$ as a singular plus a regular part $U$ as follows:
		\[
		V(\gamma(t)) = -\frac{m_i}{\alpha d_g(c,\gamma(t))^{\alpha}} + U(\gamma(t)).
		\]
		Computing the gradient of $V$ and substituting in the equation above, it yields:
		\begin{equation}
			\label{eq:lagrange_jacobi_exponential_coordinates}
			\begin{aligned}
			\frac12\frac{d^2}{dt^2} d_g(c,\gamma(t))^2 &= g(\dot{\gamma},X)^2 - \frac{m_i}{d_g(c,\gamma(t))^{\alpha}} + d_g(c,\gamma(t))\left( k(\gamma)\,  g(Y,\dot \gamma)^2+g( \nabla U(\gamma(t)), X)\right)\\
			&= \vert\dot\gamma(t)\vert_g^2- \frac{m_i}{d_g(c,\gamma(t))^{\alpha}} +o(1) =\frac{2-\alpha}{\alpha} \frac{m_i}{d_g(c,\gamma(t))^\alpha} +O(1),
			\end{aligned}
		\end{equation} 
		where, in the last equality, we have used the conservation of energy \eqref{eq:energy}. This shows that $\frac{d^2}{dt^2} d_g(c,\gamma(t))$ blows up to $+\infty$ as $\gamma(t)$ approaches $c$, providing strict convexity.
	\end{rmk}
	
	It is well known that outside the collision set $\mathcal{I}_c$, minimisers of the Maupertuis functional \eqref{eq:def_Maupertuis} are $\mathscr C^2$ and satisfy a non linear system of ODEs. The following proposition holds as an application of Proposition \ref{prop:maupertuis_principle}.
	\begin{prop}[Regularity outside the collision set]
		\label{prop:regularity_outside_collision_set}
		Assume that $\gamma$ is a minimiser of \eqref{eq:def_Maupertuis}. For every subinterval $I\sset J\setminus \mathcal{I}_c$, the restriction $\gamma|_I$ belongs to $\mathscr C^2(I,\widehat{M})$. Moreover, it is a re-parametrisation of a solution of the following system:
		\begin{equation}
			\label{eq:euler_lagrage_equations}
			\frac{D \dot \eta}{dt} = -\nabla V(\eta), \quad \eta(t) = \gamma(\omega t),
		\end{equation}
		where the parameter $\omega>0$ is determined as in \eqref{eq:omega}.
	\end{prop}
	
	Note that \eqref{eq:euler_lagrage_equations} is a Lagrangian system. Thus the associated Lagrangian $\frac{1}{2}\vert \dot \eta \vert -V(\eta)$ is locally constant on $J \setminus \mathcal{I}_c$. It is known that the total energy is conserved through collisions too, as the following shows.
	
	\begin{prop}[Conservation of energy through collisions]
		\label{prop:conservation_energy_through_collisions}
		Assume that $\gamma$ is a minimiser of \eqref{eq:def_Maupertuis} on $\mathcal{H}_\Delta(\tau)$. Then, we have that
		\[
		\frac12|\dot{\gamma}(t)|_g^2+V(\gamma(t))=h,\quad\text{for a.e.}\ t\in J.
		\]
		\begin{proof}
			Even if $\gamma$ may be on the boundary of $\mathcal{H}_\Delta(\tau)$, we can still exploit extremality of $\gamma$ with respect to time reparametrisations. Take $\varphi\in\mathscr{C}_c^\infty(\mathring{J})$ and define the function $f_\lambda(t)=t+\lambda\varphi(t)$. It is easy to see that if $|\lambda|$ is sufficiently small, say $\lambda\in[-\delta,\delta]$, then $f_\lambda$ is a change of variable on the interval $J$. Now, let us define the loop 
			\[
			\gamma_\lambda(t)=\gamma(f_\lambda(t))=\gamma(t+\lambda\varphi(t)).
			\]
			It clearly belongs to the space $\mathcal{H}_\Delta(\tau)$ and thus $\mathcal{M}_h(\gamma)\leq\mathcal{M}_h(\gamma_\lambda)$, for every $\lambda\in[-\delta,\delta]$. Defining the new variable $s=f_\lambda(t)=t+\lambda\varphi(t)$, we can write		
			\[
			\begin{aligned}
				\mathcal{M}_h(\gamma_\lambda)&=\frac12\int_J\left\lvert\frac{d}{dt}\gamma_\lambda(t)\right\rvert_g^2\,dt\int_J\left[h-V(\gamma_\lambda(t))\right]\,dt\\&=\frac12\int_J|\dot{\gamma}(s)|_g^2(1+\lambda\dot{\varphi}(f_\lambda^{-1}(s)))\,ds\int_J\frac{h-V(\gamma(s))}{1+\lambda\dot{\varphi}(f_\lambda^{-1}(s))}\,ds.
			\end{aligned}
	     	\]
			On the other hand, we can write the time variable $t$ as an implicit function of $s$ in this way
			\[
			t(s)=f_\lambda^{-1}(s)=s-\lambda\varphi(f_\lambda^{-1}(s));
			\]
			moreover, we can provide the following estimate 
			\[
			|f_\lambda^{-1}(s)-s|=|s-\lambda\varphi(f_\lambda^{-1}(s))-s|\leq|\lambda|\|\varphi\|_{\infty}\to 0^+,\quad\text{as}\ \lambda\to 0
			\]
			so that $f_\lambda^{-1}(s)$ uniformly converge to $s$ in $J$ as $\lambda\to 0$. For this reason, the minimality condition can be written as follows
			\[
			\begin{aligned}
				0&=\frac{d}{d\lambda}\mathcal{M}_h(\gamma_\lambda)\Big\rvert_{\lambda=0} \\
				&=\frac12\int_J|\dot{\gamma}(s)|_g^2\dot\varphi(s)\,ds\int_J\left[h-V(\gamma(s))\right]\,ds-\frac12\int_J|\dot{\gamma}(s)|_g^2\,ds\int_J\left[h-V(\gamma(s))\right]\dot\varphi(s)\,ds \\
				&=\int_J\left\lbrace\frac12\left(\int_J\left[h-V(\gamma(s))\right]\,ds\right)|\dot{\gamma}(s)|_g^2-\left(\frac12\int_J|\dot{\gamma}(s)|_g^2\,ds\right)\left[h-V(\gamma(s))\right]\right\rbrace\dot\varphi(s)\,ds.
			\end{aligned}
	    	\]
			Since the previous holds for any $\varphi\in\mathscr{C}_c^\infty(\mathring{J})$, we have that there exists $k\in\R$ such that
			\[
			\frac12\left(\int_J\left[h-V(\gamma(s))\right]\,ds\right)|\dot{\gamma}(s)|_g^2-\left(\frac12\int_J|\dot{\gamma}(s)|_g^2\,ds\right)\left[h-V(\gamma(s))\right]=k,\quad\text{for a.e.}\ s\in J.
			\]
			Recalling the expression \eqref{eq:omega} of $\omega^2$, dividing both sides by $\frac12\int_J|\dot{\gamma}(s)|_g^2\,ds$ we obtain
			\[
			\frac{\omega^2}{2}|\dot{\gamma}(s)|_g^2-h+V(\gamma(s))=k;
			\]
			note that here we have used the fact that the minimum is attained at a positive level and we have used the same constant $k$. 
			Integrating both sides in $J$, we obtain that
			\[
			\omega^2=\frac{\int_J\left[h+k-V(\gamma(s))\right]\,ds}{\frac12\int_J|\dot{\gamma}(s)|_g^2\,ds}
			\]
			and definition \eqref{eq:omega} gives $k=0$. 
		\end{proof}
	\end{prop}
	
	Another important feature of minimisers of the Maupertuis functional on $\mathcal{H}_\Delta(\tau)$ is that they \emph{tend to be} taut, meaning that they tend to minimise the number of self-intersections in their homotopy class. Intuitively, there can be no $1-$gon or $2-$gon in the regular portion of a minimiser (recall Definition \ref{def:monogon-bigon}). However, much more attention should be paid when the minimisers lie on the boundary of $\widehat{\mathcal{H}}_\Delta(\tau)$ and have some collisions with the centres. For such singular curves the word \emph{taut} does not make much sense. However, the number of self-intersections of these \emph{singular} minimisers can be reduced by strongly exploiting the minimality of $\mathcal{M}_h(\gamma)$ and employing some careful surgical procedures.
	
	\begin{prop}
		\label{prop:no_excess_intesections_regular}
		Suppose that $\gamma$ is a minimiser of \eqref{eq:def_Maupertuis} on $\mathcal{H}_\Delta(\tau)$. Then:
		\begin{enumerate}[label = \roman*)]
			\item $\gamma$ has no singular 1-gon in $\widehat{M}$;
			\item $\gamma$ has no singular 2-gon in $\widehat{M}$;    
			\item $\gamma$ has no singular 1-gon in $\widehat{M}\cup \{c_j\}$ for any $j$;
			\item $\gamma$ has no $2-$gon colliding with just one centre, i.e., no singular $2-$gon in $\widehat{M}\cup \{c_j\}$ for any $j$;
			\item all isolated self-intersections are transversal.
		\end{enumerate}
		\begin{proof}
			Property $i)$ follows from super-additivity of the Maupertuis functional. In fact, if $[a,b]\subset J$ is an interval such that $\gamma(a) = \gamma(b)$ and $\gamma\vert_{[a,b]}$ is null-homotopic, the loop $\eta \uguale\gamma\vert_{J\setminus(a,b)/(a\sim b)}$ is still continuous and in the same homotopy class as $\gamma$. However, $\mathcal{M}_h(\gamma) >\mathcal{M}_h(\eta)$.
			
			To see $ii)$ we use regularity. Suppose that we can find two intervals $[a,b]$ and $[c,d]$ such that, for instance,  $\gamma(a) = \gamma(c)$, $\gamma(b) = \gamma(d) $ and $\gamma\vert_{[a,b]} \#\gamma\vert_{[c,d]}^{-1} $ is null-homotopic. Then we can obtain a curve $\eta$ exchanging $\gamma\vert_{[a,b]}$ and $\gamma\vert_{[c,d]}$. This curve has the same Maupertuis value as $\gamma$, but it is no longer $\mathscr C^2$ and thus it cannot be a minimiser (see Proposition \ref{prop:regularity_outside_collision_set}). 
			
			The argument for $iii)$ is very similar to case $ii)$. Assume first that there exists $[a,b]$ as in $i)$ and $\gamma(a) \ne c_j$, but $\gamma\vert_{[a,b]}$ is null-homotopic in $\mathbb{R}\setminus \mathcal{C} \cup \{c_j\}$. Then, if we run twice the portion between $\gamma(a)$ and $c_j$, we end up with another minimiser which is not $\mathscr C^2$. The case in which $\gamma(a) =c_j$ follows from super-additivity as in $i)$ (see Figure \ref{fig:1gon_regular}). Indeed, if we remove the sub-loop $\gamma\vert_{[a,b]}$ from $\gamma$, we obtain a curve still lying in the boundary of $\mathcal{H}_\Delta(\tau)$, but on which the Maupertuis functional takes a lower value.
			
			Point $iv)$ is again a matter of regularity. Switching between two possible branches gives a non $\mathscr C^2$ minimiser exactly as in the previous point (see Figure \ref{fig:2gon_regular}).
						
			Point $v)$ is obvious given Proposition \ref{prop:regularity_outside_collision_set}. In fact, by uniqueness for Cauchy problems, if position and velocity coincide at some instant, we must be dealing with two pieces of the same trajectory. Conservation of energy implies that the norm of the velocity is a function of the position and initial condition alone. Thus, if $\gamma$ self intersects, the velocities cannot be multiples unless they coincide or are opposite, and in this latter case we are dealing with a time inversion. 
		\end{proof}
	\end{prop}
	
	\begin{figure}[t]
		\centering
		\resizebox{.9\textwidth}{!}{\tikzset{every picture/.style={line width=0.75pt}} 

\begin{tikzpicture}[x=0.75pt,y=0.75pt,yscale=-1,xscale=1]
	
	\draw    (490.45,60.64) .. controls (509.34,75.98) and (579.36,123.99) .. (606.04,103.99) .. controls (632.71,83.98) and (634.71,41.3) .. (610.7,46.64) .. controls (586.7,51.97) and (534.68,121.32) .. (522.01,144) ;
	\draw  [dash pattern={on 0.84pt off 2.51pt}]  (490.45,60.64) .. controls (479.33,51.97) and (486.67,58.64) .. (473.33,45.3) ;
	\draw  [dash pattern={on 0.84pt off 2.51pt}]  (522.01,144) .. controls (517.35,153.33) and (518.68,152.67) .. (514.68,156) ;
	\draw  [fill={rgb, 255:red, 0; green, 0; blue, 0 }  ,fill opacity=1 ] (549.8,98.98) .. controls (549.8,97.33) and (551.14,95.98) .. (552.8,95.98) .. controls (554.46,95.98) and (555.8,97.33) .. (555.8,98.98) .. controls (555.8,100.64) and (554.46,101.99) .. (552.8,101.99) .. controls (551.14,101.99) and (549.8,100.64) .. (549.8,98.98) -- cycle ;
	\draw    (81.16,36.04) .. controls (84.12,61.06) and (101.45,147.24) .. (135.69,151.63) .. controls (169.93,156.03) and (202.22,125.79) .. (180.55,112.43) .. controls (167.67,104.49) and (131.15,106.19) .. (98.22,110.49) .. controls (75.72,113.43) and (54.9,117.58) .. (44.44,120.7) ;
	\draw  [dash pattern={on 0.84pt off 2.51pt}]  (81.16,36.04) .. controls (79.17,21.58) and (79.8,31.83) .. (79.53,12.3) ;
	\draw  [dash pattern={on 0.84pt off 2.51pt}]  (44.44,120.7) .. controls (34.24,124.26) and (35.71,124.73) .. (30.33,124.32) ;
	\draw  [fill={rgb, 255:red, 0; green, 0; blue, 0 }  ,fill opacity=1 ] (122.41,83.8) .. controls (122.41,82.09) and (123.8,80.7) .. (125.52,80.7) .. controls (127.23,80.7) and (128.62,82.09) .. (128.62,83.8) .. controls (128.62,85.52) and (127.23,86.91) .. (125.52,86.91) .. controls (123.8,86.91) and (122.41,85.52) .. (122.41,83.8) -- cycle ;
	\draw  [fill={rgb, 255:red, 155; green, 155; blue, 155 }  ,fill opacity=1 ] (550.91,98.98) .. controls (550.91,97.94) and (551.76,97.1) .. (552.8,97.1) .. controls (553.84,97.1) and (554.69,97.94) .. (554.69,98.98) .. controls (554.69,100.03) and (553.84,100.87) .. (552.8,100.87) .. controls (551.76,100.87) and (550.91,100.03) .. (550.91,98.98) -- cycle ;
	\draw  [fill={rgb, 255:red, 155; green, 155; blue, 155 }  ,fill opacity=1 ] (123.56,83.8) .. controls (123.56,82.72) and (124.44,81.85) .. (125.52,81.85) .. controls (126.6,81.85) and (127.47,82.72) .. (127.47,83.8) .. controls (127.47,84.89) and (126.6,85.76) .. (125.52,85.76) .. controls (124.44,85.76) and (123.56,84.89) .. (123.56,83.8) -- cycle ;
	\draw    (373.21,27.43) .. controls (359.45,47.38) and (317.21,120.63) .. (339.15,145.56) .. controls (361.08,170.49) and (403.61,169.17) .. (396.45,145.75) .. controls (389.29,122.32) and (316.39,76.06) .. (292.9,65.24) ;
	\draw  [dash pattern={on 0.84pt off 2.51pt}]  (373.21,27.43) .. controls (380.96,15.72) and (374.9,23.52) .. (387.11,9.24) ;
	\draw  [dash pattern={on 0.84pt off 2.51pt}]  (292.9,65.24) .. controls (283.26,61.33) and (284.03,62.6) .. (280.41,58.89) ;
	\draw  [fill={rgb, 255:red, 0; green, 0; blue, 0 }  ,fill opacity=1 ] (361.96,109.37) .. controls (363.05,108.13) and (364.94,108.01) .. (366.18,109.1) .. controls (367.42,110.19) and (367.54,112.07) .. (366.45,113.31) .. controls (365.36,114.55) and (363.47,114.67) .. (362.23,113.58) .. controls (360.99,112.49) and (360.87,110.6) .. (361.96,109.37) -- cycle ;
	\draw  [fill={rgb, 255:red, 155; green, 155; blue, 155 }  ,fill opacity=1 ] (362.79,110.1) .. controls (363.48,109.32) and (364.67,109.24) .. (365.45,109.93) .. controls (366.23,110.61) and (366.31,111.8) .. (365.62,112.58) .. controls (364.93,113.36) and (363.75,113.44) .. (362.96,112.75) .. controls (362.18,112.07) and (362.11,110.88) .. (362.79,110.1) -- cycle ;
	
	\draw (530.79,92.99) node [anchor=north west][inner sep=0.75pt]   [align=left] {$\displaystyle c_j$};
	\draw (565.98,113.66) node [anchor=north west][inner sep=0.75pt]   [align=left] {$\displaystyle \gamma  \vert_{[a,b]}$};
	\draw (144.2,150.41) node [anchor=north west][inner sep=0.75pt]   [align=left] {$\displaystyle \gamma  \vert_{[a,b]}$};
	\draw (132.67,64.79) node [anchor=north west][inner sep=0.75pt]   [align=left] {$\displaystyle c_{j}$};
	\draw (280.9,106.69) node [anchor=north west][inner sep=0.75pt]   [align=left] {$\displaystyle \gamma \vert_{[a,b]}$};
	\draw (377.66,98.19) node [anchor=north west][inner sep=0.75pt]   [align=left] {$\displaystyle c_{j}$};
	\draw (38,144.4) node [anchor=north west][inner sep=0.75pt]    {$a)$};
	\draw (258,144.4) node [anchor=north west][inner sep=0.75pt]    {$b)$};
	\draw (478,144.4) node [anchor=north west][inner sep=0.75pt]    {$c)$};

\end{tikzpicture}}
		\caption{Picture $a)$ shows the 1-gon described in point $i)$ of Proposition \ref{prop:no_excess_intesections_regular}, while pictures $b)$ and $c)$ depict two examples of the behaviour excluded by point $iii)$, namely when $\gamma(a)\neq c_j$ and $\gamma(a)=c_j$.}\label{fig:1gon_regular}
	\end{figure}
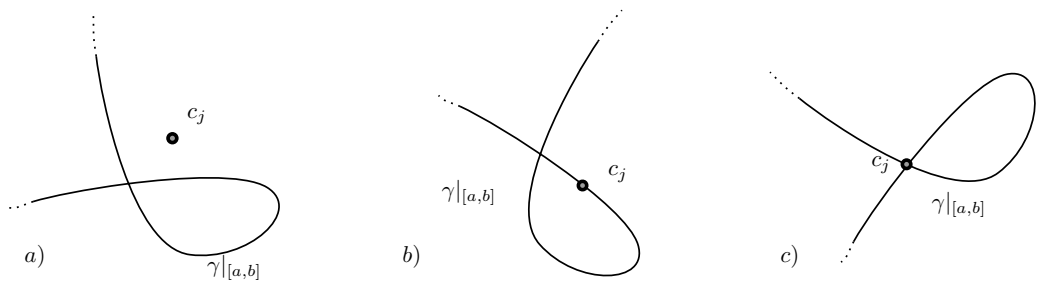

	\begin{figure}[t]
		\centering
		\resizebox{.7\textwidth}{!}{\begin{tikzpicture}[x=0.75pt,y=0.75pt,yscale=-1,xscale=1]
	
	\draw    (100,126) .. controls (113.33,39) and (187.33,13) .. (350.33,71) ;
	\draw    (84.33,80) .. controls (183.33,158) and (299.33,70) .. (339.33,40) ;
	\draw  [dash pattern={on 0.84pt off 2.51pt}]  (100,126) -- (97.33,145) ;
	\draw  [dash pattern={on 0.84pt off 2.51pt}]  (70.33,65) -- (84.33,80) ;
	\draw  [dash pattern={on 0.84pt off 2.51pt}]  (354.33,30) -- (339.33,40) ;
	\draw  [dash pattern={on 0.84pt off 2.51pt}]  (350.33,71) -- (367.33,77) ;
	\draw  [fill={rgb, 255:red, 0; green, 0; blue, 0 }  ,fill opacity=1 ] (105,95.5) .. controls (105,93.01) and (107.01,91) .. (109.5,91) .. controls (111.99,91) and (114,93.01) .. (114,95.5) .. controls (114,97.99) and (111.99,100) .. (109.5,100) .. controls (107.01,100) and (105,97.99) .. (105,95.5) -- cycle ;
	\draw  [fill={rgb, 255:red, 0; green, 0; blue, 0 }  ,fill opacity=1 ] (301,139.5) .. controls (301,137.01) and (303.01,135) .. (305.5,135) .. controls (307.99,135) and (310,137.01) .. (310,139.5) .. controls (310,141.99) and (307.99,144) .. (305.5,144) .. controls (303.01,144) and (301,141.99) .. (301,139.5) -- cycle ;
	\draw    (520.97,100.96) .. controls (506.94,187.85) and (432.73,213.25) .. (270.21,153.94) ;
	\draw    (536.27,147.08) .. controls (437.9,68.29) and (321.2,155.35) .. (280.96,185.03) ;
	\draw  [dash pattern={on 0.84pt off 2.51pt}]  (520.97,100.96) -- (523.79,81.98) ;
	\draw  [dash pattern={on 0.84pt off 2.51pt}]  (550.15,162.19) -- (536.27,147.08) ;
	\draw  [dash pattern={on 0.84pt off 2.51pt}]  (265.88,194.91) -- (280.96,185.03) ;
	\draw  [dash pattern={on 0.84pt off 2.51pt}]  (270.21,153.94) -- (253.25,147.81) ;
	\draw  [fill={rgb, 255:red, 155; green, 155; blue, 155 }  ,fill opacity=1 ] (106.67,95.33) .. controls (106.67,93.77) and (107.94,92.5) .. (109.5,92.5) .. controls (111.06,92.5) and (112.33,93.77) .. (112.33,95.33) .. controls (112.33,96.9) and (111.06,98.17) .. (109.5,98.17) .. controls (107.94,98.17) and (106.67,96.9) .. (106.67,95.33) -- cycle ;
	\draw  [fill={rgb, 255:red, 155; green, 155; blue, 155 }  ,fill opacity=1 ] (302.67,139.5) .. controls (302.67,137.94) and (303.94,136.67) .. (305.5,136.67) .. controls (307.06,136.67) and (308.33,137.94) .. (308.33,139.5) .. controls (308.33,141.06) and (307.06,142.33) .. (305.5,142.33) .. controls (303.94,142.33) and (302.67,141.06) .. (302.67,139.5) -- cycle ;
	
	\draw (79,92) node [anchor=north west][inner sep=0.75pt]   [align=left] {$\displaystyle c$};
	\draw (318,114) node [anchor=north west][inner sep=0.75pt]   [align=left] {$\displaystyle c_{j}$};
	\draw (154,84) node [anchor=north west][inner sep=0.75pt]   [align=left] {$\displaystyle \gamma \vert_{[a,b]}$};
	\draw (243.05,14.6) node [anchor=north west][inner sep=0.75pt]  [rotate=-10.48] [align=left] {$\displaystyle \gamma \vert_{[c,d]}$};
	\draw (456.43,187.39) node [anchor=north west][inner sep=0.75pt]  [rotate=-339.37] [align=left] {$\displaystyle \gamma \vert_{[a,b]}$};
	\draw (377.5,133.8) node [anchor=north west][inner sep=0.75pt]  [rotate=-344.32] [align=left] {$\displaystyle \gamma \vert_{[c,d]}$};
	
	\draw (96,34.4) node [anchor=north west][inner sep=0.75pt]    {$a)$};
	\draw (309,186) node [anchor=north west][inner sep=0.75pt]   [align=left] {$\displaystyle b)$};

\end{tikzpicture}}
		\caption{An example of the 2-gon excluded by $iv)$ of Proposition \ref{prop:no_excess_intesections_regular} is shown in picture $a)$, while $b)$ shows the 2-gons treated in $ii)$.}\label{fig:2gon_regular}
	\end{figure}
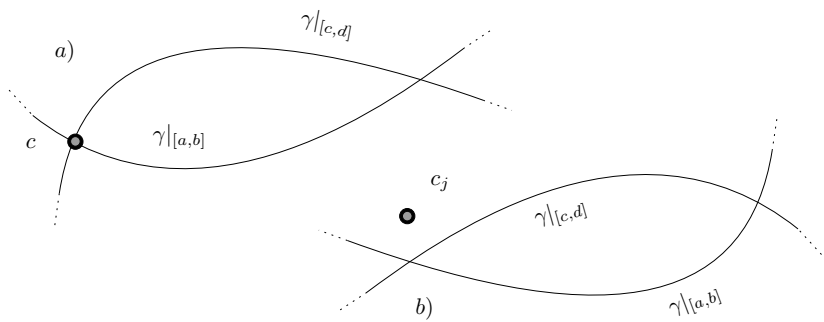
	
	We can extend points $iii)$ and $iv)$ of the previous proposition also to those minimisers which collide with more then one centre. However, in this case, we have to keep track of the homotopy class $\tau$ we are starting in.
	\begin{prop}
		\label{prop:no_excess_intesections_singular}
		Suppose that $\gamma$ is a minimiser of \eqref{eq:def_Maupertuis} on $\mathcal{H}_\Delta(\tau)$. Then:
		\begin{enumerate}[label = \roman*)]
			\item there is no interval $[a,b] \sset J$ and no subset of centres $\mathcal{C}'$ such that 
			\begin{itemize}
				\item  $\gamma(a) = \gamma(b)$,
				\item $\mathcal{C}' \sset \gamma([a,b])$,
				\item there exists a null homotopic curve $\eta$ in $\widehat{M}\cup\mathcal{C}' \cup \{\gamma(a)\}$ with the property that \[\gamma\vert_{J\setminus[a,b]}\#\eta\in\mathcal{H}_\Delta(\tau).\]
			\end{itemize}
			\item there are no intervals $[a,b],[c,d] \subset J$, with $b<c$, and no subset of centres $\mathcal{C}'$ such that:
			\begin{itemize}
				\item  $\gamma(a) = \gamma(c)$ and $\gamma(b) = \gamma(d)$ and at least one between these two points is not a centre,
				\item $\mathcal{C}' \sset \gamma([a,b]) \cup \gamma([c,d])$,
				\item there exists a curve $\eta$ in $\widehat{M} \cup \{\gamma(a),\gamma(b)\}$ joining $\gamma(a)$ and $\gamma(b)$, with the property that 
				\[
				\gamma\vert_{J \cap (-\infty,a]}\#\eta\#\gamma\vert_{[b,c]} \# \eta\# \gamma\vert_{J \cap [d,+\infty)}\in\mathcal{H}_\Delta(\tau).
				\]
			\end{itemize}
			\item if in point $ii)$ above $\gamma(a)$ and $\gamma(b)$ are both centres, we can build a minimiser $\tilde{\gamma}$ in $\mathcal{H}_\Delta(\tau)$ which coincides with $\gamma$ on $J \setminus\left([a,b]\cup [c,d]\right)$ and such that $\tilde{\gamma}([a,b]) = \tilde \gamma([c,d]) = \gamma([a,b]).$ 
		\end{enumerate}
		\begin{proof}
			The proof is completely analogous to the one of Proposition \ref{prop:no_excess_intesections_regular}. The assumptions guarantee that performing the same type of surgeries as before does not change the property of belonging to the boundary of $\widehat{\mathcal{H}}_\Delta(\tau)$. Point $i)$ follows from super-additivity, point $ii)$ from regularity and minimality and fro point $iii)$ minimality is enough: the Maupertuis functional must coincide on $\gamma\vert_{[a,b]}$ and $\gamma\vert_{[c,d]}$ (see also Figures \ref{fig:1gon_singular}-\ref{fig:2gon_singular}).
		\end{proof}
	\end{prop}
	
	\begin{figure}[t]
		\centering
		\resizebox{.7\textwidth}{!}{\begin{tikzpicture}[x=0.75pt,y=0.75pt,yscale=-1,xscale=1]
	
	\draw    (88,59) .. controls (93.06,63.1) and (99.63,69.76) .. (107.34,77.94) .. controls (142.82,115.61) and (202.33,185.71) .. (248.33,189) .. controls (304.33,193) and (344.33,63) .. (268.33,38) .. controls (192.33,13) and (88.33,146) .. (69.33,180) ;
	\draw  [dash pattern={on 0.84pt off 2.51pt}]  (73.33,45) -- (88,59) ;
	\draw  [dash pattern={on 0.84pt off 2.51pt}]  (62.33,192) -- (69.33,180) ;
	\draw  [dash pattern={on 4.5pt off 4.5pt}]  (130,103) .. controls (192.33,72) and (177.33,56) .. (210.33,61) .. controls (243.33,66) and (210.33,36) .. (258.33,50) .. controls (306.33,64) and (308.33,121) .. (298.33,134) .. controls (288.33,147) and (283.33,130) .. (282.33,159) .. controls (281.33,188) and (202.33,167) .. (206.33,150) .. controls (210.33,133) and (194.33,136) .. (182.33,141) .. controls (170.33,146) and (176.33,103) .. (130.5,103) ;
	\draw  [fill={rgb, 255:red, 0; green, 0; blue, 0 }  ,fill opacity=1 ] (126,103) .. controls (126,100.51) and (128.01,98.5) .. (130.5,98.5) .. controls (132.99,98.5) and (135,100.51) .. (135,103) .. controls (135,105.49) and (132.99,107.5) .. (130.5,107.5) .. controls (128.01,107.5) and (126,105.49) .. (126,103) -- cycle ;
	\draw  [fill={rgb, 255:red, 0; green, 0; blue, 0 }  ,fill opacity=1 ] (181,155) .. controls (181,152.51) and (183.01,150.5) .. (185.5,150.5) .. controls (187.99,150.5) and (190,152.51) .. (190,155) .. controls (190,157.49) and (187.99,159.5) .. (185.5,159.5) .. controls (183.01,159.5) and (181,157.49) .. (181,155) -- cycle ;
	\draw  [fill={rgb, 255:red, 0; green, 0; blue, 0 }  ,fill opacity=1 ] (200,46) .. controls (200,43.51) and (202.01,41.5) .. (204.5,41.5) .. controls (206.99,41.5) and (209,43.51) .. (209,46) .. controls (209,48.49) and (206.99,50.5) .. (204.5,50.5) .. controls (202.01,50.5) and (200,48.49) .. (200,46) -- cycle ;
	\draw  [fill={rgb, 255:red, 0; green, 0; blue, 0 }  ,fill opacity=1 ] (293,153) .. controls (293,150.51) and (295.01,148.5) .. (297.5,148.5) .. controls (299.99,148.5) and (302,150.51) .. (302,153) .. controls (302,155.49) and (299.99,157.5) .. (297.5,157.5) .. controls (295.01,157.5) and (293,155.49) .. (293,153) -- cycle ;
	\draw    (537.8,80.45) .. controls (532.03,83.46) and (523.32,86.86) .. (512.75,90.69) .. controls (502.18,94.53) and (382.33,107) .. (354.9,176.12) .. controls (327.47,245.25) and (463.33,315) .. (496.33,239) .. controls (529.33,163) and (457.79,46.27) .. (434.11,15.35) ;
	\draw  [dash pattern={on 0.84pt off 2.51pt}]  (556.47,72.53) -- (537.8,80.45) ;
	\draw  [dash pattern={on 0.84pt off 2.51pt}]  (425.87,4.17) -- (434.11,15.35) ;
	\draw  [dash pattern={on 4.5pt off 4.5pt}]  (480.33,96) .. controls (484.09,165.52) and (508.15,164.65) .. (490.48,192.97) .. controls (472.81,221.28) and (513.43,202.88) .. (481.55,241.4) .. controls (449.68,279.92) and (396.55,259.16) .. (388.58,244.83) .. controls (380.61,230.49) and (398.2,232.64) .. (371.97,220.23) .. controls (345.74,207.82) and (387.31,135.59) .. (401.33,146) .. controls (415.36,156.41) and (426.17,141) .. (426.33,128) .. controls (426.5,115) and (462.58,139.5) .. (480.75,97.42) ;
	\draw  [fill={rgb, 255:red, 0; green, 0; blue, 0 }  ,fill opacity=1 ] (401,125) .. controls (401,122.51) and (403.01,120.5) .. (405.5,120.5) .. controls (407.99,120.5) and (410,122.51) .. (410,125) .. controls (410,127.49) and (407.99,129.5) .. (405.5,129.5) .. controls (403.01,129.5) and (401,127.49) .. (401,125) -- cycle ;
	\draw  [fill={rgb, 255:red, 0; green, 0; blue, 0 }  ,fill opacity=1 ] (363,235) .. controls (363,232.51) and (365.01,230.5) .. (367.5,230.5) .. controls (369.99,230.5) and (372,232.51) .. (372,235) .. controls (372,237.49) and (369.99,239.5) .. (367.5,239.5) .. controls (365.01,239.5) and (363,237.49) .. (363,235) -- cycle ;
	\draw  [fill={rgb, 255:red, 0; green, 0; blue, 0 }  ,fill opacity=1 ] (501,201) .. controls (501,198.51) and (503.01,196.5) .. (505.5,196.5) .. controls (507.99,196.5) and (510,198.51) .. (510,201) .. controls (510,203.49) and (507.99,205.5) .. (505.5,205.5) .. controls (503.01,205.5) and (501,203.49) .. (501,201) -- cycle ;
	\draw  [fill={rgb, 255:red, 155; green, 155; blue, 155 }  ,fill opacity=1 ] (127.67,103) .. controls (127.67,101.44) and (128.94,100.17) .. (130.5,100.17) .. controls (132.06,100.17) and (133.33,101.44) .. (133.33,103) .. controls (133.33,104.56) and (132.06,105.83) .. (130.5,105.83) .. controls (128.94,105.83) and (127.67,104.56) .. (127.67,103) -- cycle ;
	\draw  [fill={rgb, 255:red, 155; green, 155; blue, 155 }  ,fill opacity=1 ] (182.67,155) .. controls (182.67,153.44) and (183.94,152.17) .. (185.5,152.17) .. controls (187.06,152.17) and (188.33,153.44) .. (188.33,155) .. controls (188.33,156.56) and (187.06,157.83) .. (185.5,157.83) .. controls (183.94,157.83) and (182.67,156.56) .. (182.67,155) -- cycle ;
	\draw  [fill={rgb, 255:red, 155; green, 155; blue, 155 }  ,fill opacity=1 ] (294.67,153) .. controls (294.67,151.44) and (295.94,150.17) .. (297.5,150.17) .. controls (299.06,150.17) and (300.33,151.44) .. (300.33,153) .. controls (300.33,154.56) and (299.06,155.83) .. (297.5,155.83) .. controls (295.94,155.83) and (294.67,154.56) .. (294.67,153) -- cycle ;
	\draw  [fill={rgb, 255:red, 155; green, 155; blue, 155 }  ,fill opacity=1 ] (201.67,46.17) .. controls (201.67,44.6) and (202.94,43.33) .. (204.5,43.33) .. controls (206.06,43.33) and (207.33,44.6) .. (207.33,46.17) .. controls (207.33,47.73) and (206.06,49) .. (204.5,49) .. controls (202.94,49) and (201.67,47.73) .. (201.67,46.17) -- cycle ;
	\draw  [fill={rgb, 255:red, 155; green, 155; blue, 155 }  ,fill opacity=1 ] (364.67,235) .. controls (364.67,233.44) and (365.94,232.17) .. (367.5,232.17) .. controls (369.06,232.17) and (370.33,233.44) .. (370.33,235) .. controls (370.33,236.56) and (369.06,237.83) .. (367.5,237.83) .. controls (365.94,237.83) and (364.67,236.56) .. (364.67,235) -- cycle ;
	\draw  [fill={rgb, 255:red, 155; green, 155; blue, 155 }  ,fill opacity=1 ] (402.67,125) .. controls (402.67,123.44) and (403.94,122.17) .. (405.5,122.17) .. controls (407.06,122.17) and (408.33,123.44) .. (408.33,125) .. controls (408.33,126.56) and (407.06,127.83) .. (405.5,127.83) .. controls (403.94,127.83) and (402.67,126.56) .. (402.67,125) -- cycle ;
	\draw  [fill={rgb, 255:red, 155; green, 155; blue, 155 }  ,fill opacity=1 ] (502.67,201) .. controls (502.67,199.44) and (503.94,198.17) .. (505.5,198.17) .. controls (507.06,198.17) and (508.33,199.44) .. (508.33,201) .. controls (508.33,202.56) and (507.06,203.83) .. (505.5,203.83) .. controls (503.94,203.83) and (502.67,202.56) .. (502.67,201) -- cycle ;
	\draw  [fill={rgb, 255:red, 0; green, 0; blue, 0 }  ,fill opacity=1 ] (479.33,97.42) .. controls (479.33,96.63) and (479.97,96) .. (480.75,96) .. controls (481.53,96) and (482.17,96.63) .. (482.17,97.42) .. controls (482.17,98.2) and (481.53,98.83) .. (480.75,98.83) .. controls (479.97,98.83) and (479.33,98.2) .. (479.33,97.42) -- cycle ;
	
	\draw (177,198) node [anchor=north west][inner sep=0.75pt]   [align=left] {$\displaystyle \gamma \vert_{[a,b]}$};
	\draw (505,118) node [anchor=north west][inner sep=0.75pt]   [align=left] {$\displaystyle \gamma\vert_{[a,b]}$};
	\draw (87,94) node [anchor=north west][inner sep=0.75pt]   [align=left] {$\displaystyle \gamma ( a)$};
	\draw (480,67) node [anchor=north west][inner sep=0.75pt]   [align=left] {$\displaystyle \gamma ( a)$};
	\draw (180,17) node [anchor=north west][inner sep=0.75pt]   [align=left] {$\displaystyle c_{j_1}$};
	\draw (299.5,160.5) node [anchor=north west][inner sep=0.75pt]   [align=left] {$\displaystyle c_{j_2} $};
	\draw (154,158) node [anchor=north west][inner sep=0.75pt]   [align=left] {$\displaystyle c_{j_3} $};
	\draw (392,92) node [anchor=north west][inner sep=0.75pt]   [align=left] {$\displaystyle c_{j_1} $};
	\draw (520,205) node [anchor=north west][inner sep=0.75pt]   [align=left] {$\displaystyle c_{j_2} $};
	\draw (334,237) node [anchor=north west][inner sep=0.75pt]   [align=left] {$\displaystyle c_{j_3} $};
	\draw (188,75) node [anchor=north west][inner sep=0.75pt]   [align=left] {$\displaystyle \eta $};
	\draw (442,229) node [anchor=north west][inner sep=0.75pt]   [align=left] {$\displaystyle \eta $};

\end{tikzpicture}}
		\caption{The situation described in point $i)$ of Proposition \ref{prop:no_excess_intesections_singular}. Since $\gamma$ collides with the centres $c_{j_1}$, $c_{j_2}$ and $c_{j_3}$, the curve $\eta$ keeps trace of the homotopy class of $\gamma$. In the first case the point $\gamma(a)$ coincides with one centre, while in the other not.}\label{fig:1gon_singular}
	\end{figure}
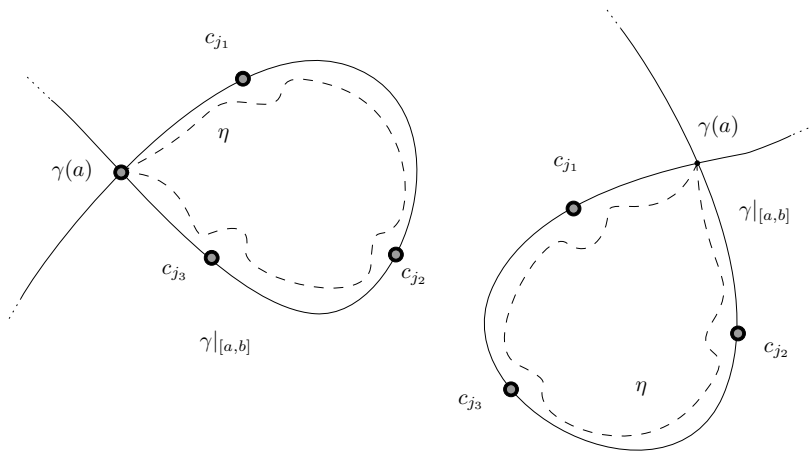
	 
	\begin{figure}[t]
		\centering
		\resizebox{.75\textwidth}{!}{\input{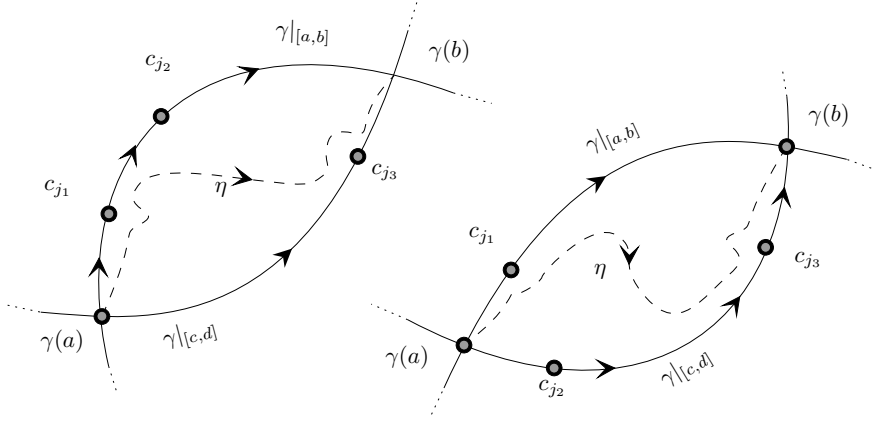}}
		\caption{The pathological situations described in points $ii)$ and $iii)$ of Proposition \ref{prop:no_excess_intesections_singular}.}\label{fig:2gon_singular}
	\end{figure}

	\section{Colliding trajectories, geodesics and obstacles}
	\label{sec:obstacle}
		
	In this section we introduce the main variational tools to exclude collisions: the obstacle problem and the blow-up analysis of collision solutions. The main idea is to study the qualitative properties of near collision solutions, going gradually closer to a singularity.  More precisely, the proof of Theorem \ref{thm:main_theorem} heavily relies on the properties of a suitable sequence of geodesics with obstacle, which approximates a collision solution in a small neighbourhood of a centre. After a rescaling, we identify the limit as a zero energy solution of a Kepler problem and we exploit some of its known properties.

	\subsection{The obstacle problem on surfaces}

	Assume that $\gamma$ is a minimiser of the Maupertuis functional \eqref{eq:def_Maupertuis} on $\mathcal{H}_\Delta(\tau)$ and that there exist a time $\bar{t} \in J$ and a centre $c$ such that $\gamma(\bar{t}) =c$. Proposition \ref{prop:collision_isolated} guarantees that such a collision instant is isolated and so we can switch to a local analysis of the minimiser.  In particular, we can find a subinterval $[a,b]\subset J$ such that
	\begin{itemize}
		\item $\bar{t}\in[a,b]$ and $\mathcal{I}_c\cap[a,b]=\{\bar{t}\}$;
		\item the function $I(t)=d_g(\gamma(t),c)^2$ is strictly convex in $[a,b]$, attaining its minimum in $\bar{t}$ (see Remark \ref{rem:lagrange_jacobi}).
	\end{itemize}	
	We call $p=\gamma(a)$ and $q=\gamma(b)$ and, without loss of generality, we assume that $p,q\in\partial B_r(c)$ for some $r>0$, where $B_r(c)$ is a metric ball with respect to $d_g$. The main idea here is to pass from a global analysis of the loop $\gamma$ to a local analysis of the path $\gamma|_{[a,b]}$, which is entirely contained in $B_r(c)$. For this reason, from now on we will lighten the notation on the homogeneity degrees $\al_j$ which appear in the definition of $V$ (see \eqref{eq:potential}), and we write
	\[
	V(q)\sim -\dfrac{m_j}{\al d_g(q,c)^\al}+U(q),
	\]
	for a smooth function $U$, whenever $q\in B_r(c)$. Indeed, since $\gamma$ is continuous, we can also suppose that
	\[
	d_g(\gamma(t),c_k)\geq C>2r,\quad\text{for any}\ t\in[a,b],\ c_k\neq c\ \text{and some}\ C>0,
	\]
	so that we can focus our efforts on a unique singularity $c$. A first important property required on $\gamma|_{[a,b]}$ is that it minimises the Maupertuis functional among all those paths which, concatenated with $\gamma|_{J\setminus[a,b]}$, belong to the space $\mathcal{H}_\Delta(\tau)$. Recalling the definition \eqref{eq:def_banach_manifold} of the Banach manifold $\mathcal{H}$, with a slight abuse of notation on the interval $J$, we introduce the space of paths
	\[
	\hat{K}=\{\eta\in\mathcal{H}:\eta(t)\in B_r(c)\setminus\{c\},\ \forall\,t\in[a,b],\ \eta(a)=p,\ \eta(b)=q\}
	\]
	and its weak $H^1$-closure $K$. Thanks to Lemma \ref{lem:restriction}, $\gamma|_{[a,b]}$ is a minimiser of $\mathcal{M}_h$ in the space
	\[
	\mathcal{K}(\tau)=\{\eta\in K:\,\eta\#\gamma|_{J\setminus[a,b]}\in\mathcal{H}_\Delta(\tau)\}.
	\]
	Following \cite{SoaTer2012}, for $\ve>0$ we introduce the function
	\[
	d(\ve)=\min\left\lbrace\mathcal{M}_h(\eta):\,\eta\in\mathcal{K}(\tau),\ \min_{t\in[a,b]}d_g(\eta(t),c)=\ve\right\rbrace,
	\]
	which is well-defined because the space
	\[
	\left\lbrace \eta\in\mathcal{K}(\tau):\,\min_{t\in[a,b]}d_g(\eta(t),c)=\ve\right\rbrace
	\]
	is weakly closed in $H^1$. \begin{lem}
		\label{lemma:continuity_d}
		The value $d(0)$ is achieved by $\gamma|_{[a,b]}$ and the function $d(\ve)$ is continuous in $\ve=0$. 
	\end{lem}
	\begin{proof}
		The first part of this lemma is again a consequence of Lemma \ref{lem:restriction}, while the continuity follows from a slight modification of \cite[Lemma 17]{TerVen2007}, once similar asymptotic estimates are provided (see also the proof of Proposition \ref{prop:convergence_blowups}). Indeed, using exponential coordinates around $c$, it is possible to describe the asymptotic behaviour of a collision solution as a small perturbation of the Euclidean case described in \cite{TerVen2007}.
	\end{proof}
	
	For $0<\ve_1<\ve_2$ define the space
	\[
	\mathcal{K}_{\ve_1,\ve_2}(\tau)=\left\lbrace \eta\in\mathcal{K}(\tau):\,\min_{t\in[a,b]}d_g(\eta(t),c)\in[\ve_1,\ve_2]\right\rbrace,
	\]
	which is weakly closed in $H^1$ too and so $\mathcal{M}_h$ admits a minimiser therein. For this reason, the following set of paths is well defined too:
	\[
	\tilde{\mathcal{K}}_{\ve_1,\ve_2}(\tau)=\left\lbrace \eta\in\mathcal{K}_{\ve_1,\ve_2}(\tau):\,\mathcal{M}_h(\eta)=\min_{\mathcal{K}_{\ve_1,\ve_2}(\tau)}\mathcal{M}_h,\ \min_{t\in[a,b]}d_g(\eta(t),c)<\ve_2\right\rbrace.
	\]
	Since we have assumed that $\gamma$ collides in $c$, it is reasonable to expect that the previous set of paths is definitely non-empty. This is the content of the next result. 
	\begin{lem}
		Assume that $\gamma$ collides with $c$. Then, for any $\ve>0$, there exist $0<\ve_1<\ve_2<\ve$ such that
		\[
		\tilde{\mathcal{K}}_{\ve_1,\ve_2}(\tau)\neq\emptyset.
		\]
	\end{lem}
	\begin{proof}
		The proof goes exactly as in \cite[Lemma 5.3]{Cas2017}.
	\end{proof}
	
    As a consequence of the previous lemma, we can find two sequences of positive numbers $(\ve_n),(\bar{\ve}_n)$ such that
    \[
    0<\ve_n<\bar{\ve}_n,\quad \ve_n,\bar\ve_n\to 0^+,
    \]
    and a sequence of paths $(\eta_n)\subset\mathcal{K}(\tau)$ such that
	\[
	\eta_n \in \tilde{\mathcal{K}}_{\ve_n,\bar{\ve}_n}(\tau), \quad \mathcal{M}_h(\eta_n) = d(\ve_n),\ \forall\,n\in\N.
	\]	
	In particular, since by Lemma \ref{lemma:continuity_d} the function $d$ is continuous at $0$, we see that
	\[
	\lim_{n\to+\infty} \mathcal{M}_h(\eta_n) = \mathcal{M}_h(\gamma|_{[a,b]}).
	\]	
	Since the paths $\eta_n$ are minimisers of a geodesic with obstacle problem (or geodesics on a surface with boundary problem, see for instance \cite{geodesic_with_obstacle_I,geodesic_with_obstacle_II}), they share some nice regularity properties which are summarised in the following result.

	\begin{prop}[Regularity of obstacle minimisers]
		\label{prop:regularity_obstacle}
		For any $n\in\N$:
		\begin{enumerate}[label=\roman*)]
			\item $\eta_n \in \mathscr{C}^1(a,b)$;
			\item  $\eta_n$ is of class $\mathscr{C}^2$ on any sub-interval of $[a,b] \setminus T_n \uguale  \eta_n^{-1}(\partial B_{\ve_n}(c))$ and solves the following second order system:
			\[
				\omega^2_n\frac{D \dot {\eta}_n}{d t} = - \nabla V(\eta_n), \quad \omega_n^2 = \frac{ \int_a^b\left[h-V(\eta_n)\right]}{\frac12\int_a^b \vert\dot \eta_n\vert_g^2};
			\]
			\item the total energy of $\eta_n(t)$ is constant on $[a,b]$. In particular:
			\[
				\frac{\omega_n^2}{2}\vert\dot \eta_n(t)\vert_g^2 + V(\eta_n(t)) = h,\ \forall t \in [a,b];
			\]
			\item  the set $T_n =  \eta_n^{-1}(\partial B_{\ve_n}(c))$ is an interval;
			\item using exponential coordinates centred at $c$,  we can write $\eta_n(t) = \exp_c(r_n(t) e^{i\vt_n(t)})$. The angular part $\vt_n(t)$ is strictly monotone and $\mathscr C^2$ on $T_n$.
		\end{enumerate}
		\begin{proof}
			Point $i)$ follows from \cite[Theorem 1]{geodesic_with_obstacle_I}. Point $ii)$ is basically a Maupertuis principle and thus follows by direct differentiation of the Maupertuis functional. Point $iii)$ is a consequence of the same argument used in Proposition \ref{prop:conservation_energy_through_collisions}. Point $iv)$ can be proved using a version of \eqref{eq:lagrange_jacobi_exponential_coordinates} for $\eta_n$. Indeed, using point $iii)$ one has
			\[
				\frac12\frac{d^2}{dt^2} d_g(c,\eta_n(t))^2 = g(X,\dot\eta_n(t))^2- \frac{m_i}{\omega_n^2d_g(c,\eta_n(t))^{\alpha}} + d_g(c,\eta_n(t))(k(\eta_n)g(Y,\dot{\eta}_n)^2+g( \nabla U(\eta_n(t)), X))
			\]
			on any interval contained in $[a,b] \setminus T_n$. Notice that, by the coercivity of $\mathcal{M}_h$, $\omega_n^2$ is a bounded sequence (see also \cite[Lemma 4.30]{SoaTer2012}). It follows that $\frac{d^2}{dt^2} d_g(c,\eta_n(t))^2$ can be uniformly bounded from below on $[a,b]$ for all $n$ sufficiently large, in a small ball centred at $c$. This implies that $d_g(c,\eta_n(t))^2$ are definitely convex near the obstacle, thus if there exist two instants $t_1<t_2$ such that $d_g(c,\eta_n(t_i)) = \ve_n$, then the same holds for all intermediate times. 
			
			Point $v)$ follows easily from the conservation of energy. In fact, introducing exponential coordinates on the obstacle and using $iv)$, we have:
			\[
			\begin{aligned}
			\dot \eta_n(t) &=\dot \vt_n(t) \ve_n \,d_{\ve_ne^{i\vt_n}}\exp_c(ie^{i\vt_n(t)}), \\ \frac{2(h-V(\eta_n(t))}{\omega_n^2} &=\vert d_{\ve_ne^{i\vt_n}} \exp_c(ie^{i\vt_n(t)})\vert_g^2 \dot\vt_n(t)^2 \ve_n^2.
			\end{aligned}
			\]			
			One easily deduces monotonicity of $\vt_n(t)$ from the last formula. Set $Y_t = d_{\ve_ne^{i\vt_n}} \exp_c (i e^{i \vt_n(t)})$, differentiating the energy identity yields an equation for $\vt_n(t)$:
			\begin{equation}
				\label{eq:ODE_angular_part_obstacle}
				\ddot{\vt}_n(t) =- \frac{1}{\vert Y_t \vert_g }\left(g \left(\frac{\nabla V (\eta_n(t))}{\omega_n^2 \ve_n}+\dot\vt_n(t) D_tY_t,\frac{Y_t}{\vert Y_t\vert_g}\right)\right).
			\end{equation}
		\end{proof}
	\end{prop}
	
	\subsection{Blow-up analysis}
	
	In the previous section we introduced the obstacle technique and we provided a sequence of curves $\eta_n(t)$ on $[a,b]$ and a sequence of radii $\ve_n$ converging to zero, which satisfy:
	\begin{enumerate}[label = \roman*)]
		\item $\eta_n(a) = p\in \partial B_r(c)$, and $\eta_n(b) = q \in \partial B_r(c)$,
		\item $\eta_n^{-1}(\partial B_{\ve_n}(c)) = T_n \uguale [t_n^-,t_n^+]$.
	\end{enumerate}
	Now, we want to investigate the behaviour of $\eta_n$ as the parameter $\ve_n$ goes to zero. In the flat case, we would define a blow-up sequence of the form $\ve_n^{-1}( \eta_n( \ve_n^\nu t)-c)+c$ for a suitable value of $\nu$. In this way we would map an open ball around $c$ to larger and larger balls contained in $\mathbb{R}^2$ and the obstacle $\partial B_{\ve_n}(c)$ to an euclidean sphere of radius 1. Since we are working on a surface $M$, we have to slightly modify the argument using exponential coordinates centred at $c$. In any case, the rate $\nu$ depends on the time spent by $\eta_n$ on the obstacle, i.e., on the quantity  $t_n^+-t_n^-$. The next lemma gives an estimate for this quantity.
	
	\begin{lem}
		Let $[t_n^-,t_n^+]$ be the interval on which $\eta_n(t)$ lies on the obstacle. There exists a positive constant $C_1$ such that:
		\[
		0\le t_n^+-t_n^- \le C_1 \ve_n^{\frac{\alpha+2}{2}}.
		\]
		\begin{proof}
			We use the same notation as in the proof of Proposition \ref{prop:regularity_obstacle}, referring to the proof of point $v)$. Assuming that $\vt_n$ is strictly increasing, when $t\in[t_n^-,t_n^+]$ and $\ve_n\to 0^+$ we have:
			\begin{equation}
				\label{eq:angular_velocity_expansion}
				\dot\vt_n(t) = \sqrt{\frac{2(h-V(\eta_n(t))}{\omega_n^2 \ve_n^2 \vert Y_t \vert_g^2}} =\sqrt{\frac{1}{\ve_n^{\alpha+2}}\left(\frac{2m_j}{\alpha\omega_n^2 \vert Y_t \vert_g^2} + o(\ve^{\alpha}_n)\right)}.
			\end{equation}
			Moreover, using the fundamental theorem of calculus:
			\[
			\vt_n(t_n^+)-\vt_n(t_n^-) = \ve_n^{-\frac{\alpha+2}{2}} \int_{t_n^-}^{t_n^+} \sqrt{\frac{2m_j}{\alpha \omega_n^2 \vert Y_t \vert_g^2} + o(\ve_n^{\alpha})} \ge C_n \ve_n^{-\frac{\alpha+2}{2}}(t_n^+-t_n^-).
			\]
			Notice that $C_n>0$ is uniformly bounded in $n$ and that the angular variation of $\theta_n$ is bounded by $2 \pi k$, $k \in \mathbb{N}.$ The number $k$ depends solely on the prescribed homotopy class  $[\tau]$ and can be estimated a priori.
		\end{proof}
	\end{lem}
	
	We now define the following family of continuous functions on $\mathbb{R}$ with values in $(T_cM,g) \sim (\mathbb{R}^2,g_e)$:
	\[
		u_n(s) = \begin{cases}
			\ve_n^{-1}\exp_c^{-1}(p) &\text{if } s< \ve_n^{-\frac{\alpha+2}{2}}(a-\frac{t_n^++t_n^-}{2}) \\
			\ve_n^{-1}\left(\exp_c^{-1}\eta_n(\ve_n^{\frac{\alpha+2}{2}} s+ \frac{t_n^++t_n^-}{2})\right) &\text{elsewhere} \\
			\ve_n^{-1} \exp_c^{-1}(q)  &\text{if } s> \ve_n^{-\frac{\alpha+2}{2}}(b-\frac{t_n^++t_n^-}{2})
		\end{cases}.
	\]
	
	\begin{prop}[Uniform convergence on compacts of $u_n$]\label{prop:convergence_blowups}
		There exists a subsequence of the sequence $u_n$ which converges to a limit $u$ in the $\mathscr C^1$ norm, on any compact subset of $\mathbb{R}$.
		
		There exists $s_0 \ge0$ such that the limit $u\in\mathscr C^2(\mathbb{R}\setminus\{\pm s_0\})$. Moreover, on $\mathbb{R} \setminus [-s_0,s_0]$, $u$ is a solution of an $\alpha-$Kepler problem centred at 0, which has constant velocity on the boundary of a ball of radius $1$, for any $s \in [-s_0,s_0].$
		
		Furthermore, up to subsequences, set $\omega^2 = \lim_n \omega_n^2$. For all $s\in \mathbb{R}$, $u$ has zero energy and constant angular momentum equal to $\frac{ 2m }{\alpha \omega^2}$.
		\begin{proof}
			The first step of the proof is to write the differential equation satisfied by $u_n(s)$ on growing intervals centred at $0$. Recall that $t_n^\pm$ are defined as the extrema of the interval on which $\eta_n$ lies on the obstacle, a metric ball of radius $\ve_n$. Set 
			\[
			s_n = \ve_n^{-\frac{\alpha+2}{2}}\frac{t_n^+-t_n^-}{2};
			\]
			notice that, for $s \in [-s_n,s_n]$, $u_n(s)$ lies on the unit sphere and satisfies a different motion equation than on $[-s_n,s_n]^c$. If we put $t(s)=\ve_n^{\frac{\al+2}{2}}s+\frac{t_n^++t_n^-}{2}$, the velocity of $u_n(s)$ reads:
			\begin{equation}
				\label{eq:velocity_u_and_eta}	
				\dot{u}_n(s) = \ve_n^{\frac{\alpha}{2}}d_{\eta_n(t)} \exp_c^{-1} \dot \eta_n(t(s)) \iff \dot\eta_n(t(s))  = \ve_n^{-\frac{\alpha}{2}}d \exp_c \dot u_n(s).
			\end{equation}			
			We now compute the second derivative of $u_n(s)$ differentiating the above equation in local coordinates. Let us denote by $\Gamma(\cdot,\cdot)$ the operator $\frac{D}{dt}-\frac{d^2}{dt^2}$, i.e., the part of the covariant derivative involving the curvature. We have:
			\begin{equation}
				\label{eq:second_derivative_un}
			\begin{aligned}
				\ddot{u}_n(s) &= \ve_n^{\frac{\alpha}{2}}\frac{d}{ds} (d_{\eta_n(t(s))}\exp_c^{-1})\dot\eta_n(t(s))+\ve_n^{\alpha+1} d_{\eta_n(t(s))}\exp_c^{-1} \ddot \eta_n(t(s))  \\	&= \ve_n^{\alpha+1}\left(d^2_{\eta_n} \exp_c^{-1} (\dot\eta_n,\dot \eta_n) +d_{\eta_n} \exp_c^{-1}\ddot \eta_n\right) \\&= \ve_n \left( d^2_{\exp_c(\ve_n u_n)} \exp^{-1}_c(d \exp_c \dot u_n,d \exp_c \dot u_n)- \Gamma (d \exp_c \dot u_n,d \exp_c \dot u_n)\right) \\
				&\quad+ \ve_n^{\alpha+1}d_{\eta_n}\exp_c^{-1}\left(\frac{D\dot\eta_n}{dt}\right).
			\end{aligned}
		    \end{equation}
			Now, using point $iii)$ of Proposition \ref{prop:regularity_obstacle}, we specialise this formula to the case when $s$ is outside the interval $[-s_n,s_n]$. Recall that the potential splits as a sum of a singular and regular part, namely:
			\begin{equation}
				\label{eq:motion_sing_plus_reg}
				\frac{D\dot\eta_n}{dt} = -\frac{\nabla V(\eta_n)}{\omega_n^2} = \frac{m_j X}{\omega_n^2 d_g(c,\eta_n)^{\alpha+1}} + \nabla U(\eta_n),
			\end{equation}    
			where $X$ stands for the dual of the differential of $ d_g(c,\eta_n)$ and coincides with the direction of the unique unit-speed radial geodesic connecting $c$ and $\eta_n$.
			Now, by definition of $u_n$, we can rewrite $d_g(c,\eta_n)$ in terms of the Euclidean norm $\vert u_n\vert$, namely:
			\[
			d_g(c,\eta_n) = d_g(c,\exp_c(\ve_n u_n)) = \ve_n \vert u_n\vert.
			\]			
			Plugging \eqref{eq:motion_sing_plus_reg} into \eqref{eq:second_derivative_un} and collecting all the higher order terms in $\ve_n$, we obtain the following perturbed Kepler problem equation for $u_n$:
		    \begin{equation}
				\label{eq:ODE_blowup_outside_obstacle}
				\ddot{u}_n(s) =- \frac{m_i u_n(s)}{{\omega_n^2}\vert u_n(s) \vert^{\alpha+2}} + O(\ve_n).
			\end{equation}
			
			To get an equation for $u_n$ on $[-s_n,s_n]$, we introduce polar coordinates as in Proposition \ref{prop:regularity_obstacle}. Since $\eta_n$ in this case takes values on the boundary of the ball $B_{\ve_n}(c)$, we can express $u_n$ as $u_n(s) = e^{i\vt_n(t(s))}$. Explicit differentiation yields:
			\[
			\ddot{u}_n(s) = \ve_n^{\alpha+2} \left(\ddot{\vt}_n(t(s)) i e^{i \vt_n(t(s))}-\dot{\vt}^2_n(t(s))e^{i\vt_n(t(s))}\right).
			\]			
			Reasoning as above, we want to single out the leading order term in $\ve_n$. We start considering the second order derivative $\ddot{\vt}_n$. Setting $Y_t = d_{\ve_ne^{i\vt_n(t)}} \exp_c(ie^{i \vt_n(t)})$ and using \eqref{eq:ODE_angular_part_obstacle}, we obtain:
			\[
			\ve_n^{\alpha+2} \ddot\vt_n(t(s))=- \frac{\ve_n^{\alpha+2} }{\vert Y_{t(s)} \vert_g }\left(g\left(\frac{\nabla V (\eta_n(t(s)))}{\omega_n^2 \ve_n}+\dot\vt_n(t(s)) D_{t(s)}Y_{t(s)},\frac{Y_{t(s)}}{\vert Y_{t(s)}\vert_g}\right)\right).
			\]		
	    	We have already expanded $\nabla V(\eta_n(t))$ in \eqref{eq:motion_sing_plus_reg}. Notice that, by Gauss Lemma, $X$ and $Y_{t}$ are always orthogonal. Thus, the only contribution of $\nabla V$ that survives is the one coming from the regular part of $V$ which, however, is of order $\ve_n^{\alpha+2}$. 
			
	    	Let us consider the term $\vert Y_t \vert^{-2}_g g(D_tY_t,Y_t)$. A straightforward computation using the explicit expression of $Y_t$ shows that:
		    \[
		    \begin{aligned}
			    g\left(\frac{D_t Y_t}{\vert Y_t \vert_g},\frac{ Y_t}{\vert Y_t \vert_g}\right) &= \frac{1}{\vert Y_t\vert_g}  \frac{d}{dt}\vert Y_t\vert_g, \\ \frac{d}{dt}\vert Y_t\vert_g &= 2 \dot\vt_n(t)\left[ g\left(Y_t, \ve_n d^2_{u_n}\exp_c(ie^{i\vt_n(t)} ,ie^{i\vt_n(t)})- d_{u_n}\exp_c(e^{i\vt_n(t)})\right)\right].
		    \end{aligned} 
	        \]
		
		    Again, notice that $Y_t$ is orthogonal to $d_{u_n} \exp_c(e^{i\vt_n(t)})$ and that $\vert Y_t\vert_g \to 1$ as $\ve_n\to 0$. Plugging the above equation in the expression for $\ve_n^{\alpha+2}\ddot\vt_n$ and using \eqref{eq:angular_velocity_expansion} we find that $\ve_n^{\alpha+2}\ddot\vt_n = O(\ve_n)$.
			
	    	To get an asymptotic estimate for $\dot \vt_n$ on the obstacle we use again \eqref{eq:angular_velocity_expansion}, from which it is clear that $\dot \vt_n \sim \sqrt{\frac{2 m_j}{\alpha \omega_n^2}} \ve_n^{-\frac{\alpha+2}{2}}$. Thus, the counterpart of \eqref{eq:ODE_blowup_outside_obstacle} on the obstacle reads:
			\begin{equation}  	\label{eq:ODE_blowup_on_the_obstacle}
				\ddot{u}_n(s) = - \frac{2 m_i}{\alpha \omega_n^2} u_n(s) + O(\ve_n).
			\end{equation}     
			
			Now, we want to show the existence of a converging subsequence of $u_n$. We do this by a straightforward application of Ascoli-Arzelà theorem. We have to show that $(u_n)$ is bounded and equi-continuous. To do so, we give a uniform bound on the velocities $\dot{u}_n$ using the conservation of energy for $\eta_n$ (see $iii)$ in Proposition \ref{prop:regularity_obstacle}). From \eqref{eq:velocity_u_and_eta} and \eqref{hyp:bounds_metric}, we have:
			\[
			\frac{2(h-V(\eta_n))}{\omega_n^2} =\vert \dot \eta_n\vert_g^2 = \ve_n^{-\alpha} \vert d\exp_c \dot u_n \vert^2_g \ge \ve_n^{-\alpha} \lambda \vert \dot u_n \vert^2,
			\]
			for some $\lambda>0$. Notice that the quantity $ \ve_n^{\alpha}\frac{2(h-V(\eta_n))}{\lambda \omega_n^2}$ is uniformly bounded in $n$ since the sequence $(\eta_n)$ lives in an annulus of inner radius $\ve_n$. It follows that:
			\[
			\vert u_n(s) \vert = \begin{cases}
					1 &\text{ if } s \in [-s_n,s_n],\\
					\vert \int_{\pm s_n}^s\dot{u}_n(x)dx\vert \le C|s\mp s_n| &\text{ otherwise}
			\end{cases}
			\]			
			This settles equi-boundedness, once a compact interval is fixed. Similarly, the uniform bound on $\dot{u}_n$ implies equi-continuity of $u_n$, again by the fundamental theorem of calculus. Notice that the estimates on $\dot u _n$ hold everywhere since, from point $iv)$ of Proposition \ref{prop:regularity_obstacle}, the energy is always conserved. Thus, $(u_n)$ admits a subsequence converging uniformly on any compact interval.  
			
			Now we apply again Ascoli-Arzelà theorem to $(\dot u_n)$. We have already seen that $(\dot{u}_n)$ is bounded in the $\mathscr{C}^0$ topology. To ensure equi-continuity, it is sufficient to use  \eqref{eq:ODE_blowup_outside_obstacle} and \eqref{eq:ODE_blowup_on_the_obstacle} to get a uniform bound on $\ddot{u}_n(s).$  This is possible since the above equations express the second derivative $\ddot{u}_n$ in term of continuous functions on a bounded domain, $u_n$ and $\dot u_n$, which we already know to be uniformly bounded in $n$. 
			
			At this point, choosing intervals $[-k,k]$ for $k \in \mathbb{N}$, we can build a subsequence of $(u_n)$ converging to a function $v$ in $\mathscr{C}^1$-norm on compact intervals. Looking at the definition of $s_n$ at the beginning of this proof, up to subsequences we can assume that $s_n \to s_0\ge0$. Iterating the application of Ascoli-Arzelà given in the previous steps, we get that the limit $v$ is actually $\mathscr{C}^2$ (or $\mathscr{C}^k$, if that is the regularity of $(\eta_n)$) on $\mathbb{R}\setminus \{-s_0,s_0\}$. 
			
			Moreover, \eqref{eq:ODE_blowup_outside_obstacle} implies that $v$ is a classical solution of an $\alpha-$Kepler problem on $\mathbb{R}\setminus [-s_0,s_0]$. It is well known that $\alpha-$Kepler problems are integrable and first integrals are the total energy and the angular momentum. We now compute these quantities for the limit $v$ and prove that they are actually conserved, even on $[-s_0,s_0]$, where $v$ moves with constant angular speed on the ball of radius 1.
			
			Consider the energy first. Since $u_n$ solves a perturbed $\alpha-$Kepler problem, we can define:
			\[
			h_n(s) = \frac12 \vert \dot u_n (s) \vert^2-\frac{m_i}{\alpha \omega_n^2|u_n (s)|^{\alpha}},
			\]
			which converges to the corresponding quantity $h$ for $v$ for all $s \in \mathbb{R}\setminus [-s_0,s_0]$
			\[
			h(s) = \frac12 \vert \dot v (s) \vert^2-\frac{m_i}{\alpha \omega^2|v (s)|^{\alpha}}.
			\]
			On the other hand, using the conservation of energy for $\eta_n$, we see that $h_n(s) = O(\ve_n)$ and thus converges to $0$. We have already pointed out that, thanks to \eqref{eq:ODE_blowup_on_the_obstacle}, $u_n$ is almost a circular motion of constant angular velocity $\pm\sqrt{\frac{2 m_i}{\alpha \omega^2_n}}$ for $s \in [-s_n,s_n]$ and thus $h(s) = 0$ for all $s \in \mathbb{R}$.     
			The proof of the conservation of the angular momentum is completely analogous.
		\end{proof}
	\end{prop}

	From the proposition just proved we can deduce the following lemma. The proof is completely analogous to the one in \cite[Section 4.4]{SoaTer2012}.
	\begin{lem}[Total angular variation]
		\label{lemma:total_angular_variation_limit_blowup}
		The total angular variation of $u = \lim_n u_n$ is grater or equal than $\frac{2 \pi}{2-\alpha}$. Equality can hold if and only if $s_0=0$ and $u$ touches the boundary of the ball in just one point. In particular, if $\alpha>1$ the sequence $(\eta_n)$ is definitely non simple. 
	\end{lem}
	
	\section{Proof of Theorem  \ref{thm:main_theorem}}
	\label{sec:proof_thm}
	
	This section is devoted to the proof of Theorem \ref{thm:main_theorem}, which concerns the existence of infinitely many collision-less and periodic solutions. Actually, we will first prove a more technical result, and the proof of Theorem \ref{thm:main_theorem} will follow as a consequence. In particular, we will see that we are able to avoid collisions with the centres in many situations, except for some particular homotopy classes when we are dealing with Newtonian centres ($\al_j=1$). In this case, some peculiar collision solutions in which the particle bounces back and front between two centres may arise. We give a rigorous definition of these \emph{generalised solutions} in the following 
	\begin{defn}\label{def:collision_reflection}
		We say that $u\colon J\to M$ is a \emph{collision-reflection solution} of the motion and energy equations \eqref{eq:Newton}-\eqref{eq:energy} if:
		\begin{itemize}
			\item $u$ collides with two different centres $c_j,c_k$;
			\item $u$ solves \eqref{eq:Newton} outside collision instants;
			\item $u$ has constant energy $h$;
			\item $u$ reflects after any collision instant $\bar{t}\in J$, namely:
			\[
			u(\bar{t}-t)=u(\bar{t}+t),\quad\text{for any}\ t\in J.
			\]
		\end{itemize}
	\end{defn}
	
	\begin{thm}\label{thm:technical_thm}
		Recalling the definition of the potential $V$ given in \eqref{eq:def_potential}, we have two results:
		\begin{enumerate}[label = \roman*)]
		\item Assume that there exists at most one element $j\in\{1,\ldots,N\}$ such that $\alpha_j=1$ (so $\al_k>1$ whenever $k\neq j$) and that $[\tau]$ is an admissible class as in Definition \ref{def:admissible-class}. Then, any minimiser of the Maupertuis functional \eqref{eq:def_Maupertuis} in the space $\mathcal{H}_\Delta(\tau)$ is collision-less.	 
		
		\item Assume that there exist at least two distinct elements $j,k\in\{1,\ldots,N\}$ such that $\al_j=\al_k=1$ and $[\tau]$ is an admissible class as in Definition \ref{def:admissible-class}. Then, either there is a collision-less minimiser of the Maupertuis functional \eqref{eq:def_Maupertuis} in the space $\mathcal{H}_\Delta(\tau)$, or there exists a collision-reflection solution as in Definition \ref{def:collision_reflection}. In this latter case, the  minimiser $\gamma$ parametrises an arc joining two centres. In particular, this can happen if and only if there exists a representative $\xi$ of $[\tau]$ which is contained in a disk bounding the two Newtonian centres $c_j$ and $c_k$.
     	\end{enumerate}
	\end{thm}
	
	To prove this theorem, we will assume as in Section \ref{sec:obstacle} that a minimiser $\gamma$ of \eqref{eq:def_Maupertuis} in $\mathcal{H}_{\Delta}(\tau)$ has a collision with a centre $c$ at time $\bar{t}$. As before, take a sufficiently small metric ball $B_r(c)$ around the collision centre and let $[a,b]$ be the connected component of $\gamma^{-1}B_r(c)$ containing $\bar{t}$; call $p = \gamma(a)$ and $q = \gamma(b)$. Unless $\gamma\vert_{[a,\bar{t})}$ is a re-parametrisation of $\gamma\vert_{(\bar{t},b]}$, we can always assume that $p\ne q$. Indeed, we already know by point $iii)$ of Proposition \ref{prop:no_excess_intesections_regular} that there can be no singular $1-$gons, even having a centre on the image of $\gamma$. Assume that $\al\in[1,2)$ is the homogeneity degree of $V$ in $B_r(c)$ (see \eqref{eq:def_potential}). The separated study of the cases $\alpha>1$ and $\alpha=1$ as presented in the statement of Theorem \ref{thm:technical_thm} will be deepened later. Again, as in Section \ref{sec:obstacle}, we will suppose that the following conditions hold true:
	\begin{enumerate}[label = \emph{\roman*)}]
		\item  $d_g(c_j,c_k)>2 r $ for all $j\ne k$;
		\item  if $\gamma(t) \in B_r(c_j)$ for some $j$, then $\gamma$ has a collision with $c_j$ before leaving the ball $B_r(c_j)$.
	\end{enumerate} 
	From the discussion of the previous section (see in particular Proposition \ref{prop:convergence_blowups} and Lemma \ref{lemma:total_angular_variation_limit_blowup}), we know that there is a sequence of functions $(\eta_n)$ defined on $[a,b]$ and a sequence $\ve_n \to 0^+$ such that:
	\begin{enumerate}[label = \emph{\roman*)}]
		\item  $\eta_n(a) = p$, $\eta_n(b) = q$, $\eta_n([a,b]) \sset B_r(c)\setminus \mathring{ B}_{\ve_n}(c)$ and $\gamma\vert_{J\setminus[a,b]}\#\eta \in \mathcal{H}_{\Delta}(\tau)$;
		\item $\mathcal{M}_h(\eta_n) \to \mathcal{M}_h(\gamma\vert_{[a,b]})$;
		\item $\eta_n \wconv \eta$ in the $H^1$ topology and $\eta$ is a collision solution;
		\item $\eta_n$ are $\mathscr{C}^1$ and definitely non simple.
	\end{enumerate}	
	The last property turns out to be crucial. It is not compatible with our admissibility condition on the homotopy class $\tau$ given in Definition \ref{def:admissible-class}. The fact that $\eta_n$ is not simple suggests that the curve $\gamma\vert_{J\setminus[a,b]} \# \eta_n$ should not be taut. However, if we want to count the number of self intersections properly, we have to be sure that $\gamma\vert_{J\setminus[a,b]} \# \eta_n$ is contained in $M\setminus\mathcal{C}$ and it is in general position. To do so, it is enough to first perturb $\gamma$ in a neighbourhood of every collision centre preserving the homotopy class constraint (except for the collision at time $\bar{t}$, which has already been dealt with defining $\eta_n$) and then replace any arc which is run twice with two transversal arcs.  A detailed construction in the case $M= \mathbb{R}^2$ can be found in \cite{Cas2017}, however the argument is standard and local, so it will be omitted. A visual explanation is give in $b)$ and $c)$ of Figure \ref{fig:desingularization_proof}. We will denote by $\tilde{\gamma}_n $ a curve in general position obtained form $\gamma\vert_{J \setminus [a,b]} \# \eta_n$, applying the procedure just described.
	\begin{figure}[ht]
		\resizebox{.9\textwidth}{!}{\input{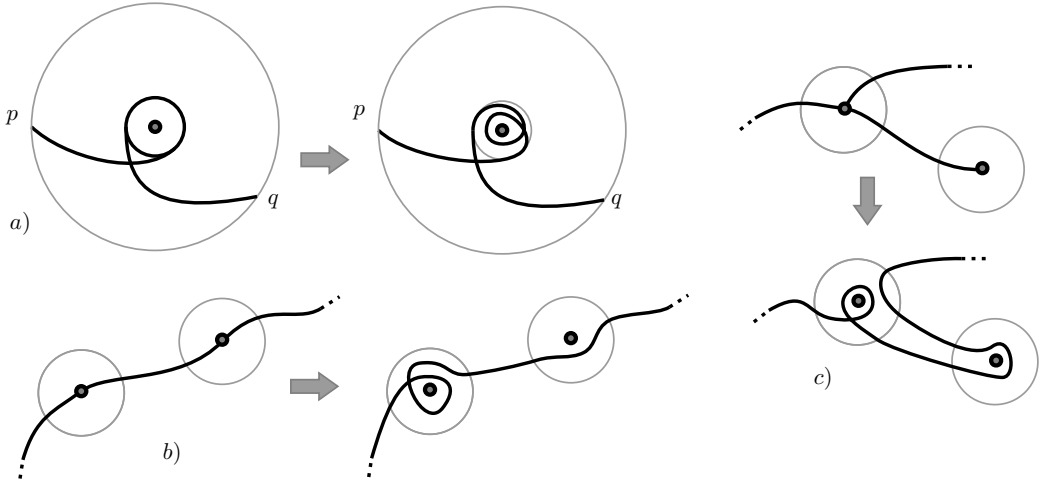}}		
		\caption{In $a)$ the arc $\zeta_n$ on the obstacle is replaced by $\tilde \zeta_n$. In $b)$ and $c)$ the desingularization process described is illustrated.}\label{fig:desingularization_proof}
	\end{figure}

	Let us introduce a little bit of notation. For any $n\in\N$, we want to split $\eta_n$ into three pieces, which we call $\eta^1_n$, $\eta^2_n$ and $\zeta_n$. We know by Proposition \ref{prop:regularity_obstacle} that the time spent on the obstacle $\eta_n^{-1}(\partial B_{\ve_n} (c)) = T_n$ is an interval, so we put $T_n=[t_n^1,t_n^2]$. We define:
	\begin{itemize}
	\item $\zeta_n$ as the portion of $\eta_n$ on the obstacle $\partial B_{\ve_n}(c)$, namely $\zeta_n\uguale\eta_n|_{T_n}=\eta_n|_{[t_n^1,t_n^2]}$;
	\item $\eta_n^1$ as the portion of $\eta_n$ before entering the obstacle, $\eta_n^1 = \eta_n \vert_{[a,t_n^1]}$;
	\item $\eta_n^2$ as the portion of $\eta_n$ after visiting the obstacle, $\eta_n^2 = \eta_n \vert_{[t_n^2,b]}$. 
    \end{itemize}
	Notice that in general $\zeta_n$ may have non transversal self-intersections, i.e., it may run over the obstacle multiple times, but this is not an issue. Indeed, we can replace $\zeta_n$ with a curve $\tilde\zeta_n$ contained in an annulus of  outer radius $\ve_n$ and inner radius smaller than $\ve_n$ as depicted in $a)$ of Figure \ref{fig:desingularization_proof}. This is made in order to ensure that $\tilde\gamma_n$ is in general position.  
	
	\begin{lem}\label{lem:self_intersections_eta}
		\label{lemma:self_intersection_eta_n}
		For $n$ sufficiently large, we have that:
		\begin{enumerate}[label = \roman*)]
			\item $\eta_n^1$ and $\eta_n^2$ are simple curves;
			\item if $\eta^1_n$ intersects $\eta_n^2$, the intersection is transversal;
			\item $\eta_n$ cannot have neither singular $1-$gon or $2-$gon;
			\item $\zeta_n$ is \emph{monotone} on the obstacle;
			\item $\tilde \gamma_n$ is not taut.
		\end{enumerate}
		\begin{proof}
			The proofs of $i)-iii)$ are straightforward adjustments of Proposition \ref{prop:no_excess_intesections_regular}. Point $iv)$ follows from the blow-up analysis provided in Proposition \ref{prop:convergence_blowups}. Point $v)$ holds since $\tilde{\gamma}_n$ belongs to an admissible homotopy class and thus there can be no innermost sub-loop enclosing just one centre.
		\end{proof}
	\end{lem}
	
	\begin{proof}[Proof of Theorem \ref{thm:technical_thm}]
		To reach a contradiction, we show that if $[\tau]$ is admissible then the sequence $\eta_n$ is made of simple curves. The proof we are going to present now essentially consists in excluding all the possible ways in which $\eta_n$ can self-intersect.
		
		Thanks to Lemma \ref{lemma:self_intersection_eta_n}, we know that self-intersections can occur in two instances alone: if $\eta_n^1$ meets $\eta_n^2$ or if $\zeta_n$ (or $\tilde\zeta_n$) self-intersects.
		
		Assume first that $\zeta_n$ has some self-intersections. We know by the blow-up analysis (see point $v)$ of Proposition \ref{prop:regularity_obstacle} and equation \eqref{eq:angular_velocity_expansion}) that the angular velocity of $\eta_n$ on the obstacle is never vanishing. Thus, we can assume that $\tilde \zeta_n$ is a path winding around $c$ several times in the same direction, having transversal intersections. Pick a point $t_n$ such that $\tilde\zeta_n (t_n)$ is a non vertex point of the innermost loop formed by $\tilde \zeta_n$. It is not hard to see that $\eta_n\vert_{[a,t_n]}$ and $\eta_n\vert_{[t_n,b]}$ have the same properties of the $\eta_n^i$ listed in Lemma \ref{lemma:self_intersection_eta_n}. Since we are now dealing only with transversal intersections, it makes sense to rename $\eta_n^1=\eta_n|_{[a,t_n]}$ and $\eta_n^2=\eta_n|_{[t_n,b]}$. In this way, the only step needed to conclude is to show that $\eta_n^1$ and $\eta_n^2$ do not intersect.
		
		Assume then that $\eta_n^1$ meets $\eta_n^2$. This can occur only in a finite number of points $p_1,\dots,p_k$, ordered increasingly with respect to the distance from the centre $c$. From Lemma \ref{lem:self_intersections_eta}, we know that $\tilde \gamma_n$ is not taut and thus, from Theorem \ref{thm:hass_scott}, it must have either a $1-$gon or a $2-$gon. We are then left to rule out these two situations:
		\begin{itemize}
		\item Assume that there is a $1-$gon. Since $1-$gons are contractible, the vertex must coincide with $p_k$. The $1-$gon can be either contained in $\left(\cup_{c_j \in \mathcal{C}}B_r(c_j)\right)^c$ or pass through some of the balls. By construction, if $\gamma$ enters any of those balls it must have a collision with the associate centre. Thus $\tilde \gamma_n$ has a singular $1-$gon if and only if $\gamma$ is as in point $i)$ of Proposition \ref{prop:no_excess_intesections_regular} or as in point $i)$ of Proposition \ref{prop:no_excess_intesections_singular}. Neither of these cases agrees with the minimality of $\gamma$ in $\mathcal{H}_\Delta(\tau)$.
		
		\item Assume that there is a $2-$gon. By point $iii)$ of Lemma \ref{lemma:self_intersection_eta_n} any such $2-$gon cannot be completely contained in the ball $B_r(c)$. Hence we have again two possibilities: either there is an edge of the $2-$gon contained in $B_r(c)$, or both edges lie partially outside the said ball. 
		
		Assume that there is one edge of the $2-$gon completely contained in the ball $B_r(c)$. One of the two vertices must be $p_k$, denote  the other one by $q_k$. Recall that, by construction, if $\tilde\gamma_n$ comes back to the ball $B_r(c)$ then $\gamma$ must have a collision with $c$. This means that the singular $2-$gon of $\tilde \gamma_n$ corresponds to a singular $1-$gon of $\gamma$ with vertex in $c$. This possibility cannot occur for minimisers of $\mathcal{M}_h$ in $\mathcal{H}_\Delta(\tau)$, as stated in point $iii)$ of Proposition \ref{prop:no_excess_intesections_regular} and point $i)$ of Proposition \ref{prop:no_excess_intesections_singular}.
		
		Assume now that both edges of the $2-$gon lie partially outside of the ball $B_r(c)$. We have again two possibilities: the vertex $q_k$ can be contained in one of the balls $B_r(c_j)$ (including $B_r(c)$) or not.
		
		Assume first that $q_k \notin B_r(c_j)$ for all $c_j \in \mathcal{C}$. Arguing as before, the $2-$gon of $\tilde \gamma_n$ must correspond to one of the configurations listed in Proposition \ref{prop:no_excess_intesections_regular} and Proposition \ref{prop:no_excess_intesections_singular}. In this case, we have either the situation in point $iv)$ of Proposition \ref{prop:no_excess_intesections_regular} if the edges do not touch any of the balls $B_r(c_j)$, or we are in the case of point $ii)$ of Proposition \ref{prop:no_excess_intesections_singular} if they do. Both cases are not possible.
		
		Assume that $q_k \in B_r(c_j)$ for some $c_j$ or $q_k\in B_r(c)$. This is the first situation in which we have to treat separately the case $\alpha =1$ and $\alpha>1$.
		
		Suppose first that $\alpha>1$. Arguing as before, we see that we are in case $iii)$ of Proposition \ref{prop:no_excess_intesections_singular}. Without loss of generality we can assume that $\gamma$ collides with $c$ and has a collision-reflection arc. Since, by Lemma \ref{lemma:total_angular_variation_limit_blowup}, the total angular variation is  strictly greater than $2\pi$, it must be at least $4\pi$. Since $[\tau]$ is an admissible class, if $\tilde \gamma$ winds at least twice around $c$ it must also wind once around it in the opposite direction, after leaving the ball $B_r(c_j)$. A visual explanation of the argument is given in Figure \ref{fig:proof_2gon_1} below.
		
		\begin{figure}[ht]
			\centering
			\resizebox{.75\textwidth}{!}{\tikzset{every picture/.style={line width=0.75pt}} 

\begin{tikzpicture}[x=0.75pt,y=0.75pt,yscale=-1,xscale=1]
	
	\draw [line width=1.5]    (77.5,97.5) .. controls (116.33,99) and (142.33,95) .. (157.5,45.5) ;
	\draw [line width=1.5]    (77.5,97.5) .. controls (85.33,45) and (78.33,29) .. (40.33,41) ;
	\draw [line width=1.5]    (77.5,97.5) .. controls (102.33,130) and (66.33,146) .. (113.33,162) ;
	\draw  [fill={rgb, 255:red, 0; green, 0; blue, 0 }  ,fill opacity=1 ] (73,97.5) .. controls (73,95.01) and (75.01,93) .. (77.5,93) .. controls (79.99,93) and (82,95.01) .. (82,97.5) .. controls (82,99.99) and (79.99,102) .. (77.5,102) .. controls (75.01,102) and (73,99.99) .. (73,97.5) -- cycle ;
	\draw  [fill={rgb, 255:red, 155; green, 155; blue, 155 }  ,fill opacity=1 ] (74.67,97.5) .. controls (74.67,95.94) and (75.94,94.67) .. (77.5,94.67) .. controls (79.06,94.67) and (80.33,95.94) .. (80.33,97.5) .. controls (80.33,99.06) and (79.06,100.33) .. (77.5,100.33) .. controls (75.94,100.33) and (74.67,99.06) .. (74.67,97.5) -- cycle ;
	\draw  [fill={rgb, 255:red, 0; green, 0; blue, 0 }  ,fill opacity=1 ] (153,45.5) .. controls (153,43.01) and (155.01,41) .. (157.5,41) .. controls (159.99,41) and (162,43.01) .. (162,45.5) .. controls (162,47.99) and (159.99,50) .. (157.5,50) .. controls (155.01,50) and (153,47.99) .. (153,45.5) -- cycle ;
	\draw  [fill={rgb, 255:red, 155; green, 155; blue, 155 }  ,fill opacity=1 ] (154.67,45.5) .. controls (154.67,43.94) and (155.94,42.67) .. (157.5,42.67) .. controls (159.06,42.67) and (160.33,43.94) .. (160.33,45.5) .. controls (160.33,47.06) and (159.06,48.33) .. (157.5,48.33) .. controls (155.94,48.33) and (154.67,47.06) .. (154.67,45.5) -- cycle ;
	\draw [line width=1.5]  [dash pattern={on 1.69pt off 2.76pt}]  (22.33,47) -- (35.23,42.7) -- (40.33,41) ;
	\draw [line width=1.5]  [dash pattern={on 1.69pt off 2.76pt}]  (113.33,162) -- (132.33,167) ;
	\draw [line width=1.5]    (321.33,113) .. controls (340.33,103) and (339.94,79.06) .. (356.33,68) .. controls (372.73,56.94) and (406.33,69.5) .. (409.33,58) .. controls (412.33,46.5) and (387.33,29.8) .. (390.33,52.8) .. controls (393.33,75.8) and (433.33,37) .. (404.33,33) .. controls (375.33,29) and (388.33,52) .. (389.33,66) .. controls (390.33,80) and (338.33,105) .. (326.33,93) ;
	\draw [line width=1.5]    (321.33,113) .. controls (292.33,123) and (331.33,161) .. (355.33,167) ;
	\draw [line width=1.5]    (282.33,46) .. controls (332.33,34) and (300.33,69) .. (326.33,93) ;
	\draw  [fill={rgb, 255:red, 0; green, 0; blue, 0 }  ,fill opacity=1 ] (315,102.5) .. controls (315,100.01) and (317.01,98) .. (319.5,98) .. controls (321.99,98) and (324,100.01) .. (324,102.5) .. controls (324,104.99) and (321.99,107) .. (319.5,107) .. controls (317.01,107) and (315,104.99) .. (315,102.5) -- cycle ;
	\draw  [fill={rgb, 255:red, 155; green, 155; blue, 155 }  ,fill opacity=1 ] (316.67,102.5) .. controls (316.67,100.94) and (317.94,99.67) .. (319.5,99.67) .. controls (321.06,99.67) and (322.33,100.94) .. (322.33,102.5) .. controls (322.33,104.06) and (321.06,105.33) .. (319.5,105.33) .. controls (317.94,105.33) and (316.67,104.06) .. (316.67,102.5) -- cycle ;
	\draw  [fill={rgb, 255:red, 0; green, 0; blue, 0 }  ,fill opacity=1 ] (395,50.5) .. controls (395,48.01) and (397.01,46) .. (399.5,46) .. controls (401.99,46) and (404,48.01) .. (404,50.5) .. controls (404,52.99) and (401.99,55) .. (399.5,55) .. controls (397.01,55) and (395,52.99) .. (395,50.5) -- cycle ;
	\draw  [fill={rgb, 255:red, 155; green, 155; blue, 155 }  ,fill opacity=1 ] (396.67,50.5) .. controls (396.67,48.94) and (397.94,47.67) .. (399.5,47.67) .. controls (401.06,47.67) and (402.33,48.94) .. (402.33,50.5) .. controls (402.33,52.06) and (401.06,53.33) .. (399.5,53.33) .. controls (397.94,53.33) and (396.67,52.06) .. (396.67,50.5) -- cycle ;
	\draw [line width=1.5]  [dash pattern={on 1.69pt off 2.76pt}]  (264.33,52) -- (277.23,47.7) -- (282.33,46) ;
	\draw [line width=1.5]  [dash pattern={on 1.69pt off 2.76pt}]  (355.33,167) -- (374.33,172) ;
	\draw  [color={rgb, 255:red, 74; green, 74; blue, 74 }  ,draw opacity=1 ][dash pattern={on 4.5pt off 4.5pt}] (294.5,102.5) .. controls (294.5,88.69) and (305.69,77.5) .. (319.5,77.5) .. controls (333.31,77.5) and (344.5,88.69) .. (344.5,102.5) .. controls (344.5,116.31) and (333.31,127.5) .. (319.5,127.5) .. controls (305.69,127.5) and (294.5,116.31) .. (294.5,102.5) -- cycle ;
	\draw  [color={rgb, 255:red, 74; green, 74; blue, 74 }  ,draw opacity=1 ][dash pattern={on 4.5pt off 4.5pt}] (132.5,45.5) .. controls (132.5,31.69) and (143.69,20.5) .. (157.5,20.5) .. controls (171.31,20.5) and (182.5,31.69) .. (182.5,45.5) .. controls (182.5,59.31) and (171.31,70.5) .. (157.5,70.5) .. controls (143.69,70.5) and (132.5,59.31) .. (132.5,45.5) -- cycle ;
	\draw  [color={rgb, 255:red, 74; green, 74; blue, 74 }  ,draw opacity=1 ][dash pattern={on 4.5pt off 4.5pt}] (52.5,97.5) .. controls (52.5,83.69) and (63.69,72.5) .. (77.5,72.5) .. controls (91.31,72.5) and (102.5,83.69) .. (102.5,97.5) .. controls (102.5,111.31) and (91.31,122.5) .. (77.5,122.5) .. controls (63.69,122.5) and (52.5,111.31) .. (52.5,97.5) -- cycle ;
	\draw  [color={rgb, 255:red, 74; green, 74; blue, 74 }  ,draw opacity=1 ][dash pattern={on 4.5pt off 4.5pt}] (374.5,50.5) .. controls (374.5,36.69) and (385.69,25.5) .. (399.5,25.5) .. controls (413.31,25.5) and (424.5,36.69) .. (424.5,50.5) .. controls (424.5,64.31) and (413.31,75.5) .. (399.5,75.5) .. controls (385.69,75.5) and (374.5,64.31) .. (374.5,50.5) -- cycle ;
	\draw  [color={rgb, 255:red, 128; green, 128; blue, 128 }  ,draw opacity=1 ][fill={rgb, 255:red, 155; green, 155; blue, 155 }  ,fill opacity=1 ] (208,101.29) -- (224.4,101.29) -- (224.4,97.38) -- (235.33,105.19) -- (224.4,113) -- (224.4,109.1) -- (208,109.1) -- cycle ;
	
	\draw (168,29.4) node [anchor=north west][inner sep=0.75pt]    {$c$};
	\draw (99,111.4) node [anchor=north west][inner sep=0.75pt]    {$c_{k}$};
	\draw (335,118.4) node [anchor=north west][inner sep=0.75pt]    {$c_{k}$};
	\draw (429,40.4) node [anchor=north west][inner sep=0.75pt]    {$c$};

\end{tikzpicture}}		
			\caption{The profile of the minimiser on the left and the desingularised version $\tilde{\gamma}_n$ on the right.}
			\label{fig:proof_2gon_1}
		\end{figure}
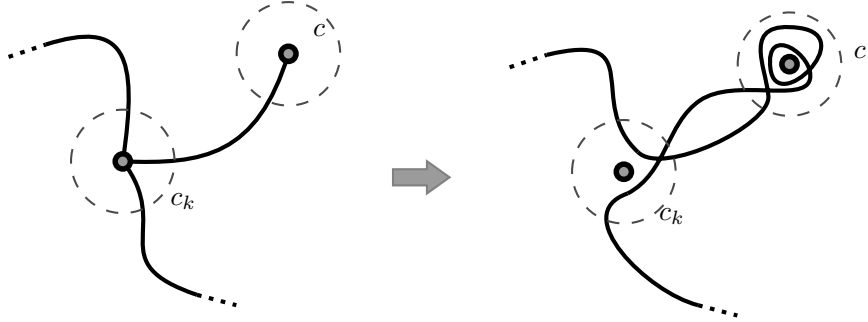
		
		Assume, without loss of generality, that $\tilde{\gamma}_n\vert_{J\cap (-\infty,a]}$ does so and suppose that this segment corresponds to a collision arc of $\gamma$, i.e., $\tilde{\gamma}_n$ crosses again $B_r(c)$. It follows that $\gamma$ has a singular $1-$gon as in point $i)$ of Proposition \ref{prop:no_excess_intesections_singular}, made by the collision arcs between $c$ and $c_j$. This is not possible if $\gamma$ is a minimiser of $\mathcal{M}_h$. See $a)$ in Figure \ref{fig:proof_2gon_2} below.
		
		Assume then that $\gamma$ does not collide with $c$. Then, it must intersect the collision-reflection arc issuing from $c$, and thus form a 2-gon as in point $iv)$ of Proposition \ref{prop:no_excess_intesections_regular}. This fact, again, is not compatible with minimality. See also $b)$ in Figure \ref{fig:proof_2gon_2} below.
		\begin{figure}[ht]
			\centering
			\resizebox{.85\textwidth}{!}{\input{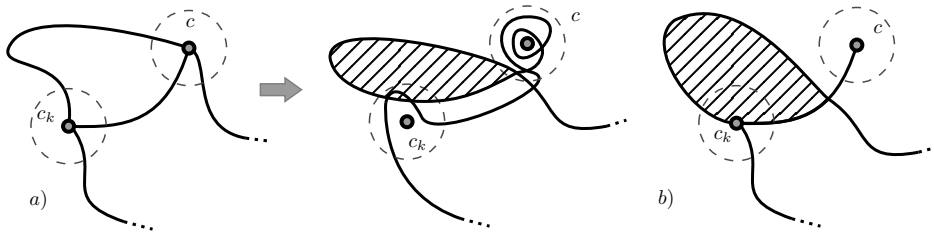}}		
			\caption{Case $a)$ illustrates the situation in which $\gamma\vert_{J\cap(-\infty,a]}$ collides again with $c$. Case $b)$ depicts the case in which $\gamma\vert_{J\cap(-\infty,a]}$ intersects the collision-reflection arc. }
			\label{fig:proof_2gon_2}
		\end{figure}
	    Thus, if $\alpha>1$, there are not intersections between $\eta_n^1$ and $\eta_n^2$, and so $\gamma$ is a collision-less minimiser. Note that, if the other centre $c_j$ is a Newtonian singularity  ($\al_j=1$), the proof does not change and thus we have completely proved the assertion in situation $i)$ of this theorem.
	    
		To prove assertion $ii)$, assume that $\alpha=1$ and there exists another centre $c_j\neq c$ such that $\al_j=1$. Then, reasoning as above, it could be that the angular variation is exactly $2\pi$. After bouncing back form $c$, it must collide with another centre $c_k$. If $\al_k>1$, we argue as before and no collisions for $\gamma$ can occur between these two centres. However, if $c_k=c_j$, we apply the same argument to the collision centre $c_j$ (since $\al_j=1$) and we can conclude that $\gamma$ must collide and be reflected back also in this case. Thus this periodic trajectory bounces between the centres $c$ and $c_j$. This fact provides an explicit restriction on the homotopy classes in which minimisers can have collisions. If $\gamma$ is in the boundary of $\mathcal{H}_\Delta(\tau)$, there should be a representative arbitrarily close to it. Since $\gamma$ is either a segment (when $c\ne c_j$) or a loop (if $c=c_j$), small neighbourhoods around it are either disks or immersed annuli. The second case, however, is not compatible with the admissibility condition of Definition \ref{def:admissible-class}. See also Figure \ref{fig:remark_admissible}.
		\end{itemize}
	\end{proof}
	
	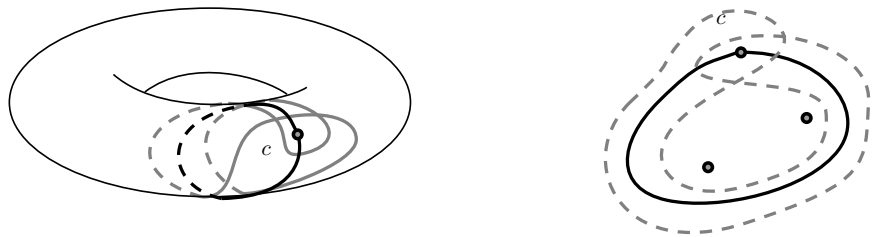
\begin{figure}[t]
	\centering
	\resizebox{.75\textwidth}{!}{\tikzset{every picture/.style={line width=0.75pt}} 

\begin{tikzpicture}[x=0.75pt,y=0.75pt,yscale=-1,xscale=1]
	
	\draw [color={rgb, 255:red, 128; green, 128; blue, 128 }  ,draw opacity=1 ][line width=1.5]    (154.33,65) .. controls (207.33,68) and (174.33,103) .. (202.33,95) .. controls (230.33,87) and (210.83,64) .. (179.33,63) ;
	\draw [color={rgb, 255:red, 128; green, 128; blue, 128 }  ,draw opacity=1 ][line width=1.5]  [dash pattern={on 5.63pt off 4.5pt}]  (163.33,120) .. controls (137.33,108) and (123.33,70) .. (179.33,63) ;
	\draw [color={rgb, 255:red, 128; green, 128; blue, 128 }  ,draw opacity=1 ][line width=1.5]  [dash pattern={on 5.63pt off 4.5pt}]  (138.33,122) .. controls (87.33,108) and (101.33,69) .. (154.33,65) ;
	\draw [color={rgb, 255:red, 128; green, 128; blue, 128 }  ,draw opacity=1 ][line width=1.5]    (138.33,122) .. controls (164.33,123) and (150.33,84) .. (178.33,76) .. controls (206.33,68) and (232.33,72) .. (232.33,87) .. controls (232.33,102) and (179.08,120.5) .. (163.33,120) ;
	\draw [line width=1.5]    (196.62,83.93) .. controls (200.49,95.49) and (197.33,122) .. (150.33,123) ;
	\draw [line width=1.5]    (166.33,66) .. controls (189.33,63) and (192.33,75) .. (196.62,83.93) ;
	\draw [line width=1.5]    (466.62,33.93) .. controls (515.33,33) and (544.33,68) .. (526.33,93) .. controls (508.33,118) and (453.33,131) .. (421.33,124) .. controls (389.33,117) and (391.33,89) .. (416.33,64) .. controls (441.33,39) and (445.33,41) .. (464.61,33.93) ;
	\draw   (21,64.41) .. controls (21,32.67) and (75.7,6.94) .. (143.17,6.94) .. controls (210.64,6.94) and (265.33,32.67) .. (265.33,64.41) .. controls (265.33,96.16) and (210.64,121.89) .. (143.17,121.89) .. controls (75.7,121.89) and (21,96.16) .. (21,64.41) -- cycle ;
	\draw    (103.55,58.03) .. controls (131.93,36.74) and (175.45,48.8) .. (190.12,59.45) ;
	\draw    (84.39,47.38) .. controls (117.74,76.48) and (188.93,64.41) .. (202.18,55.19) ;
	\draw  [fill={rgb, 255:red, 0; green, 0; blue, 0 }  ,fill opacity=1 ] (193.43,83.93) .. controls (193.43,82.16) and (194.86,80.73) .. (196.62,80.73) .. controls (198.39,80.73) and (199.82,82.16) .. (199.82,83.93) .. controls (199.82,85.69) and (198.39,87.12) .. (196.62,87.12) .. controls (194.86,87.12) and (193.43,85.69) .. (193.43,83.93) -- cycle ;
	\draw  [fill={rgb, 255:red, 155; green, 155; blue, 155 }  ,fill opacity=1 ] (194.61,83.93) .. controls (194.61,82.82) and (195.51,81.92) .. (196.62,81.92) .. controls (197.73,81.92) and (198.63,82.82) .. (198.63,83.93) .. controls (198.63,85.04) and (197.73,85.94) .. (196.62,85.94) .. controls (195.51,85.94) and (194.61,85.04) .. (194.61,83.93) -- cycle ;
	\draw  [fill={rgb, 255:red, 0; green, 0; blue, 0 }  ,fill opacity=1 ] (443.43,103.93) .. controls (443.43,102.16) and (444.86,100.73) .. (446.62,100.73) .. controls (448.39,100.73) and (449.82,102.16) .. (449.82,103.93) .. controls (449.82,105.69) and (448.39,107.12) .. (446.62,107.12) .. controls (444.86,107.12) and (443.43,105.69) .. (443.43,103.93) -- cycle ;
	\draw  [fill={rgb, 255:red, 155; green, 155; blue, 155 }  ,fill opacity=1 ] (444.61,103.93) .. controls (444.61,102.82) and (445.51,101.92) .. (446.62,101.92) .. controls (447.73,101.92) and (448.63,102.82) .. (448.63,103.93) .. controls (448.63,105.04) and (447.73,105.94) .. (446.62,105.94) .. controls (445.51,105.94) and (444.61,105.04) .. (444.61,103.93) -- cycle ;
	\draw  [fill={rgb, 255:red, 0; green, 0; blue, 0 }  ,fill opacity=1 ] (463.43,33.93) .. controls (463.43,32.16) and (464.86,30.73) .. (466.62,30.73) .. controls (468.39,30.73) and (469.82,32.16) .. (469.82,33.93) .. controls (469.82,35.69) and (468.39,37.12) .. (466.62,37.12) .. controls (464.86,37.12) and (463.43,35.69) .. (463.43,33.93) -- cycle ;
	\draw  [fill={rgb, 255:red, 155; green, 155; blue, 155 }  ,fill opacity=1 ] (464.61,33.93) .. controls (464.61,32.82) and (465.51,31.92) .. (466.62,31.92) .. controls (467.73,31.92) and (468.63,32.82) .. (468.63,33.93) .. controls (468.63,35.04) and (467.73,35.94) .. (466.62,35.94) .. controls (465.51,35.94) and (464.61,35.04) .. (464.61,33.93) -- cycle ;
	\draw  [fill={rgb, 255:red, 0; green, 0; blue, 0 }  ,fill opacity=1 ] (503.43,73.93) .. controls (503.43,72.16) and (504.86,70.73) .. (506.62,70.73) .. controls (508.39,70.73) and (509.82,72.16) .. (509.82,73.93) .. controls (509.82,75.69) and (508.39,77.12) .. (506.62,77.12) .. controls (504.86,77.12) and (503.43,75.69) .. (503.43,73.93) -- cycle ;
	\draw  [fill={rgb, 255:red, 155; green, 155; blue, 155 }  ,fill opacity=1 ] (504.61,73.93) .. controls (504.61,72.82) and (505.51,71.92) .. (506.62,71.92) .. controls (507.73,71.92) and (508.63,72.82) .. (508.63,73.93) .. controls (508.63,75.04) and (507.73,75.94) .. (506.62,75.94) .. controls (505.51,75.94) and (504.61,75.04) .. (504.61,73.93) -- cycle ;
	\draw [color={rgb, 255:red, 128; green, 128; blue, 128 }  ,draw opacity=1 ][line width=1.5]  [dash pattern={on 5.63pt off 4.5pt}]  (398.33,65) .. controls (431.33,34) and (412.33,46) .. (435.33,21) .. controls (458.33,-4) and (499.33,14) .. (492.33,30) .. controls (485.33,46) and (439.33,59) .. (423.33,82) .. controls (407.33,105) and (430.33,126) .. (467.33,116) .. controls (504.33,106) and (521.33,95) .. (518.33,70) .. controls (515.33,45) and (418.33,61) .. (443.33,36) .. controls (468.33,11) and (546.33,26) .. (544.33,64) .. controls (542.33,102) and (546.33,113) .. (469.33,138) .. controls (392.33,163) and (364.33,100) .. (398.33,65) -- cycle ;
	\draw [color={rgb, 255:red, 0; green, 0; blue, 0 }  ,draw opacity=1 ][line width=1.5]  [dash pattern={on 5.63pt off 4.5pt}]  (150.33,123) .. controls (115.33,116) and (110.33,73) .. (166.33,66) ;
	
	\draw (173,89.4) node [anchor=north west][inner sep=0.75pt]    {$c$};
	\draw (450,9.4) node [anchor=north west][inner sep=0.75pt]    {$c$};

\end{tikzpicture}}		
	\caption{Instances of collision-reflection solutions. A representative of the original homotopy class (which would not be admissible according to Definition  \ref{def:admissible-class}) is dashed in grey. }
	\label{fig:remark_admissible}
    \end{figure}

    \section{Construction of symbolic dynamics}\label{sec:symbolic_dynamics}
    
    In the previous sections we provided a huge set of periodic orbits for the $N$-centre problem on a generic surface $M$. In the configuration space, the corresponding trajectories are always collision-less, except for some precise homotopy classes when $V$ has at least two Newtonian singularities, as described in Theorem \ref{thm:technical_thm}. Such paths can be distinguished with respect to their topological properties and this fact enriches the dynamics on $T^*M$, so that we are naturally motivated to investigate the existence of a topological conjugation with a chaotic dynamical system.

    \subsection{Existence of infinitely many distinct admissible classes}
    
    The aim of this first subsection is to show that there are infinitely many distinct independent admissible homotopy classes in any of the situations listed in Theorem \ref{thm:main_theorem}. Despite the global nature of our results, the construction of such classes turns out to be purely local. In the first paragraph we will consider the case in which $M = \mathbb{R}^2$ and $N = 3$. This is, of course, completely analogous to considering a small disk bounding three centres on any surface $M$. In particular, our local approach will imply that there are infinitely many distinct admissible homotopy classes on  $\mathbb{R}^2\setminus\mathcal{C}$ with $N\ge3$ and on $\TT^2\setminus\mathcal{C}$ with $N\ge5$. Indeed, concerning the sphere, it is enough to treat $\TT^2$ as the union of two disks, one containing $3$ centres and the other the remaining $N-3\ge2$. 
     
    The second paragraph focuses on the case of genus $g>1$ and on the non compact case. The construction of the homotopy classes is completely explicit in all these cases.

    \subsubsection{$M = \mathbb{R}^2$ and $N =3$}\label{subsec:plane_sphere}
    
    Consider three centres $c_1,c_2,c_3\in\R^2$. Define $\beta_i$ as a simple loop that winds once around the centres $\{c_1,c_2,c_3\}\setminus\{c_i\}$ and $\eta$ as a simple loop winding once around all the centres, both counter-clockwise. For any non zero natural numbers $m,n \in \N$, define the following homotopy classes:
    \begin{equation}
    	\label{eq:def_classes_3_centres}
    	\omega^i_{n,m} \uguale[\beta_i^n\eta^m].
    \end{equation}
    
    \begin{figure}[ht]
    		\resizebox{.9\textwidth}{!}{\input{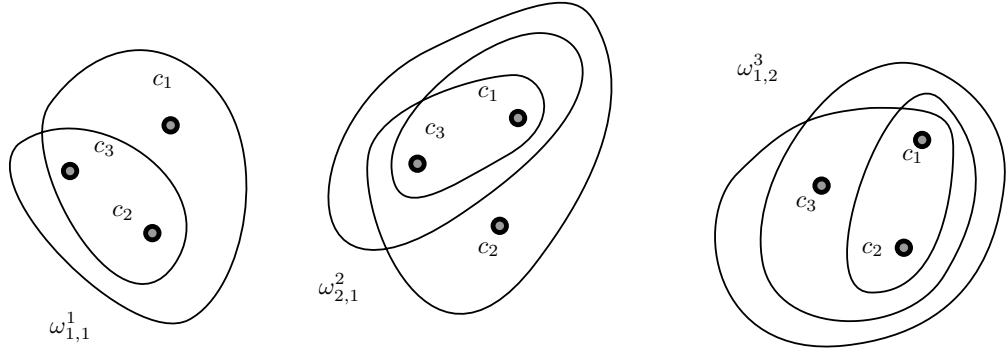}}		
    	\caption{Representatives for the homotopy classes $\omega_{1,1}^1$, $\omega_{2,1}^2$ and $\omega_{2,1}^3$ respectively.}\label{fig:letters}
    \end{figure}

    \begin{lem}
    	\label{lem:infinite_class_3centre}
    	The classes $\omega_{n,m}^i$ are admissible in the sense of Definition \ref{def:admissible-class}. Moreover, so are they products.
    	
        In particular, all the elements contained in the semigroup generated by the  $\omega_{n,m}^i$ are admissible.
    	\begin{proof}
    		The first assertion follows directly by definition of $\omega_{n,m}^i$ (see also Figure \ref{fig:letters}). Instead, the second assertion follows by induction. Take a product of  $k$ homotopy classes $\omega_{n_j,m_j}^{i_j}$, $j=1, \dots,k$. By induction hypothesis, assume that all the homotopy classes that can be written as a product of at most $k-1$ factors are admissible. Notice that, multiplying to one of this shorter words another letter $\omega^j_{l,r}$ does not create any sub-loop containing only one centre. If this was the case, also concatenating a taut representative of $\eta$ (the class of loops winding once around all three centres) with $\omega_{l,r}^j$ would. This is clearly not the case.
    	\end{proof}
    \end{lem}

    \subsubsection{Genus $g$, $g\ge1$}
    \label{subsec:genus_greater_1}
    
    Now we consider the case of genus grater than or equal to $1$. This is essentially simpler than the case stressed in the previous section. Recall that any orientable compact surface is diffeomorphic to a sphere with $g-$handles. Pick a meridian and longitude of one of the handles in such a way that they avoid all the centres $c_j$. The semigroup of the homotopy classes generated by these loops is again made up of admissible classes. 
    
    \begin{lem}
      Let $M$ be a compact surface with genus $g\ge1$ and assume $N\ge 1$. There are infinitely many admissible homotopy classes in the sense of Definition \ref{def:admissible-class}.
    \begin{proof}
    	Let us consider first the case in which $M$ is a torus with only one centre $c$. Since $M \setminus\{c\}$ retracts to a bouquet of two loops, the fundamental group $\pi_1(M\setminus\{c\})$ is isomorphic to $\mathbb{Z}*\mathbb{Z}$ . Pick two generators of $\pi_1(M\setminus \{c\})$ and consider the semigroup generated by them. All these classes are admissible since no sub-loop of taut representatives disconnects the torus (the generators do not!). Thus, there is a bijection between the set of words of two letters and elements of this semigroup, i.e., a set of admissible classes.
    	
    	Now, let $M$ be a surface of genus greater than or equal to 1 with $N\ge2$ centres. There exists $T\subset M$ homeomorphic to a torus with a disk removed such that:
    	\begin{itemize}
    		\item $T$ is a smooth submanifold with boundary of $M$;
    		\item $\{c_1,\dots,c_N\}\sset T^c$.  
    	\end{itemize}
        Clearly, now we can reason as for the torus with one centre in order to produce the desired family of admissible classes.               
    \end{proof}
    \end{lem}
    
    The same kind of reasoning applies to the case of non-compact surfaces $M$. Indeed, a similar classification theorem holds in the non-compact case as well. Loosely speaking, an orientable surface $M$ is homeomorphic to a sphere with $g$ handles glued to it and minus a certain number of points. In general, orientable non-compact surfaces are homeomorphic if and only if they have the same genus and homeomorphic ideal boundary. For further details we refer to \cite{classification_non_compact}.
    
    We can easily extend the argument of the compact case to the non-compact one. Notice that choosing a meridian and a longitude (that this time avoids the \emph{holes} as well) satisfy also the third condition of Definition \ref{def:admissible-class}. As a compact set $K$, it suffices to take a torus (with some open disks removed) which contains the two curves.

    \subsection{Construction of symbolic dynamics}
        
    We begin our study of symbolic dynamics introducing some notations. We recall that the configuration space is $\widehat{M}=M\setminus\{c_1,\ldots,c_N\}$ and that we are studying a fixed energy problem with $h$ satisfying \eqref{hyp:energy_bound}. For this reason, the periodic orbits given by Theorem \ref{thm:main_theorem} live in the energy shell
    \[
    \mathcal{E}_h\uguale\left\lbrace (q,v)\in T_q^*\widehat{M}:\,\frac12|v|_g^2+V(q)=h\right\rbrace.
    \]
    In this section we work with three centres $c_1,c_2$ and $c_3$; this is enough to succeed in our construction of symbolic dynamics. Working with more centres requires just a slight modification of our argument and would only make our notation heavier. The idea here is to associate to each periodic orbit a bi-infinite sequence of symbols, in order to investigate the presence of a topological conjugation with the prototypical chaotic system: the Bernoulli shift. For more details on this dynamical system we refer the reader to the Introduction (see \eqref{eq:def_bernoulli_shift}).
    
    The first step is to define a set of symbols and a \emph{rule} which produces distinct sequences for each distinct homotopy class at our disposal. This is achieved describing homotopy classes via intersection numbers. Note that we will often refer to some technical results contained in the forthcoming Subsection \ref{sec:technical_lemmas} (in particular, Lemma \ref{lem:construnction_segments_solutions} and Lemma \ref{lemma:approximating_minimisers}).
    
    Following \eqref{eq:def_classes_3_centres}  and recalling Lemma \ref{lem:infinite_class_3centre}, let us denote by $\beta_i$ the minimisers of \eqref{eq:def_Maupertuis} in the homotopy class enclosing the centres $\{c_1,c_2,c_3\}\setminus \{c_i\}$. By Theorem \ref{thm:technical_thm}, they always exist provided that the classes are admissible. If one of the centres is \emph{not} Newtonian, they are collision-less and enclose a region $D_i$ homeomorphic to a disk. On the other hand, they may be a collision-reflection solution if both centres are Newtonian. In the latter case, they do not bound any disk and $D_i$ stands for the support of $\beta_i$. 
    
    Now, we will assume that the homotopy class of paths enclosing the three centres $c_1,c_2$ and $c_3$ is admissible and we denote by $\gamma$ a minimiser of the Maupertuis functional \eqref{eq:def_Maupertuis} in the said class. Notice that, if for instance  $M=\mathbb{S}^2$, we would need at least 5 centres on $M$ to fulfil the admissibility condition (see also Definition \ref{def:admissible-class}).
    
    Denote the disk which has $\gamma$ as boundary by $D$.  The next result shows that the minimisers of the Maupertuis functional in the classes $\omega_{n,m}^i$ described in \eqref{eq:def_classes_3_centres} are confined in $D$.     
    
    \begin{prop}
    	\label{prop:bound_on_minimisers}
    	Assume that $\eta_1$ and $\eta_2$ are collision-less minimisers belonging to two different homotopy classes. Then, $\eta_1$ and $\eta_2$ can form no $2-$gon.
    	
    	Consider the homotopy classes $\omega_{n,m}^i$ introduced in \eqref{eq:def_classes_3_centres} and let $\gamma_{n,m}^i$ be any minimiser of the Maupertuis functional \eqref{eq:def_Maupertuis} therein. Then:
    	\begin{enumerate}[label = \roman*)]
    		\item the support of $\gamma_{n,m}^i$ is contained in $D$, for all $i,n,$ and $m$;
    		\item the support of $\gamma_{n,m}^i$ is contained in $D_i^c$, for all $i,n,$ and $m$;
    		\item any minimisers in the classes $[\gamma_{n_1,m_1}^{i_1} \dots \gamma_{n_k,m_k}^{i_k}]$, for all $i_k,n_k,m_k$ and $k \in \mathbb{N}$, is contained in $D$.
    	\end{enumerate}
    	\begin{proof}
    		The first assertion is a slight modification of Proposition \ref{prop:no_excess_intesections_regular} and follows by regularity of minimisers.    	
    		Items $i)$, $ii)$ and $iii)$ follow by a straightforward application of the above property to $\gamma_i$ and $\gamma$. 
    	\end{proof}
    \end{prop}
    
    As a consequence of this result, all the minimisers $\gamma_{n,m}^i$ in the homotopy class $\omega_{n,m}^i$ (see the definition \eqref{eq:def_classes_3_centres}) are confined to the topological disk $D$ bounded by $\gamma$. Consider now any simple curve that does not meet the boundary of $D$, passes through the centres $c_1,c_2$ and $c_3$ and encloses a disk $D'$. Draw a segment $\ell_i$ from $c_i$ to $\partial D$ in such a way that:
    \begin{itemize}
    	\item $\ell_i\cap\ell_j = \emptyset$ if $i\ne j$;
    	\item $\ell_i \sset D \setminus D'$ for all $i$.
    \end{itemize}
    A qualitative picture is given in Figure \ref{fig:half_lines}.

    \begin{figure}[t]
    	\centering
    	\resizebox{.9\textwidth}{!}{\input{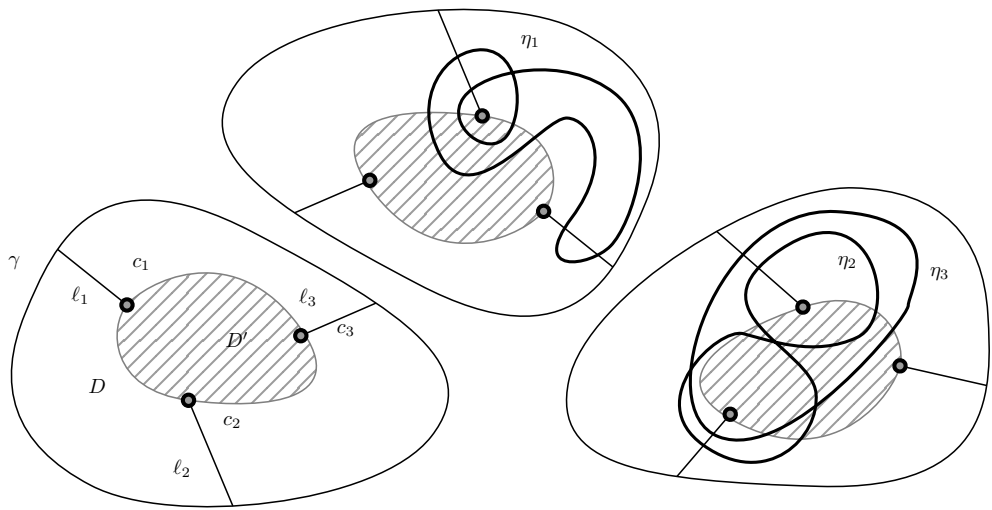}}		
    	\caption{On the left, a qualitative picture of the segments $\ell_i$. In the middle, a curve $\eta_1$ with removable intersection and, on the right, two non homotopic curves $\eta_2,\eta_3$ having the same (unsigned) itinerary.}
    	\label{fig:half_lines}
    \end{figure}
    
    \begin{rem}\label{rem:choice_ell}
    It is possible to choose the segments $\ell_i$ as reparametrisations of (collision) solutions of the motion equation \eqref{eq:Newton}. Indeed, in a forthcoming technical result (see Lemma \ref{lem:construnction_segments_solutions}), we minimise the Maupertuis functional $\mathcal{M}_h$ among paths joining one of the centres and a generic boundary point in $D\cap \ell_i$. Recalling the properties of Maupertuis minimisers proved in Section \ref{sec:variational_frame}-\ref{sec:obstacle}-\ref{sec:proof_thm}, we can do this in such a way that any of such minimisers do not meet the centres it is not emanating from. Moreover, because of the same properties, these minimisers do not intersect. Thus, without loss of generality, we can choose the segments $\ell_i$ among collision solutions of \eqref{eq:Newton} at energy $h$ too. As a consequence of this choice, we can parametrise any of the segment $\ell_i$ with respect to the time variable $t$, and for instance we can consider the curve $(\ell_i(t),\dot{\ell}_i(t))$ as a curve which lies on the energy shell $\mathcal{E}_h$. Moreover, if a solution of \eqref{eq:Newton} intersects any of the segments $\ell_i$, then it does so transversally.
    \end{rem}

    Let us introduce the set of all initial data in the energy shell $\mathcal{E}_h$ which lie in one of the segments $\ell_i$, namely
    \[
    \mathcal{T}\uguale\{(q,v)\in\mathcal{E}_h:\,q\in \ell_j,\ \text{for some}\ j\in\{1,2,3\}\}.
    \]
    Note that, since the energy $h$ is fixed, $\mathcal{T}$ is the union of $3$ open cylinders in $T\widehat{M}$. We can decompose each connected component of $\mathcal{T}$ into two  disconnected sets $C_i^\pm$  using the curves $(\ell_i(\pm t),\pm\dot\ell_i(\pm t))$. Following Remark \ref{rem:choice_ell}, the velocity $v$ of each solution of \eqref{eq:Newton} intersecting $\ell_i$ necessarily satisfies $g(v,\dot\ell_i^\perp) \ne0$ at intersection points. Thus, it always belongs to one of the two connected components $C_i^\pm$. We say that a solution \emph{crosses positively} if $v \in C_i^+$  and \emph{negatively} otherwise.

    Denote by $\Phi_t: \mathcal{E}_h\to \mathcal{E}_h$ the flow determined by the motion equation \eqref{eq:Newton}. We will denote the image of the initial data $(q,v)$ after a time $t$ by $\Phi^t(q,v)$. We can now introduce a subset of $\mathcal{T}$ which collects all the initial conditions which come back to $\mathcal{T}$ in finite time without leaving $D$, namely:
    \[
    	\Sigma\uguale\left\lbrace (q,v)\in\mathcal{T}:\,\inf\left\lbrace t>0:\,\Phi^t(q,v)\in\mathcal{T}\right\rbrace<+\infty, \,\Phi^t(q,v) \in D,\  \forall t\right\rbrace
    \]
    and, for any $(q,v)\in\Sigma$, we can define its \emph{first return time} $T(q,v)$ as the infimum in the above formula, which is actually a minimum on $\Sigma$. 
    
    At this point, we can inductively define the set 
    \[\Sigma^n \uguale \left\lbrace (q,v)\in \Sigma^{n-1}:\,\inf\left\lbrace t>0:\,\Phi^t(q,v)\in\mathcal{T}\right\rbrace<+\infty, \,\Phi^t(q,v) \in D\, \forall t\right\rbrace, 
    \] 
    where we have set $\Sigma^1\uguale \Sigma$. Let $\mathcal{J}$ denote the map $(q,v) \mapsto \mathcal{J}(q,v)=(q,-v)$ and define the collection of trajectories coming back infinitely many times, both in the future and in the past, as the following intersection:
    \[
    \Sigma^\infty \uguale \cap_{n \ge1} \left(\Sigma^n\cap \mathcal{J}\left(\Sigma^n\right)\right).
    \]    
    Note that, as a consequence of the existence of (collision-less) periodic solutions (see Theorem \ref{thm:technical_thm}), $\Sigma^{\infty}$ is non empty, since all periodic trajectories belonging to the admissible homotopy classes listed in Lemma \ref{lem:infinite_class_3centre} return to $\mathcal{T}$ infinitely many times. 
    
    Now, we define the \emph{itinerary} of a point of $\Sigma^{\infty}$. Take $(q,v) \in \Sigma^{\infty}$ and let us look at the segment $t \mapsto \Phi^t(q,v)$ with $t\in[0,T(q,v)]$. By definition, it connects two segments, say $\ell_i$ to $\ell_j$ (see Figure \ref{fig:half_lines}, picture in the middle). By construction, any trajectory with initial condition $(q,v)\in\Sigma^\infty$ is defined on the whole $\mathbb{R}$ and meets transversally the segments $\ell_i$ infinitely many times. Thus, we can associate to it an infinite sequence
    \[
    (s_k)_{\in \mathbb{Z}},\quad s_k \in \{\pm1,\pm2,\pm3\},
    \]
    as the sequence of signed intersections with the segments $\ell_i$. In particular, we set $s_k = i$ if the trajectory starting from $(q,v)$ hits $\ell_i$ after $k$ iterations and the crossing is positive. Similarly, we set $s_k = -i$ if the crossing is negative. For negative $k$, we can use a time inversion and define positive and negative intersections accordingly. We have thus defined an equivariant map: 
    \begin{equation}
    	\label{eq:def_semiconj}
    	\pi \colon \Sigma^{\infty} \to S^\mathbb{Z}, \quad\text{such that}\ \pi \circ \Phi^{T(\cdot,\cdot)}= \sigma \circ \pi.
    \end{equation} 
    Notice that, for minimisers in the classes defined in \eqref{eq:def_classes_3_centres}, not all the possible sequences are allowed. In fact, for those minimisers, the associated sequence cannot contain any piece of the form:
    \[
    \dots, -i_1,-i_2,\dots, \boxed{-i_k,i_k}, \dots i_2,i_1 \dots.
    \]
    Indeed, the presence of two consecutive intersections with the segment $\ell_{i_k}$ would correspond to the rise of a $2-$gon between the minimiser and $\ell_{i_k}$, which is impossible for the main assertion of Proposition \ref{prop:bound_on_minimisers}.    
    
    Following this rule, it is easy to guess what kind of (periodic) sequences corresponds to the minimisers $\gamma_{n,m}^i$ introduced in Proposition \ref{prop:bound_on_minimisers}:
    \[
    \begin{aligned}
    	\gamma_{n,m}^3 \Rightarrow \dots \underbrace{1,2,1,2\dots , 1,2}_{n \mathrm{\,\, times}},\underbrace{1,2,3,\dots, 1,2,3}_{m \mathrm{\,\, times}},\dots \\
    	\gamma_{n,m}^2 \Rightarrow \dots \underbrace{1,3,1,3\dots , 1,3}_{n \mathrm{\,\, times}},\underbrace{1,2,3,\dots, 1,2,3}_{m \mathrm{\,\, times}},\dots \\ 
    	\gamma_{n,m}^1 \Rightarrow \dots \underbrace{2,3,2,3\dots ,2,3}_{n \mathrm{\,\, times}},\underbrace{1,2,3,\dots, 1,2,3}_{m \mathrm{\,\, times}},\dots \\
    \end{aligned}
    \]
    Notice that, since the system is reversible, reversing time changes all the signs. As a shorthand notation, we will write $(i,j)^k$ to denote the string $i,j$ repeated $k$ times. Thus, for instance, a periodic trajectory which  minimises $\mathcal{M}_h$ in the class of point $iii)$ of Proposition \ref{prop:bound_on_minimisers} corresponds to a periodic sequence of the form:
    \begin{equation}
    	\label{eq:periodic_sequences}
    	\ldots(i_1,j_1)^{k_1}, (1,2,3)^{w_1},(i_2,j_2)^{k_2},\dots, (1,2,3)^{w_r},(i_r,j_r)^{k_r},\ldots,
    \end{equation}
    for $i_l< j_l$, $w_l,k_l>0$, and for any $l =1, \dots r$.
    
    To be precise, one should index any string of the latter form and work with the resulting sequence. Hoping it causes no confusion, we will sometimes use the word \emph{sequence} to denote the plain string too. This is equivalent to replacing an element with its orbit under the Bernoulli shift map $\sigma$ (recall that $\sigma$ acts by shifting the whole sequence on the right (see \eqref{eq:def_bernoulli_shift})). 
    
    We denote by $X$ the closure, with respect to the metric defined in \eqref{eq:def_metric_bernoulli}, of the set of elements of the form \eqref{eq:periodic_sequences}. 
    \begin{lem}
	    The set $(X,\sigma)$ is a chaotic sub-system of $(S^\Z,\sigma)$ and it is homeomorphic to a Cantor set.
	    \begin{proof}
		    Since $X$ is the closure of an invariant set and $\sigma$ is continuous, $X$ is invariant. To show that $X$ is chaotic, we have to check the three items of Definition \ref{def:chaos}. By construction, periodic trajectories are dense. To see that the system is transitive, pick any two sequences $x_1$ and $x_2$ in $X$. By construction, they are limits of periodic sequences $(x_i^n)$, i.e., $x_i^n\to x_i$ as $n\to +\infty,$ $i=1,2$. Define a new sequence $y$ as the concatenation of larger and larger portions of $x_1^n$ and $x_2^n$.  The orbit of $y$ is arbitrarily close to both $x_1$ and $x_2$. Sensitivity with respect to initial conditions can be proved by similar arguments. 
		
            Now we show that $X$ is homeomorphic to a Cantor set. We will make use of the following result contained in \cite{BrouwerOnTS}:
            \begin{center}
        	    \emph{Suppose that $X$ is a non-empty, compact, Hausdorff space without isolated points and that there is a countable base consisting of open and closed sets. Then $X$ is homeomorphic to a Cantor set.}
            \end{center}        
            By definition, $X$ is closed and, being a closed subset of the compact space $S^\mathbb{Z}$ (with respect to the \emph{product} topology), it is compact itself. Moreover, since it is metrizable, it is also Hausdorff. Pick now a finite string of symbols as in \eqref{eq:periodic_sequences}, say $f = (f_{-k},f_{-k+1},\dots,f_{l-1},f_{l})$ with $k , l \in \mathbb{N}$. If we define the sets
            \[
        	U_f \uguale\{s\in X:\, s_j = f_j,\ \forall\, j =- k,\dots,l\}
            \]
            they are closed and also open. In fact, they are open since any point in $U_f$ contains a ball with respect to the metric defined in \eqref{eq:def_metric_bernoulli}, for sufficiently small radius. On the other hand, they are closed since the evaluation maps $\mathrm{ev}_k: s \to s_k$, $k \in \Z$ are continuous. Notice that these sets are a basis for the product topology (and also for the one defined by the metric \eqref{eq:def_metric_bernoulli}). Lastly, using the same arguments, one sees that there can be no isolated points in $X$ and thus $X$ is a Cantor set.  
	    \end{proof} 
    \end{lem}
    
    Now we are finally ready to prove the semi-conjugation with the Bernoulli shift as stated in Theorem \ref{thm:symb_dyn}, which is equivalent to the following result.
    \begin{thm}\label{thm:technical_symbolic_dyn}
	    The map $\pi$ is continuous and its image contains $X$. Therefore, there exists a compact invariant subset $\Sigma'\subset \Sigma^{\infty}$ which is semi-conjugated to the chaotic system $(X,\sigma).$ 
	    \begin{proof}
		    The proof is a slight modification of the argument given in \cite[Proposition 6.1]{SoaTer2012} or in \cite[Theorem 1.7]{BarCanTer2021}. 
		
		    According to Definition \ref{def:semi_conj}, we have to show that $\pi$ is a continuous and surjective map. First of all, we prove continuity. This is essentially a reformulation of continuous dependence on initial data for regular flows. For $(q_0,v_0)\in \Sigma^\infty$, we have to show that, for every $\ve >0$, there exists $\delta>0$ such that, for every point in $B_{\delta}(q_0,v_0)$, we have
		    \[
		    d(\pi(q,v),\pi(q_0,v_0))<\ve,
		    \]
		    with $d$ as in \eqref{eq:def_metric_bernoulli}. The condition that $\pi(q,v)$ is close to $\pi(q_0,v_0)$ can be equivalently formulated in this way: for some $m_0>0$, the images $\pi(q,v)$ and $\pi(q_0,v_0)$ have the same $-m_0,\dots,m_0$ components. By continuous dependence on data, there exists a small neighbourhood of $(q_0,v_0)$ such that the image of any $(q,v)$ via the flow $\Phi^t$ stay $\ve-$close to $\Phi^t(q_0,v_0)$ on any fixed compact time interval. This means that, if we choose an interval big enough to reach the first $2m_0 +1$ intersections with the segments $\ell_i$, sufficiently close initial conditions will intersect the same $\ell_i$ in the same order (notice that we are actually close in the $\mathscr{C}^1$ topology for the configuration surface).  Thus $\pi$ is continuous.
		
		    Now we have to show now that $\pi(\Sigma^{\infty}) \supset X$. Indeed, once this is done, we can define $\Sigma' \uguale \pi^{-1}(X)$, which is a closed (with respect to the subspace topology) and invariant subset of $\Sigma^{\infty}$. To this aim, consider $x \in X$ and fix a sequence of periodic minimisers whose itinerary approximates $x$. Note that here we are going to use some results from the next section, which have been postponed being rather technical (see in particular the forthcoming Lemma \ref{lemma:approximating_minimisers}). As we will do before proving Lemma \ref{lemma:approximating_minimisers}, we can define a sequence $a_r>0$ with $r\in \mathbb{N}$, such that $a_r\to+\infty$ (cf the forthcoming  \eqref{eq:def_minimal_period}). Applying Lemma \ref{lemma:approximating_minimisers}, we can refine this sequence in such a way that it converges on compact intervals of the form $[-a_r+\ve,a_r-\ve]$ for $\ve>0$ in the $\mathscr{C}^1$ topology. The resulting limit is a minimiser defined on $[-a_r,a_r]$, collision-less on $(-a_r,a_r)$ and whose itinerary coincides with $x_{-r},\dots,x_r.$ Then, we iterate the procedure, we fix a longer piece of $x$ and we apply the previous construction to the refined sequence obtained before, obtaining a collision-less limit on $(-a_{r+1},a_{r+1})$. By point $ii)$ of Lemma \ref{lemma:approximating_minimisers}, the intervals $[-a_r,a_r]$ are increasing and arbitrarily large. Thus, a diagonal argument produces a sequence which converges $\mathscr{C}^1$ on arbitrarily large compact intervals. Denote by $\gamma$ this limit. By construction, it immediately follows that  $\pi(\gamma) = x$ and $\pi$ is surjective.
      	\end{proof}
    \end{thm}
    The next result is the most important application of Theorem \ref{thm:main_theorem} and establishes the presence of chaotic invariant subsets of the phase space of the $N$-centre problem. This result, which is stronger than Theorem \ref{thm:technical_symbolic_dyn}, requires some further hypotheses on the scalar curvature of the Jacobi-Maupertuis metric, which is defined in this way:
    \[
    g_J(v,v)\uguale (h-V(x))g(v,v),\quad v\in T_xM.
    \]
    \begin{thm} \label{thm:conjugation_neg_curvature}
	    Assume that the scalar curvature $\kappa_J$ of the Jacobi metric $g_J$ is negative. Then, the set $\Sigma'$ is conjugate to $(X,\sigma)$ and is chaotic. 
		\begin{proof}
			The core of the proof is an application of Gauss-Bonnet Theorem together with the computation of the scalar curvature $\kappa_J$ of the Jacobi metric $g_J$ of the forthcoming Lemma \ref{lemma:curvature_jacobi_metric}. This will imply a \emph{strong form} of uniqueness for geodesics contained in the disk $D$. Recall that minimisers of the Maupertuis functional \eqref{eq:def_Maupertuis} are reparametrised geodesics of $g_J$ (see for instance \cite{abraham2008foundations}). It is known that, for (sectionally) negatively curved (complete) manifolds, geodesics are unique in any homotopy class. In our case, $g_J$ is not complete; however, this is more than an existence issue. Suppose that in one of the classes appearing in Lemma \ref{lem:infinite_class_3centre} there are  two distinct geodesics $\gamma_1,\gamma_2$. Pick two points, $p$ and $q$, belonging to $\gamma_1$ and $\gamma_2$ respectively. Join them using a length minimising curve and lift the three curves to the universal cover of $D\setminus\{c_1,c_2,c_3\}$. By construction, the three curves bound a region homeomorphic to a disk $\Omega$. Moreover, the sum of the internal angles at the corner points is $2\pi$ by construction. An application of Gauss-Bonnet  theorem yields:
			\[
				0>\int_{\Omega} \kappa_J dA  = 2\pi (\chi(\Omega)-1) =0.
			\]
	   		For those minimisers which are not periodic, we can argue as follows.  We can assume, without loss of generality, that $\gamma_1$ and $\gamma_2$ have the same itinerary, intersect some $\ell_i$ at time $t=0$ and $\gamma_1(0)$ is closest to $\partial D$. Lift them to the universal cover. By construction, they bound a region homeomorphic to a strip. Denote by $\eta_t$ the geodesics issuing from $\gamma_1(t)$ with an angle of $\pi/2$. Since the curvature is negative, no $\eta_t$ can intersect $\eta_\tau$ for $t \ne \tau$ before it intersects $\gamma_2$. If that were the case, we would have a geodesic triangle with sum of the internal angles greater than $\pi$.
	   
	   		For $t<\tau$, define $\Omega_{t,\tau}$ as the disk bounded by segments of $\gamma_1$,$\gamma_2$, $\eta_t$ and $\eta_\tau$. Denote by $\alpha_\tau$ and $\beta_t$ the turning angle between $\eta_\tau$ and $\gamma_2$ and $\gamma_2$ and $\eta_t$ respectively. Applying again Gauss-Bonnet  Theorem to $\Omega_{t,\tau}$, we find:
	  		 \[
	   			\int_{\Omega_{t,\tau}} \kappa_J dA+ \pi+\alpha_t+\beta_{\tau}  = 2\pi \, \Rightarrow \,\alpha_t+\beta_\tau>\pi.
	  		\]
       		Since $\beta_t = \pi-\alpha_t$, we get that the function $t\mapsto \beta_t$ is strictly decreasing. On the other hand, since the curvature satisfy $-c^2>k_J>-C^2$ and the integral $-\int_{\Omega_{t,\tau}} k_J dA$ is bounded by $\pi$, the area of $\Omega_{t,\tau}$ is bounded for all $t,\tau \in \mathbb{R}$. In particular, the length of the segments $\eta_\tau$ must go to zero as $\tau\to +\infty$. However, this would imply that the corresponding angle $\beta_\tau$ is arbitrarily close to $\pi/2$. A contradiction.
       
       		It follows that the map $\pi: \Sigma'\to X$ is one-to-one whenever the curvature $k_J$ is negative. To prove that $\pi$ it is actually an homomorphism, we will show that the  restriction $\pi\vert_{\Sigma'}\to X$ is closed.
       		
       		Let $C\sset \Sigma'$ be a closed subset. Pick a sequence $(x_n)\sset \pi(C)$ which converges to a point $x_0\in X$. We have to show that there exists $c_0 \in C$ such that $\pi(c_0) = x_0$. By definition and thanks to the first part of the proof, if $x_n \in \pi(C),$ there exists a unique minimiser $\gamma_n$ of \eqref{eq:def_Maupertuis} with initial conditions $c_n \in C$ and having $x_n$ as itinerary. Since $x_n \to x$, we can apply again Lemma \ref{lemma:approximating_minimisers} and the argument of Theorem \ref{thm:technical_symbolic_dyn} used to prove continuity to the family $\gamma_n$. We thus obtain a subsequence which converges uniformly in the $\mathscr{C}^1$ topology on compacts intervals of $\mathbb R$. Its limit is a solution of \eqref{eq:Newton} having $x$ as itinerary. Moreover, since the convergence is uniform, the corresponding initial conditions converge to a point $c_0\in \Sigma'$. Since $C$ is closed, $c_0 \in C$ and $x \in \pi(C)$.  
		\end{proof}
	\end{thm}

    \begin{rem}
        In the next section, we are going to prove a result (see Lemma \ref{lemma:curvature_jacobi_metric}) which detects some conditions on the metric $g$ and the energy $h$ under which the hypotheses of Theorem \ref{thm:conjugation_neg_curvature} are satisfied. For instance, we will see that this is the case if $g$ is the flat metric on $\mathbb{R}^2$ or if $g$ has strictly negative curvature and the energy level $h$ is high enough. As a consequence, the statement of Theorem \ref{thm:chaotic_sym_dyn} presented in the Introduction will follows as an application of Theorem \ref{thm:conjugation_neg_curvature} above.
    \end{rem}

    \subsection{Some technical lemmas}\label{sec:technical_lemmas}
    
    As announced, this section collects and proves some technical statements that were used in the proof of Theorem \ref{thm:technical_symbolic_dyn} and \ref{thm:conjugation_neg_curvature}.

    The next lemma proves the existence of collision solutions connecting the boundary of $D$ with the centres $c_j$. These solutions are actually the segments $\ell_i$ introduced in the previous section and are fundamental to define the alphabet of our symbolic dynamics (see also Remark \ref{rem:choice_ell}). 
     \begin{lem}
    	\label{lem:construnction_segments_solutions}
    	There exist 3 points $p_1,p_2,p_3\in\partial D$ and 3 arcs $\ell_1,\ell_2,\ell_3$ inside $D$ such that, for any $i\in\{1,2,3\}$:
    	\begin{itemize}
    		\item $\ell_i$ connects $c_i$ and $p_i$;
    		\item $\ell_i\cap D=\{p_i\}$, $\ell_i\cap D'=\{c_i\}$;    		
    		\item $\ell_i$ is the support of a solution of \eqref{eq:Newton} at energy $h$;   
    		\item $\ell_i\cap \ell_j=\emptyset$, for any $i\neq j$.
    	\end{itemize}
    	\begin{proof}
    		Choose three curves joining $c_i$ and the boundary of $D$. Without loss of generality, we can choose them in such a way that they do not intersect. 
    		Denote by $p_i$ the intersection between the segment emanating from $c_i$ and $\partial D$. 
    		
    		Now, we wish to minimise the Maupertuis functional \eqref{eq:def_Maupertuis} among all the $H^1$ paths which connect $c_1$ and $p_1$ and are homotopic rel boundary in $\widehat{M}$. Following the same procedure of Section \ref{sec:variational_frame}, we start by allowing for collisions with the other two centres, so that the minimisation space becomes weakly closed and $\mathcal{M}_h$ is coercive therein. In this way, a minimiser $\gamma_1$ is provided (cf. Proposition \ref{prop:existence}).  Assume by contradiction that $\gamma_1$ collides with $c_j$ for some $j\neq 1$. Arguing exactly as in Section \ref{sec:obstacle}, we find a contradiction:
    		\begin{itemize}
    			\item if $\al_j>1$, then the obstacle minimisers would have to self-intersect forming a loop bounding $c_j$ (obtaining a similar version of Lemma \ref{lemma:total_angular_variation_limit_blowup}), but this incompatible with the homotopy class chosen;
    			\item if $\al_j=1$, then the minimiser would have to reflect after hitting $c_j$. This is again impossible since it has to connect $c_1$ and $p_1$.
    		\end{itemize} 
    	    Moreover, $\gamma_1$ collides with $c_1$ only once in the endpoint, since the Maupertuis functional is super-additive. Finally, $\gamma_1$ is contained in $D$. If this were not the case, $\gamma_1$ and $\partial D$ would form a $2-$gon contradicting the  usual regularity argument (see Propositions \ref{prop:no_excess_intesections_regular} and \ref{prop:bound_on_minimisers}).  Now, let us define $\ell_1$ as the image of $\gamma_1$ in $D$. An analogous of Proposition \ref{prop:regularity_outside_collision_set} guarantees that $\gamma_1$ solves \eqref{eq:Newton}-\eqref{eq:energy} except at the unique instant $t$ where $\gamma_1(t)=c_1$. With a similar construction, we can build $\ell_2$ and $\ell_3$, proving the first three points of the statement.
    	    
    	    To show the last property, notice that any intersection between two of them would generate a $2-$gon thanks to the homotopy constraint. This is not possible because of the usual regularity argument employed several times in the previous sections (see Propositions \ref{prop:no_excess_intesections_regular} and \ref{prop:bound_on_minimisers}).   
    	\end{proof}
    \end{lem}

    Now we prove some results needed in the proof of Theorem \ref{thm:technical_symbolic_dyn} to show that the map $\pi$ which assigns to each solution its itinerary is surjective. Consider the collision solutions arcs $\ell_i$ ($i=1,2,3$) just constructed in Lemma \ref{lem:construnction_segments_solutions}. For any solution $\gamma$ of \eqref{eq:Newton} with energy $h$ such that there exists $t_0$ for which $\gamma(t_0) \in \Sigma^{\infty}$, define the following quantity:
     \begin{equation}
     	\label{eq:def_shortest_time}
     	T_{i,j}(\gamma) = \inf_{t<s} \{s-t:\, \gamma(t) \in \ell_i \text{ and }\gamma(s) \in \ell_j\}.
     \end{equation}
    
    \begin{lem}
    	\label{lem:estimates_time}
    	Assume that $\gamma$ is a solution of \eqref{eq:Newton} and that there exists $t_0$ such that $\gamma(t_0)\in\Sigma^\infty$. Moreover, assume that, for any $a,b \in \mathbb{R}$, $\gamma|_{[a,b]}$ is a minimiser of \eqref{eq:def_Maupertuis} in the space of paths joining $\gamma(a)$ to $\gamma(b)$ homotopic to $\gamma\vert_{[a,b]}$. Then,  there exists positive constants $C_1$ and $C_2$ not depending on $\gamma$ such that 
    	\[
    	C_1\le T_{i,j}(\gamma) \le C_2.
    	\]
    	\begin{proof}
    		Consider $\eta = \gamma\vert_{[t,s]}$, where $s,t$ are such that $\gamma(t) \in \ell_i$ and  $\gamma(s) \in\ell_j$. Then, recalling that $\ell(\eta)$ is the Riemannian length of $\eta$, we easily get:
    		\[
    			\mathcal{M}_h(\eta)=\frac12\int_t^s|\dot{\gamma}|_g^2\int_s^t\left[h-V(\eta)\right]\ge \frac{h}{2} \int_t^s\vert\dot {\gamma}\vert_g^2 \ge \frac{h}{2} \left(\int_t^s \frac{\vert \dot{\gamma} \vert_g}{\sqrt{s-t}}\right)^2 \ge \frac{h \,\ell(\eta)^2}{2(s-t)}
    		\]
    	Now, since $\ell_i$ and $\ell_j$ are compact and at a bounded distance, the value of $\mathcal{M}_h$ on minimisers can be bounded below and above by positive constants $c_2\ge\mathcal{M}_h(\eta) \ge c_1$.   This  implies that 
    	\[
    	s-t\ge \frac{h \, d_g(\ell_i,\ell_j)}{2c_1}.
    	\] 
    	On the other hand, since we are working with fixed energy $h$ we have:
    	\[
    		\vert \dot \eta \vert_g^2\ge 2h \Rightarrow\int_t^s \vert\dot \eta \vert_g^2 \ge2 h (s-t). 
    	\]
        And, since $\eta$ is a segment of solution, we have: 
        \[
        \sqrt{c_2} \ge \sqrt{\mathcal{M}_h(\eta)} \ge \frac{h \,(s-t)}{2}.
        \]
    	\end{proof}    	
    \end{lem}
    
    Let us consider a bi-infinite non periodic sequence $s =(s_n)_{n \in\mathbb{Z}}$ of the form \eqref{eq:periodic_sequences}. Let $\gamma_k$ be  a sequence of periodic minimisers whose itinerary approximates the sequence $s$. For $r\in \N$, define $x_r$ to be the following truncation of $s$:
    \[
    x_r \uguale (s_{-r},s_{-r+1},\dots,s_{r-1},s_r).
    \]
    Up to a shift on the parametrisation of the elements $\gamma_k$, let us define 
    \begin{equation}
    	\label{eq:def_minimal_period}
    	a_r \uguale \lim_{k\to +\infty}\inf_{T>0}\{T: \text{the itinerary of }\gamma_k\vert_{[-T,T]} \text{ contains }x_r\}.
    \end{equation}

    \begin{lem}
    	\label{lemma:approximating_minimisers}
    The following assertions hold:
    \begin{enumerate}[label = \roman*)]
    	\item if $2T_k$ is the period of  $\gamma_k$, then $\lim_{k\to +\infty} T_k = +\infty$;
    	\item $\lim_{r\to +\infty} a_r = +\infty$; moreover, $a_r<a_{r+1}<+\infty$  for all $r$;
    	\item there exists a sequence $\ve_n\to0$ and a subsequence of the $\gamma_k$ such that, defining 
    	\[
    	\eta_n \uguale\gamma_{k_n}\vert_{[-a_r-\ve_{n_k},a_r+\ve_{n_k}]},
    	\]
    	then $\eta_n(a_r+\ve_n) \in \ell_{s_r}$ and $\eta_n(-a_r-\ve_n) \in \ell_{s_{-r}}$;
    	\item  each of the $\eta_n$ is a minimiser of the Maupertuis functional \eqref{eq:def_Maupertuis} between its endpoints;
    	\item the family $(\eta_n)$ is bounded in $H^1$ and admits a uniformly convergent subsequence;
    	\item the uniform limit $\eta$ is a minimiser and has no collision. Up to subsequence, the convergence is in $\mathscr{C}^1(I)$ for any $I\subset[-a_r,a_{r}]$.
    \end{enumerate}
    \begin{proof}
	    The proof of points $i)-iv)$ follows form a straightforward application of definitions and Lemma \ref{lem:estimates_time}.
	
	    Point $v)$ is essentially a consequence of the compactness of $\ell_{s_r}$ and $\ell_{s_{-r}}$.  Clearly, the $\eta_n$ are bounded in $L^2$ since they lie in a common compact set. To provide a bound on the velocities, using the conservation of the energy, for any $\eta_n$ we can write
     	\[\sqrt{\mathcal{M}_h(\eta_n)} = \frac12\int \dot \eta_n^2.
    	\]
    	 It follows from the definition that  $\eta_n$  joins $\ell_{-r}$ to $\ell_r$.  This implies that $\mathcal{M}_h(\eta_n)$ is uniformly bounded, for the space of initial conditions is compact and solutions having collisions at the endpoints have finite and uniformly bounded Maupertuis energy.
	
    	To prove point $vi)$ we have to show that $\eta$ is collision less. Assume that $\eta$ has a collision with a centre $c_j$ at a certain instant. We can use the sequence $(\eta_n)$ to approximate $\eta$. If $\alpha_j>1$, the blow-up argument of Section \ref{sec:obstacle} would imply that the paths $\eta_n$ are definitely tied around $c_j$. This is a contradiction, for the initial sequence $\gamma_k$ has an itinerary of the form \eqref{eq:periodic_sequences}. Instead, if  $\alpha_j=1$, this would imply that $\eta$ is a collision reflection arc. But this is again not compatible with the limit itinerary \eqref{eq:periodic_sequences}. Thus, $\eta$ is collision-less in any proper subinterval of $[-a_r,a_r]$. Now, using the motion and energy equations \eqref{eq:Newton}-\eqref{eq:energy} as in Proposition \ref{prop:convergence_blowups}, an application of Ascoli-Arzelà theorem to the derivatives $\dot \eta_n$ yields the $\mathscr{C}^1$ convergence.
    \end{proof}
    \end{lem}

    \subsection{The scalar curvature of the Jacobi-Maupertuis metric}
    
    It is well know that minimisers of \eqref{eq:def_Maupertuis} can be interpreted as geodesics for a (incomplete) Riemannian metric, the Jacobi-Maupertuis metric (see for instance \cite{abraham2008foundations}). As already observed in the Introduction, this metric is conformal to our reference metric $g$ and it is defined as follows: 
    \begin{equation}
    	\label{eq:def_Jacobi_metric}
    g_J(v,v)=(h-V(x))g(v,v), \quad v\in T_xM.
    \end{equation}
    This section is devoted to computing the curvature of the metric $g_J$. Using standard formulas for conformal changes of the metric  (see for instance \cite[Theorem  7.30]{lee_book}), one obtains that the scalar curvature of $g_J$ is
    \begin{equation}
    \kappa_J(x)=\dfrac{\kappa(x)-\Delta_g\log(h-V(x))}{h-V(x)},\quad\text{for}\ x\in M.
    \label{eq:curvature_JM}   
    \end{equation}
    On the other hand, we have:
    \[
    -\Delta_g\log(h-V(x))=\dfrac{\lvert\nabla_g V(x)\rvert^2}{(h-V(x))^2}+\dfrac{\Delta_g V(x)}{h-V(x)}.
    \]

     \begin{lem}
     	\label{lemma:curvature_jacobi_metric}
     	The curvature $\kappa_J$ is bounded and negative in the neighbourhood of every centre $c_j$. Moreover:
     	\begin{itemize}
     	\item if $\kappa<0$, then, for any fixed compact subset $\Omega$ of $M$, there exists $h(\Omega)>0$ such that $\kappa_J(x)<0$, for all $x \in \Omega$ and $h>h(\Omega)$;
     	\item if $M= \mathbb{R}^2$, $V$ is of the form \eqref{eq:def_potential} and there exist $C_1,C_2>0$ such that $-C_1^2<\kappa<-C_2^2$, then there exists $h_*\in \mathbb{R}$ such that $\kappa_J<0$, for all $h>h^*$. Alternatively, if $\kappa\equiv 0$, then $\kappa_J\le0$ for all $h>0.$
     	\end{itemize}
    \begin{proof}
 
        Let us compute first $\kappa_J(x)$ when $x$ is close to $c_j$ using the explicit expression for $V(x)$ given in \eqref{eq:potential}. As a shorthand notation denote $d_g(x,c_j)$ as $r$. A straightforward computation yields:
        \[
        \begin{aligned}
    	\vert \nabla_gV \vert^2 &= \vert \nabla_g W_j \vert^2+\frac{2 m_j g(\nabla_g r, \nabla_g W_j)}{r^{\alpha_j+1}} + \frac{m_j^2}{r^{2\alpha_j+2}} \\
    	\Delta_g V &= \Delta_g W_j - \frac{m_j(\alpha_j+1)}{r^{\alpha_j+2}}+ \frac{m_j \Delta_g r}{r^{\alpha_j+1}} 
  		\end{aligned}
    	\]
    	and, when $r\to 0^+$, we can write
    	\[
    	\frac{1}{h-V} = \frac{\alpha_j}{m_j}r^{\alpha_j}\left(1-\frac{\alpha_j(h-W_j)r^{\alpha_j}}{m_j}+\frac{\alpha_j^2(h-W_j)^2 r^{2\alpha_j}}{m_j^2}+ O(r^{3\alpha_j})\right).
        \]
        Fix geodesic normal coordinates around $c_j$ and let $S_r$ be the geodesic sphere of radius $r$. Denote by $H(x)$ the mean curvature of $S_r$ at $x$. For any smooth function $f$, the Laplacian reads:
        \[
        \Delta_g f = \dfrac{\partial^2 f}{\partial r^2 } + H\frac{\partial f}{\partial r} + \dfrac{\Delta_S f}{r^2}
        \]
        In particular, applying the formula to $f(x)=d_g(x,c_j)$, for $x \ne c_j$, one finds that:
        \[
        \Delta_g d_g(x,c_j) = H(x)
		\]
	    It is known that, on polar exponential coordinates $(r,\theta)$ centred at $c_j$, the mean curvature $H$ has the following asymptotic expansion 
		\[
		H(r, \theta) =  \frac{1}{r}+O(r),\  \text{ as } r \to 0,
		\] 
		from which we obtain that:
     	\[
     	\frac{\vert \nabla_g V\vert^2}{h-V}+\Delta_g V = - \frac{\alpha_j^2(h-W_j)}{r^2}+O(r^{-\alpha_j}),\  \text{ as } r \to 0.
        \]
   		Combining with the expression for $\kappa_J$ in \eqref{eq:curvature_JM}, we get:
   		\[
   		\kappa_J(x) = -\frac{\alpha_j^4(h-W_j(c_j))}{m_j^2} r^{2(\alpha_j-1)} + O(r^{\alpha_j}),\ \text{ as } r\to 0.
   		\]
   		This implies that, on small balls near the centres, $\kappa_J(x)$ is always negative and \emph{bounded}. 
   		
   		Moreover, expanding \eqref{eq:curvature_JM} with respect to $h$, for a certain function $F_V=F_V(h,x)$ we can write
   		\[
   		\kappa_J(x) = \frac{\kappa(x)}{h} +\frac{1}{h^2} F_V(h,x).
   		\] 
		Since we have just proved that $\kappa_J(x)$ is a bounded function on any compact subset of $M$, so is $F_V(h,x),$ uniformly with respect to $h$. Thus, for any subset $\Omega\Subset M$, if $\kappa$ is negative and $h = h(\Omega)$ is sufficiently large, $\kappa_J(x)$ is negative for all $x \in \Omega.$ 
    
		Assume now that $M = \mathbb{R}^2$ and that  $-c^2\le\kappa <-C^2$, for some constants $c,C>0$. Under these hypotheses, the distance function from a point $x$ is always smooth on $\mathbb{R}\setminus \{x\}.$ We can thus take $V$ of the form \eqref{eq:def_potential} and we have the following inequality
    	\[
    	\begin{aligned}
    	\vert\nabla_g V(x)\vert^2 &= \sum_{i,j} \dfrac{ m_i m_j g(\nabla_g d(x,c_i),\nabla_g d(x,c_j))}{d_g(x,c_i)^{\alpha_i+1}d_g(x,c_j)^{\alpha_j+1}} \\
    	 &\le \sum_{ij} \dfrac{m_i m_j}{d_g(x,c_i)^{\alpha_i+1}d_g(x,c_j)^{\alpha_j+1}} = \left(\sum_j \dfrac{m_j}{d(x,c_j)^{\alpha_j+1}}\right)^2.
    	\end{aligned}
    	\]
  	  	Now, if we define the Euclidean vectors
  	  	\[
  	  	v=\left(\dfrac{\sqrt{\al_1m_1}}{d_g(x,c_1)^{\frac{\al_1+2}{2}}},\ldots,\dfrac{\sqrt{\al_Nm_N}}{d_g(x,c_N)^{\frac{\al_N+2}{2}}}\right),\ w=\left(\dfrac{\sqrt{m_1}}{\sqrt{\al_1}d_g(x,c_1)^{\frac{\al_1}{2}}},\ldots,\dfrac{\sqrt{m_N}}{\sqrt{\al_N}d_g(x,c_N)^{\frac{\al_N}{2}}}\right),
  	  	\]
  	  	applying the Cauchy-Schwartz inequality, we obtain:
  	  	\[
   		\begin{aligned}
    	\vert \nabla_g V(x) \vert^2 &\le \left(\sum_j \dfrac{m_j}{d_g(x,c_j)^{\alpha_j+1}}\right)^2  = \langle v,w\rangle^2\le \|v\|^2\cdot\|w\|^2 \\
    	&= \left(\sum_j \dfrac{m_j}{\al_jd_g(x,c_j)^{\alpha_j}}\right)\left(\sum_j \dfrac{\al_jm_j}{d_g(x,c_j)^{\alpha_j+2}}\right) =  -V(x) \left(\sum_j \dfrac{\alpha_j \,m_j}{d_g(x,c_j)^{\alpha_j+2}}\right).
  		\end{aligned}
  		\]
  		Moreover, we can compute:
  		\[
  		\Delta_g V(x) = - \sum_j \dfrac{m_j}{\alpha_j} \Delta_g\left(\dfrac{1}{  d_g(x,c_j)^{\alpha_j}} \right)= -\sum_j \frac{ (\alpha_j+1) m_j\vert \nabla d_g(x,c_j)\vert^2}{d_g(x,c_j)^{\alpha_j+2}}+\dfrac{m_j\Delta_g d(x,c_j)}{d_g(x,c_j)^{\alpha_j+1}},
  		\]
  		with $|\nabla d_g(x.c_j)|^2=1$ for any $j$. Therefore, recalling \eqref{eq:curvature_JM}, we have obtained the following inequality:
        \[
    	\begin{aligned}
    	k_J(x) \le \dfrac{1}{{h-V(x)}}\left(\kappa(x)+\dfrac{1}{{h-V(x)}}\left( \sum_jm_j \dfrac{ d_g(x,c_j)\Delta_gd(x,c_j)-1}{d_g(x,c_j)^{\alpha+2}}\right)\right).
 	    \end{aligned}
 	    \]
	    Comparison theorems for the Laplacian of the distance (see for example\cite[Section 2]{Kasue1982ALC}) imply, for the same positive constant $C_1$ of the statement:
	     \[
	    \Delta_g d_g(x,c_j) \le C_1\, \mathrm{cotanh}( C_1 d_g(x,c_j)).
	    \]
	    This implies that the term $\dfrac{1}{{h-V(x)}}\left( \displaystyle\sum_jm_j \dfrac{ d_g(x,c_j)\Delta_gd(x,c_j)-1}{d_g(x,c_j)^{\alpha+2}}\right)$ decays at infinity and thus, for sufficiently high energies $h$, $k_J\le 0$ on the whole plane. 
	    
	    The case $k\equiv 0$ is easier, since $\Delta d_g(x,c_j) = d_g(x,c_j)^{-1}$ clearly gives $\kappa_J(x) \le0.$   
    \end{proof}
    \end{lem}

	\section{The unperturbed N-centre problem}
	\label{sec:true_kepler}
	
    If $M$ is a Riemannian manifold and $p\in M$ is an arbitrary point, it is well known that the function $q \mapsto d_g(q,p)$ may fail to be smooth, due to the presence of conjugate points or periodic geodesics (see for example \cite[Section 6.5]{book_berger}). The set of non smoothness points of $d_g(\cdot,p)$ is called \emph{cut locus} of $p$ and will be denoted by $C_p.$ As an example, consider the case of two homogeneous surfaces: the torus $\mathbb{T}^2$ and the sphere $\TT^2$, with constant curvature metrics. For these spaces, since the isometry group acts transitively, the cut locus $C_p$ is essentially the same for any point $p\in M$ and it is completely explicit. It consists of one point for the sphere and of the join of two circles for the torus.
    
    Despite this additional set of singularities, we can use Theorem \ref{thm:main_theorem} to produce $\mathscr{C}^1$ minimisers of the Maupertuis functional for the $N$-centre problem driven by potential \eqref{eq:def_potential}, which we rename here for the reader's convenience:
    \[
    	\tilde{V}(q) = -\sum_{j=1}^{N} \frac{m_i}{\al_jd_g(q,c_j)^{\alpha_j}}.
    \]
    In general, under some regularity assumptions on the cut locus (see Lemma \ref{lem:approx_distance} which conclude this section), we provide the following result:
    \begin{thm}
    	\label{thm:weak_solutions}
    	Consider the $N$-centre problem on a compact Riemannian surface $(M,g)$, driven by the potential \eqref{eq:def_potential} and with constant energy $h$ satisfying \eqref{hyp:energy_bound}. Assume that the cut locus $C_{c_j}$ admits a triangulation for every $j$ and that $[\tau]$ is an admissible homotopy class as in Definition \ref{def:admissible-class}. 
    	\begin{itemize}
    	\item If there exists at most one $j\in\{1,\ldots,N\}$ such that  $\alpha_j=1$ (i.e., $\al_k>1$ for any $k\neq j$), then any minimiser of the Maupertuis' functional \eqref{eq:def_Maupertuis} in the space $\mathcal{H}_\Delta(\tau)$ is collision-less, $\mathscr{C}^1$ and admits a reparametrisation which solves \eqref{eq:euler_lagrage_equations} weakly. 
    	\item If there exist at least two distinct elements $j,k\in\{1,\ldots,N\}$ such that $\alpha_j=\al_k=1$, the same result holds true, provided that $[\tau]$ is not in one of the homotopy classes listed in Theorem \ref{thm:technical_thm}, $ii)$.
    	\end{itemize}
    \begin{proof}
        Let us consider a centre $c_i$ and its cut locus $C_{c_i}$. The first step of the proof is to build a family $(\vp_\ve)$ of smoothings of the distance function $d_g(\cdot, c_i)$, with the following properties:
        \begin{itemize}
        \item $\text{supp}({\vp_\ve-d_g(\cdot,c_i)})\subset C_{c_i}^\ve\uguale\{q\in M:\, d(q,C_{c_i})<\ve\}$;
        \item $\vp_\ve\to d_g(\cdot,c_i)$ uniformly;
        \item $\|\nabla\vp_\ve\|_\infty\le C$ uniformly.
        \end{itemize}
        This approximating sequence is provided in Lemma \ref{lem:approx_distance} below. If we replace $d_g(\cdot,c_i)$ with $\vp_\ve$ in the definition of $\tilde{V}$ (see also \eqref{eq:def_potential}), we obtain a family of potentials $V_\ve$ which now satisfy the assumptions of Theorem \ref{thm:main_theorem}.  In particular, for any admissible class $[\tau]$ there exists a collision-less minimiser $\gamma_\ve$, for any $\ve>0$. At this point, we want to show that $(\gamma_\ve)$ is a minimising sequence in $\mathcal{H}_{\Delta}(\tau)$ for the Maupertuis functional:
        \[
        \tilde{\mathcal{M}}_h(\gamma) = \frac{1}{2} \int_J\vert \dot \gamma\vert^2_g \int_J(h-\tilde{V}(\gamma)).
        \]
        For any $\ve>0$, we also define the following \emph{approximating} Maupertuis functional on $\mathcal{H}_\Delta(\tau)$:
        \[
        \mathcal{M}_h^\ve(\gamma)\uguale\frac12\int_J\vert \dot\gamma\vert_g^2\int_J\left(h-V_\ve(\gamma)\right).
        \]
        First of all, we observe that $(\tilde{V}-V_\ve)$ is a family of functions,  each of which is compactly supported on $\bigcup_i C_{c_i}^\ve$, which converges uniformly to 0 as $\ve\to 0$.  Without loss of generality, we can also assume that $\Vert V-V_\ve \Vert_\infty\le\ve$, from which we deduce that
        \[
        	\vert \tilde{\mathcal{M}}_h(\gamma)-\mathcal{M}_h^\ve(\gamma) \vert \le \frac{1}{2} \int_J\vert \dot \gamma \vert^2_g \int_J\vert V(\gamma)-V_\ve(\gamma)\vert\le \frac{\ve}{2} \vert J \vert \int_J \vert \dot \gamma \vert_g^2,\quad\forall\gamma\in\mathcal{H}_\Delta(\tau).
        \]
        In particular, recalling that $\mathcal{M}_h^\ve(\gamma_\ve)=\min_{\mathcal{H}_\Delta(\tau)}\mathcal{M}_h^\ve$, we have that
        \begin{itemize}
        \item there exists $C>0$ such that $\mathcal{M}_h^\ve(\gamma_\ve) \le C$, for all $\ve>0$;
        \item $\vert\tilde{\mathcal{M}}_h(\gamma_\ve)-\mathcal{M}^\ve_h(\gamma_\ve)\vert \to 0,\ \text{as}\ \ve\to 0$;
        \item $\mathcal{M}_h^\ve(\gamma_\ve) \le \tilde{\mathcal{M}}_h(\gamma)+ \frac{\ve}{2} \vert J \vert \int_J \vert \dot \gamma \vert_g^2$, for any  $\gamma \in \mathcal{H}_\Delta(\tau).$
        \end{itemize}
        Thus $\lim\limits_{\ve \to 0} \tilde{\mathcal{M}}_h(\gamma_\ve)=\lim\limits_{\ve \to 0}\mathcal{M}^\ve_h(\gamma_\ve) \le \min\limits_{\mathcal{H}_\Delta(\tau)} \tilde{\mathcal{M}}_h.$ It follows that $(\gamma_\ve) $ is a minimising sequence and admits a weakly convergent subsequence $(\gamma_{\ve_n})$. Let us call $\gamma_0$ the weak (and uniform) limit.
        
        Notice that $\gamma_0$ is collision-less. In fact, we know that $\gamma_\ve$ is \emph{taut} for any $\ve$ (see Sections \ref{sec:top_framework}-\ref{sec:variational_frame}). In particular, it cannot contain a sub-loop which encloses only one centre since $[\tau]$ is admissible.  Assuming that $\gamma_0$ has a collision with a centre $c$, we see that its angular variation is no more than $2\pi$, since this is so for all the $\gamma_\ve$ thanks to the admissibility condition of the homotopy class. Thus, the only possibility here is to have collision reflection solutions, which is excluded by our assumptions (cf Section \ref{sec:proof_thm}).
        
        We have thus proved that $\gamma_\ve$ is definitely at a positive distance from the centres. We can then argue as in Proposition \ref{prop:convergence_blowups} to prove the regularity of $\gamma_0$. In fact, since the minimisers $\gamma_\ve$ have constant energy $h$ and are far from the singularities $c_i$, the conservation law
        \[
        	\frac{1}{2}\vert \dot\gamma_\ve\vert^2_g + V_\ve (\gamma_\ve) = h
        \]
         provides a uniform bound on the velocities $\dot \gamma_\ve$. Moreover, using the differential equation for $\Ddt \dot\gamma_\ve$ and the uniform bound on $\nabla \varphi_\ve$, we get equi-continuity of $\dot \gamma_\ve$  and thus $\mathscr{C}^1$ convergence.
    \end{proof}
    \end{thm}
    As announced, below we prove the approximation lemma used in the previous proof.
    \begin{lem}[Approximation of Lipschitz functions]
   	\label{lem:approx_distance}
   	Let $\vp$ be a Lipschitz function on a smooth Riemannian surface $(M,g)$ and assume that the set of points $\Sigma$ where $\vp$ is not smooth admits a triangulation. Then, there exists a family of smooth functions $(\vp_\ve)$ such that:
   	\begin{itemize}
   		\item $\vp_\ve \to \vp$ uniformly as $\ve\to 0$;
   		\item $\mathrm{supp}(\vp-\vp_\ve) \sset T_\ve(\Sigma)$, where $T_{\ve}$ stands for a tubular neighbourhood of $\Sigma$ of radius $\ve$;
   		\item there exists a constant $C>0$ such that $\Vert \nabla \vp_\ve \Vert_{\infty} \le C$. 
   	\end{itemize}
    \begin{proof}
   	    Choose a partition of unity $\{\psi_\alpha\}_{\alpha}$ supported on some open subsets $U_{\alpha}$ diffeomorphic to $\mathbb{R}^2$. Consider the functions $\vp_\al\uguale\psi_\alpha \vp$ as functions on $\R^2$. It is easy to build a smoothing kernel $\eta_\ve$ such that:
   	    \begin{itemize}
   		    \item $\mathrm{supp}(\eta_\ve)\sset B_\ve(0)$ and $\eta_\ve\ge0$;
   		    \item $\eta_\ve \in \mathscr{C}^{\infty}(\mathbb{R}^2)$ and $\int_{\R^2} \eta_\ve =1$;
   		    \item $\eta_\ve(y) = \ve^{-2} \eta(\ve^{-1} y)$, for some smooth, positive function $\eta$ of mean $1$, and such that $\eta_\ve\to \delta_0$ as distributions.
   	    \end{itemize}
        Using the $\eta_\ve$, define $\vp_{\alpha, \ve}(x)\uguale \int_{\mathbb{R}^2}\eta_\ve(y-x)\vp_\alpha(y)dy$. Since $\vp_\alpha$ is compactly supported, say in a ball of radius $R_\alpha$, $\vp_{\alpha,\ve}$ is supported in a ball of radius $R_\alpha+\ve$. Moreover, if $C_\alpha$ is the Lipschitz constant of $\vp_\alpha$, we have:
        \[
        \begin{aligned}
   	    \vert \vp_{\alpha,\ve}(x_1)-\vp_{\alpha,\ve}(x_2)\vert &\le\int_{\mathbb{R}^2} \eta_\ve(y) \vert \vp_\alpha(y+x_1)-\vp_\alpha(y+x_2)\vert dy \le C_\alpha\vert x_1-x_2\vert, \\
   	    \vert \vp_\alpha(x)-\vp_{\alpha,\ve}(x)\vert  &\le \int_{\R^2}\eta_\ve(y)\vert \vp_\alpha(y+x)-\vp_\alpha(x)\vert dy \\ &\underset{z=\ve^{-1}y}{=} \int_{\R^2} \eta(z) \vert \vp_{\alpha}(\ve z+x)-\vp_\alpha(x)\vert\,dz \le C_\alpha \ve.
        \end{aligned}
        \]
        Thus, the $\vp_{\alpha,\ve}$ uniformly converge to $\vp_\alpha$, are smooth and have the same Lipschitz constant of $\vp_\alpha$. Now, we glue back the approximating functions to a function on $M$. Set $\tilde \vp_\ve = \sum_\alpha \psi_\alpha \vp_{\alpha,\ve}.$ 

        Since we are assuming that $\Sigma$ is triangulable, the $\ve-$neighbourhood 
        \[
        \Sigma_\ve \uguale \{p \in M : d_g(p,\Sigma)\le\ve\}
        \]
        is a smooth manifold with boundary. Pick a smooth function $\psi$ which satisfies $\psi\vert_{\Sigma}\equiv 1 $ and $\psi\vert_{\Sigma_\ve^c} \equiv 0$. Consider then the functions $\vp_\ve \uguale \vp(1-\psi)+ \psi \tilde \vp_\ve$. Clearly $\vert \vp_\ve-\vp\vert = \psi \vert \vp-\tilde \vp_\ve\vert$ and thus $\vp_\ve$ satisfies the first two properties in the statement. Similarly, the uniform bound on the gradient follows from this inequality:
        \[
	    \vert \vp_\ve(p)-\vp_\ve(q)\vert \le \vert \vp(p)-\vp(q)\vert + \vert \tilde \vp_\ve(p)-\tilde \vp_\ve(q) \vert\le 2 \sum_\alpha C_\alpha \vert p-q \vert. \qedhere
        \] 
        \end{proof}
   \end{lem}
    
   \section*{Statements and Declarations}

   The authors have no relevant financial or non-financial interests to disclose. Data sharing not applicable to this article as no datasets were generated or analysed during the current study.

	\bibliography{references}
	\bibliographystyle{plain}
	
\end{document}